\documentclass[smallcondensed]{svjour3}     
\smartqed  
\usepackage{amssymb}

\usepackage{graphicx}
\usepackage{float}
\usepackage{color}
\usepackage{comment}
\usepackage{mathrsfs}

\usepackage{amsmath}

\begin{document}

\title{55- and 56-configurations are reducible
}


\author{Jin Xu         
}


\institute{J. Xu \at
              Key Laboratory of High Confidence Software Technologies (Peking University), Ministry of Education, China \\
              School of Electronics Engineering and Computer Science, Peking University, Beijing 100871, China\\
}

\date{Received: date / Accepted: date}

\maketitle

\begin{abstract}
Let $G$ be a 4-chromatic maximal planar graph (MPG)  with the minimum degree of at least 4 and let $C$ be an even-length cycle of $G$.
If $|f(C)|=2$ for every $f$ in some Kempe equivalence class of $G$, then we call $C$  an unchanged bichromatic cycle (UBC) of $G$, and correspondingly $G$ an unchanged bichromatic cycle maximal planar graph (UBCMPG) with respect to $C$, where $f(C)=\{f(v)| v\in V(C)\}$.
For an  UBCMPG $G$ with respect to an UBC $C$,  the subgraph of $G$ induced by the set of edges belonging to $C$ and its interior (or exterior), denoted by $G^C$,  is called a base-module of $G$; in particular, when the length of $C$ is equal to four, we use $C_4$ instead of $C$ and call $G^{C_4}$ a 4-base-module. In this paper, we first study the properties of UBCMPGs and   show that every 4-base-module $G^{C_4}$ contains a 4-coloring under which $C_4$ is bichromatic and there are at least two bichromatic paths with different colors between one pair of diagonal vertices of $C_4$ (these paths are called module-paths). We further prove that every 4-base-module  $G^{C_4}$  contains a 4-coloring (called decycle coloring)  for which the ends of a module-path are colored by distinct colors. Finally, based on  the technique of the contracting and extending operations of MPGs, we prove that 55-configurations and 56-configurations are reducible by converting the reducibility problem of these two classes of configurations
 into the decycle coloring  problem of 4-base-modules.
\keywords{unchanged bichromatic-cycle maximal planar graphs  \and 4-base-module \and decycle coloring \and 55-configurations and 56-configurations \and reducibility}
\end{abstract}

\section{Introduction}
\label{intro}

Maximal planar graphs (MPGs), as the ones containing the maximum number of edges in the class of planar graphs, have been well studied since the late 1970s due to the observation that MPGs can be served as the target of studying the Four-Color Conjecture (FCC). To prove the FCC, mathematicians as well as  computer scientists have been spending a tremendous amount of efforts to explore the characteristics of MPGs, including  colorings, structures and constructions, etc; see \cite{r1,r2,r3,r4,r5,r6,r7,r8,r9,r10,r11} for the detailed information on these researches. In the course of proving the FCC, many challenges are encountered inevitably, and as a result new conjectures are proposed accordingly, such as Uniquely Four Chromatic planar graphs conjecture, etc. \cite{r3,r4}, which enriches the maximal planar graph theory \cite{r11}.

It is generally known that the Euler's formula is a simple but extremely useful technique  for analyzing the structural characteristics of planar graphs. Based on this formula, a large number of consequences in terms of the structure of planar graphs are obtained, among which the fundamental and classical one is that the minimum degree $\delta(G)$ of every planar graph $G$ is at most 5, especially $3\leq \delta(G) \leq 5$ when $G$ is an MPG. There are many variations of the Euler's formula, which play an essential role in the development of ``discharging'' ---\!--- the key approach exploiting the computer-assisted proof of the FCC  by investigating  the unavoidability and  reducibility of some configurations of MPGs \cite{r5}. In the Kempe's `proof', he claimed that the $k$-wheel configuration ($k=3,4,5$)  is reducible (by introducing the method of  Kempe change) and hence proved the FCC \cite{r1}. However, it was pointed out by Heawood that there exists a hole in Kempe's proof (for the $5$-wheel) \cite{r6}, which garnered great interest to  mathematics in proving the reducibility of this type of configuration. The basic idea to cope with this hole is to (exhaustively) search for unavoidable 5-wheel-based configurations and then attempt to prove that each of these configurations is reducible. Unfortunately, the number of such configurations is so huge that the reducibility can not be examined in detail (manually) for all of these configurations.   Given this, Heesch \cite{r5} in 1969 introduced
the brilliant  method of ``discharging" at an academic conference, which can be used to design appropriate rules to explore the unavoidable sets of configurations  by the aid of computer programs.
This idea later inspired Haken and Appel, who spent seven years investigating configurations in more details and eventually in 1976 (with the help of Koch and about 1200 hour of fast mainframe computer) gave a computer-based proof of FCC  \cite{r7,r8}, where the number of discharging rules and the number of unavoidable configurations they used are 487 and 1936, respectively. Appel and Haken's work opened an avenue for logical reasoning using computers.

Many scholars have ever expressed doubts on the reliability of the computer-assisted
proof of the FCC. Indeed, Haken and Appel had modified their work several times
before the final version  was  published in 1976, in which they examined a total number of 1936 reducible
unavoidable configurations. In 1996, 20 years later, Robertson, et al. \cite{r9}, applying the same approach as that of  Apple and Haken, gave an improved proof of the FCC in which the number of reducible configurations in
the unavoidable set was reduced to 633. Nevertheless, mathematicians still expect a  conventional  simple mathematical proof for the FCC. Unfortunately, such a proof has not come out yet since 1852.

To construct MPGs effectively,  Eberhard \cite{r13} in 1891 first proposed the so-called  \textbf{pure chord-cycle}, which is a cycle $C$ of an MPG such that the interior of $C$  contains no vertices. By introducing three operators based on pure chord-cycle, denoted by $\phi_1,\phi_2,\phi_3$ (see Figure \ref{newfig1-1}), Eberhard  established an MPG generation system  $<K_4: \{\phi_1,\phi_2,\phi_3\}>$ which can  construct all MPGs starting from the initial graph $K_4$ (i.e., the complete graph of order 4).

\begin{figure}[H]
	\centering	
\includegraphics[width=12cm]{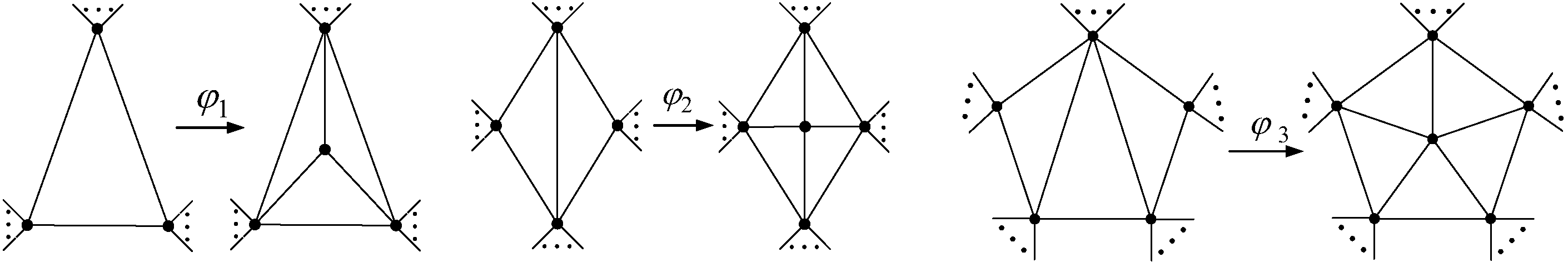}\\
	\caption{The three operators proposed by Eberhard} \label{newfig1-1}
\end{figure}

To realize the transformation between two MPGs with the same order,
Wagner \cite{r14} in 1936  proposed the method of \textbf{diagonal flip} (also called  \textbf{edge-flipping}) , which  is a local deformation of an MPG $G$  replacing  a diagonal edge $ac$ with the other edge $bd$ in a diamond subgraph $abcd$ such that  $bd\notin E(G)$, as shown in Figure \ref{newfig1-2}.

\begin{figure}[H]
	\centering	
\includegraphics[width=4.2cm]{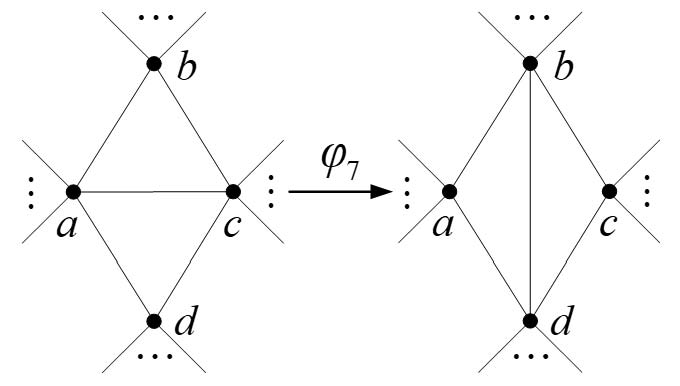}\\
	\caption{Diagonal flip} \label{newfig1-2}
\end{figure}

In 1974, Barnette \cite{r15} and Butler \cite{r16}  independently built a generation system  $<Z_{20}: \{\phi_4,\phi_5,\phi_6\}>$ (see Figure \ref{newfig1-3}) which can generate all 5-connected MPGs, where $Z_{20}$ is the icosahedron.

\begin{figure}[H]
	\centering	
\includegraphics[width=12cm]{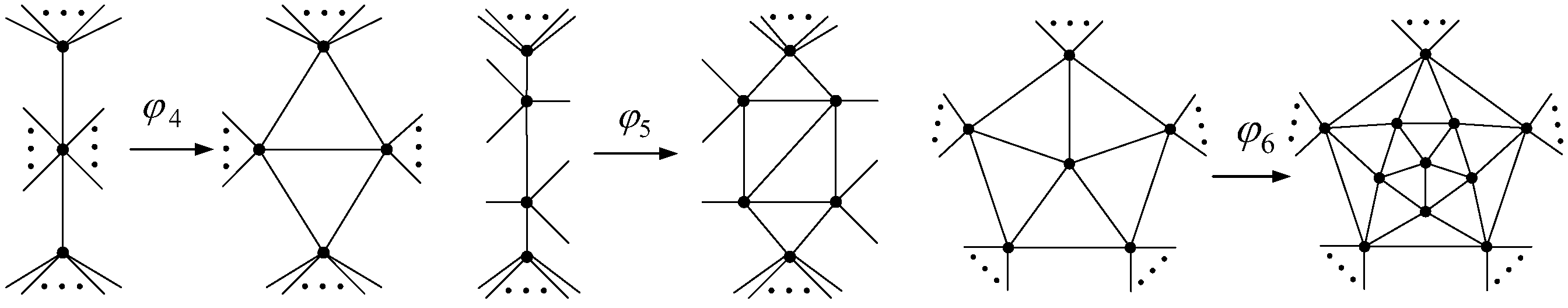}\\
	\caption{The three operators $\phi_4, \phi_5$ and $\phi_6$ proposed by Barnette and Butler} \label{newfig1-3}
\end{figure}

In 1984, Batagelj  \cite{r17} extended the generation system $<Z_{20}; \{\phi_4,\phi_5,\phi_6\}>$ to $<Z_{20}; \{\phi_4,\phi_5,\phi_7\}>$, which can  generate 3- and 4-connected MPGs of minimum degree 5. Analogously, he also depicted a method to construct (even) MPGs of minimum degree 4.

In 2016, a new method for the construction of MPGs, called the \textbf{Contracting
and Extending system (CE-system for short)}, is proposed in \cite{r18}. This system has four pairs of operators and one starting graph $K_4$. Compared with the existing approaches, CE-system can naturally connect the coloring with the construction of MPGs. We will discuss this system in detail in Section \ref{sec:ecsystem}.

As for the colorings of MPGs, a fundamental and essential work is certainly the Kempe change (or K-change for short), proposed by Kempe in 1879. The basic function of K-change is to induce a new 4-coloring from an existing 4-coloring. At present, scholars are interested in determining whether a $k (\geq 3)$-chromatic graph is a Kempe graph, i.e.,  whether a $k$-chromatic graph whose $k$-colorings can be generated from a given $k$-coloring of the graph. For more detailed information upon K-change please see the review \cite{r19}.

There are a great number of non-Kempe graphs in the class of MPGs. As an example for non-Kempe graph one can see the MPG shown in Figure \ref{newfig2-1}; this graph contains in total three 4-colorings, denoted by $f_1,f_2$ and $f_3$, in which  $f_3$ can not be generated from $f_1$ or $f_2$ by K-change. In this paper, we found a special class of non-Kempe MPGs, named as unchanged bichromatic-cycle maximal planar graphs (UBCMPGs for short). Based on UBCMPGs, we introduce  4-base-module $G^C$ and show that any 4-base-module $G^C$ of $G$ contains a decycle coloring. Moreover, with  contracting and extending operations (CE-operations) of MPGs, we are able to convert the FCC  into the decycle coloring problem of 4-base-modules and then present a mathematical proof of FCC.

\section{Preliminaries}
\label{sec:1}

All graphs considered in this paper are finite, simple, and undirected.
For a graph $G$, let $V(G)$ and $E(G)$ be the set of \emph{vertices} and the set of \emph{edges} of $G$, respectively.
If $v\in V(G)$ and $uv\in E(G)$ for some $u\in V(G)$, then $u$ is called a \emph{neighbor} of $v$. We denote by $N_G(v)$ the set of neighbors of $v$, which is called the \emph{neighborhood} of $v$ in $G$, and let $N_G[v]=N_G(v)\cup \{v\}$, which is called the \emph{closed neighborhood} of $v$ in $G$. The \emph{degree} of $v$ in $G$, denoted by $d_G(v)$, is the cardinality of $N_G(v)$, i.e., $d_G(v)=|N_G(u)|$. We use  $\delta(G)$ and $\Delta(G)$ to denote the minimum and maximum degree of $G$, respectively.  When there is no scope for ambiguity, we write  $V(G)$, $E(G)$, $d_G(v)$, $N_G(v)$, $\delta(G)$ and $\Delta(G)$ simply as $V, E, d(v), N(v), \delta$ and $\Delta$, respectively. A vertex of degree $k$ in $G$ is called a $k$-vertex of $G$. A graph $H$ is called a \emph{subgraph} of $G$ if $V(H) \subseteq V(G), E(H)\subseteq E(G)$; moreover, if for every pair of vertices $x$ and $y$ of $H$, $xy\in E(H)$ if and only if $xy\in E(G)$, then $H$ is called a \emph{subgraph induced by $V(H)$}, denoted by $G[V(H)]$. By starting with the union of two vertex-disjoint graphs $G$ and $H$, and adding edges joining every vertex of $G$ to every vertex of $H$, one obtains the \emph{join} of $G$ and $H$, denoted by $G\vee H$.  Let $K_n$ and $C_n$ denote the complete graph and cycle of order $n$, respectively.
The join $C_{n}\vee K_{1}$ is called a \emph{n-wheel} with $n$ spokes, denoted by $W_{n}$, where $C_n$ and $K_1$ are called the \emph{cycle} and \emph{center} of  $W_{n}$, respectively. Let $V(K_1)=\{x\}$ and  $C_n=x_1x_2\ldots x_nx_1$. Then, $W_n$ can be represented as  $x$-$x_1x_2\ldots x_nx_1$.  The length of a path or a cycle is the number of edges that the path or the cycle contains.  We call a path (or cycle) an $\ell$-path (or $\ell$-cycle) if its length is equal to $\ell$.

A \emph{planar graph} is a graph which can be drawn in the plane in such a way that edges meet only at their common ends, and such a drawing is called a \emph{plane graph} or \emph{planar embedding} of the graph. For any  planar graph  considered in this paper, we always refer to one of its planar embedding.
A \emph{maximal planar graph} (MPG) is a planar graph to which no new edges can be added without violating the planarity.
 A \emph{triangulation} is a planar graph in which every face is bounded by three edges (including the infinite face).
 It can be easily proved that an MPG is equivalent to a triangulation.

Let $G$ be a planar graph whose outer face has boundary $C$ containing at least four edges. If  all of its faces are triangles except the outer face, then $G$ is called a \emph{semi-maximal planar graph}  \emph{with respect to $C$} (or an \emph{SMPG} for short),  denoted by $G^C$,  where $C$ is called the \emph{outer-cycle} of $G^C$. An SMPG is also called a \emph{configuration}.

Let $G$ be an arbitrary minimum counterexample to FCC in terms of $V(G)$, i.e., $G$ is an MPG, $G$ is not 4-colorable, and every $n$-order MPG such that $n<|V(G)|$ is 4-colorable. If a configuration $G^C$ is not contained in $G$, then $G^C$ is said to be \emph{reducible}. This paper aims to prove that the two configurations, i.e., 55- and 56-configurations, shown in Figure \ref{configurations} are reducible.

\begin{figure}[H]
	\centering	
\includegraphics[width=6cm]{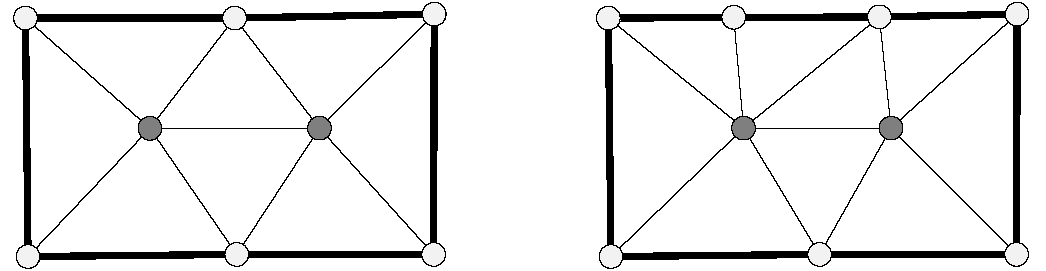}\\
	\caption{55- and  56-configurations} \label{configurations}
\end{figure}

A \emph{proper $k$-vertex-coloring}, or simply a $k$-coloring, of a
 graph \emph{G} is a mapping $f$ from  $V(G)$ to the color sets $C(k)=\{1,2,\ldots,k\}$ such
 that $f(x)\neq f(y)$ if $xy\in E(G)$.  A graph $G$ is \emph{$k$-colorable} if it has a $k$-coloring. The minimum $k$ for which a graph $G$ is $k$-colorable is called the \emph{chromatic number} of $G$, denoted by $\chi(G)$. If $\chi(G)=k$, then  $G$ is called a \emph{$k$-chromatic graph}.
Alternatively,
 each $k$-coloring $f$ of  $G$ can be viewed as a partition  $\{V_{1},V_{2},\cdots,V_{k}\}$ of $V$,
 where $V_{i}$, called the \emph{color classes} of $f$, is an independent set (every two vertices in the set are not adjacent). 
 Clearly, such partition is unique. So it can be written as $f=(V_{1},V_{2},\cdots,
 V_{k})$.
The set of all $k$-colorings of a graph
 $G$ is denoted
 by $C_{k}(G)$. We observe that any two $k$-colorings $f=(V_{1},V_{2},\cdots,
 V_{k})$ and $f'=(V'_{1},V'_{2},\cdots,
 V'_{k})$ can be viewed as the same one if there exists a set $\{j_1,j_2,\ldots, j_k\}=\{1,2,\ldots, k\}$ such that $V_i=V'_{j_i}$ for $i=1,2,\ldots, k$. We then say that $f$ and $f'$ are equivalent. All $k$-colorings equivalent to $f$ form the $f$-equivalence class. It is easy to see that $f$-equivalence class contains $k!$ colorings.
 For a $k$-chromatic graph $G$, we use $C^{0}_{k}(G)$ to
    denote the set of $k$-colorings of $G$ such that any two $k$-colorings are not in the same equivalence class and the $i$th color class of each  $k$-coloring is colored with $i$.

{\bf Remark 1} Note that we only concern 4-chromatic MPGs and SMPGs. If no specified note, we denote by $\{1,2,3,4\}$ the color set.

For a $k$-chromatic graph $G$ such that $k\geq 3$ and $g\in C_k^0(G)$, we use $g(H)$ to denote the set of colors assigned to $V(H)$ under $g$, where $H$ is a subgraph of $G$. If $|g(H)|=\ell (\ell\leq k)$, then $H$ is called an \emph{$\ell$-chromatic subgraph of $G$ under $g$}. We refer to the coloring of $g$ restricted to $V(H)$, denoted by $h$, as the \emph{restricted coloring of $g$ to $H$}, and call $g$  an \emph{extended coloring of $h$ to $G$}. If a bichromatic subgraph of $G$ under $g$ is a cycle (path), then we call the bichromatic subgraph  a \emph{bichromatic cycle} (\emph{bichromatic path}) of $g$.

Let $f$ be a 4-coloring of a 4-chromatic MPG (or  SMPG) $G$.  If $f$  does not contain any bichromatic cycle, then  $f$ is called a \emph{tree-coloring} of $G$ and $G$ is \emph{tree-colorable}; otherwise, $f$ is a \emph{cycle-coloring} and $G$ is \emph{cycle-colorable} \cite{r12,r20,r21}.

Given a 4-chromatic planar graph $G$ and a 4-coloring $f$ of $G$, we use $G_{ij}^f$ to denote the subgraphs induced by vertices colored with $i$ and $j$ under $f$, and use $\omega(G_{ij}^f)$  to denote the number of components of $G_{ij}^f$, $i\neq j$. If $f$ is fixed, we can omit the mark $f$ in $G_{ij}^f$. The components of $G_{ij}$ are called  $ij$-components of $f$; particularly, we refer to an $ij$-component as an $ij$-path of $f$ and $ij$-cycle of $f$ if it is a path and a cycle, respectively. When $\omega(G_{ij}^f)\geq 2$, the \emph{K-change} on an $ij$-component of $f$ is to interchange colors $i$ and $j$ of vertices in the $ij$-component.

 Let $G$ be a cycle-colorable MPG (or  SMPG) and $f$  a cycle-coloring of $G$. Suppose that $C$ is a bichromatic cycle of $f$ and $f(C)$=$\{i,j\}$, $i\neq j$, $i,j\in \{1,2,3,4\}$. The \emph{$\sigma$-operation of $f$ with respect to $C$}, denoted as $\sigma(f,C)$ is to interchange the colors $s$ and $t$ ($\{s,t\}=\{1,2,3,4\}\setminus \{i,j\}$) of vertices in the interior (or exterior) of $C$.

 {\bf Remark 2} Indeed, $\sigma$-operation can be viewed as one or several $K$-changes. Note that the two resulting colorings of $\sigma(f,C)$  by interchanging colors of vertices in the interior and exterior of $C$, respectively, belong to the same coloring-equivalence class. So, we always assume that colors of vertices in the interior of $C$ are interchanged when we carry out $\sigma(f,C)$.

 Clearly, $\sigma(f,C)$  transform  $f$ into a new cycle-coloring of $G$, denoted by $f^c$, that is, $\sigma(f,C)$=$f^c$. We refer to $f^c$ and $f$ as a pair of \emph{complement colorings (with respect to $C$)}. If a 4-coloring $f_0$ can be obtained from $f$ by a sequence of $\sigma$-operations, then  we say that $f$ and $f_0$ are \emph{Kempe-equivalent}. For an arbitrary 4-coloring $f\in C_4^0(G)$, we refer to
 \begin{center}
 $F^{f}(G)=\{f_0; f_0$ and $f$ are Kempe-equivalent, $f_0\in C_4^0(G)$\}
\end{center}
as \emph{the Kempe-equivalence class of $f$}.

A \emph{pseduo $k$-coloring} $f$ of a simple graph $G$ is a mapping from  $V(G)$ to $\{1,2,\ldots,k\}$ ($k\geq 3$) such that there is at least one edge (called \emph{pseudo edge}) whose endpoints receive the same color.
Observe that every edge in an MPG is incident to exactly two \emph{facial triangles} (which is a triangle whose interior is a face of the MPG.) Let
$G$ be an MPG with $\delta\geq 4$ and  $e=uv$ be a pseudo edge under a pseudo $4$-coloring $f$ of $G$. Denote by $\Delta uvu_1$ and $\Delta uvu_2$ the two facial triangles incident to $e$, where $f(u_1)=t, f(u_2)=s$, and $f(u)=f(v)=i$, $i,s,t\in \{1,2,3,4\}$ (see  Figure \ref{fig0-5} (a)). If $i\notin \{s,t\}$, then we call $e$ an \emph{pseudo $ii$-edge}, or more specifically, a \emph{$ts$-type pseudo $ii$-edge}. The pseudo edge $e$, shown in  Figure \ref{fig0-5} (b) is a $22$-type pseudo $44$-edge.

\begin{figure}[H]
	\centering	
\includegraphics[width=4cm]{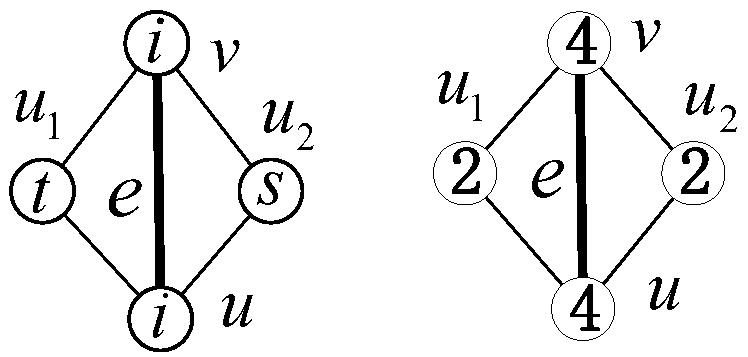}\\
(a) \hspace{1.7cm}  (b)    \\
	\caption{The structure and type of pseudo-edge in maximal planar graph} \label{fig0-5}
\end{figure}

 For terms and notations not defined here, readers can refer to \cite{r22,r10,r11,r12}.

\section{Unchanged Bichromatic-Cycle Maximal Planar Graphs}

This section introduces the unchanged bichromatic-cycle MPGs, where the definitions,  properties, and the classifications will be elaborated accordingly.

\subsection{Definitions and properties}

Let $G$ be a 4-chromatic MPG  with $\delta(G)\geq 4$, $f\in C_4^0(G)$, and $C\in C^2(f)$, where $C^2(f)$ the set of bichromatic cycles of $f$. If $|f'(C)|$=2 holds for any $f'\in F^f(G)$, then we call $C$  an \emph{unchanged bichromatic-cycle} (abbreviated by UB-cycle) of $f$, and $f$ an \emph{unchanged bichromatic-cycle coloring with respect to $C$} (abbreviated by UBC-coloring with respect to $C$) of $G$. Also, $G$ is called an \emph{unchanged bichromatic-cycle MPG  with respect to $C$} (abbreviated by \emph{UBCMPG} (or UBCSMPG) with respect to $C$). The set of bichromatic cycles of all colorings belonging to $F^f$ is called the \emph{Kempe cycle-set under $f$}, denoted by $C^2(F^f)$, i.e.

\begin{equation}
C^2(F^f)=\bigcup \limits _{f'\in F^f(G)}C^2(f')
\end{equation}

Let $C_1$ and $C_2$ be two cycles of $G$.  We say that $C_1$ and $C_2$ are \emph{intersecting} if $V(C_1)$ and $V(C_2)$ contain vertices that are in the interior of $C_2$ and $C_1$, respectively; otherwise, they are \emph{nonintersecting}.

{\bf Remark 3} Note that a UBCMPG may have more than one UB-cycles.

The graph shown in Figure \ref{newfig2-1} (a) is the minimum UBCMPG.  In Figure \ref{newfig2-1} $(a)\sim (c)$, we exhibit all the 4-colorings $f_1,f_2,f_3$ of the graph,  where the cycle labeled with bold lines in  Figure \ref{newfig2-1} $(a)$ and $(b)$ are UB-cycles. Clearly, $F^{f_1}$=$\{f_1,f_2\}$ and $F^{f_3}$=$\{f_3\}$.

\begin{figure}[H]
  \centering
  \includegraphics[width=10cm]{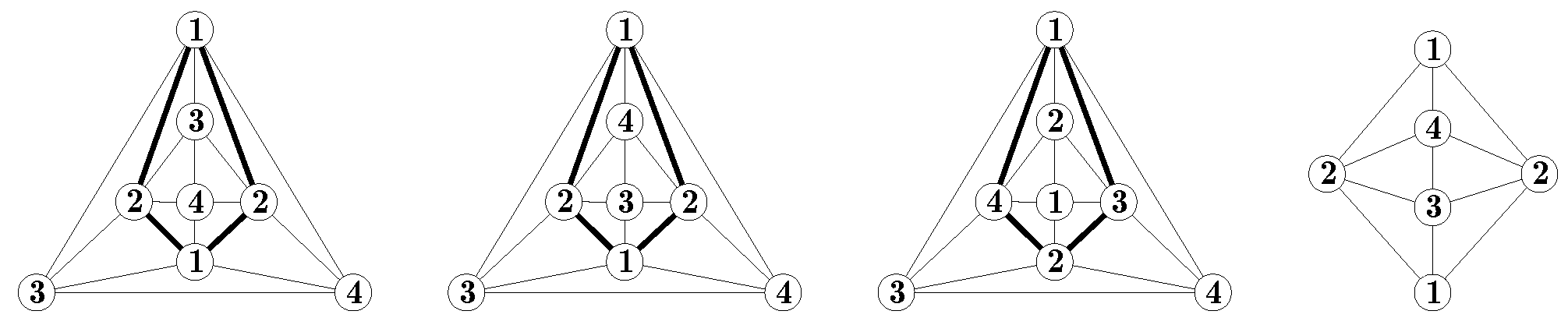}\\
  \hspace{8mm}(a) $f_1$ \hspace{1.8cm}  (b) $f_2$ \hspace{1.8cm}  (c) $f_3$ \hspace{1.4cm} (d) $B^4$
  \caption{The UBCMPG with the minimum order}\label{newfig2-1}
\end{figure}

\begin{theorem}\label{thm2.1}
Let $G$ be a 4-chromatic MPG with $\delta(G)\geq 4$, $f\in C_4^0(G)$, and $C\in C^2(f)$. Then, $C$ is a UB-cycle of $f$ if and only if for any $C'\in C^2(F^f)$ with $f(C')\neq f(C)$, $C$ and $C'$ are nonintersecting.
\end{theorem}
\begin{proof}
(Necessity). Suppose, to the contrary, that there is a cycle $C'\in C^2(F^f)$ with $f(C')\neq f(C)$ such that  $C$ and $C'$ are intersecting. Then, $C'$ is a bichromatic cycle of some coloring in  $F^f(G)$, say $f'$, i.e., $C'\in C^2(f')$. Since $C$ is a UB-cycle of $f$, it follows that $C\in C^2(f')$. Let $f''=\sigma(f',C')$. Since  $C$ and $C'$ are intersecting, we have that $C\notin C^2(f'')$,  a contradiction.

(Sufficiency). Observe that every $f'\in F^f(G)$ is obtained by implementing $\sigma$-operation of a 4-coloring $g\in F^f(G)$ with respect to a bichromatic cycle of $g$. We have that $C\in C^2(f')$ for every $f'\in F^f(G)$, since $C$ and $C'$ are nonintersecting for every $C'\in C^2(F^f)$. \qed
\end{proof}

Theorem \ref{thm2.1} is called the \emph{basic theorem} of UB-cycle, based on which we can get more methods to identify a UB-cycle. For example, the following result holds directly from Theorem \ref{thm2.1}.

\begin{corollary}\label{coro2.2}
Let $G$ be a UBCMPG with respect to $C$. Then, every vertex on $C$ has degree at least 5 in $G$.
\end{corollary}

\subsection{Types}

 \begin{figure}[H]
  \centering
  \includegraphics[width=12.5cm]{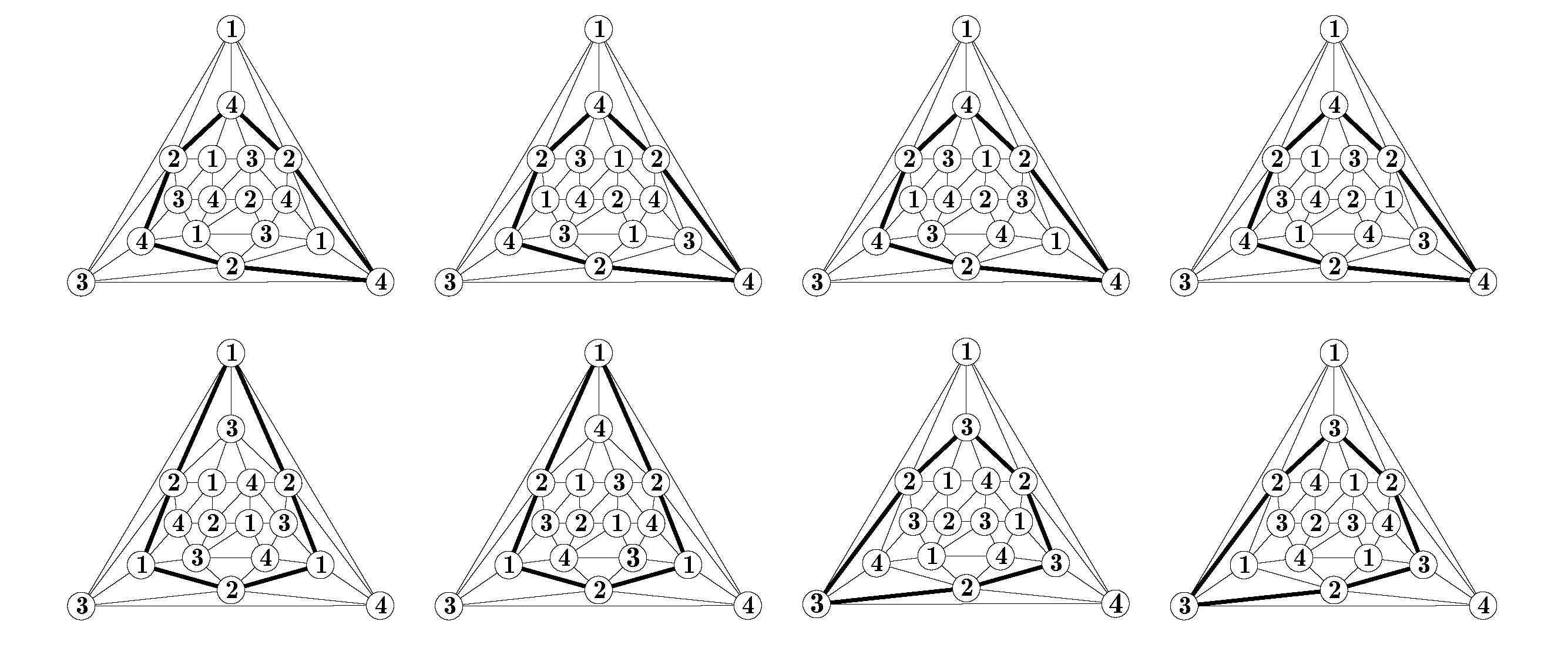}\\
  \includegraphics[width=12.5cm]{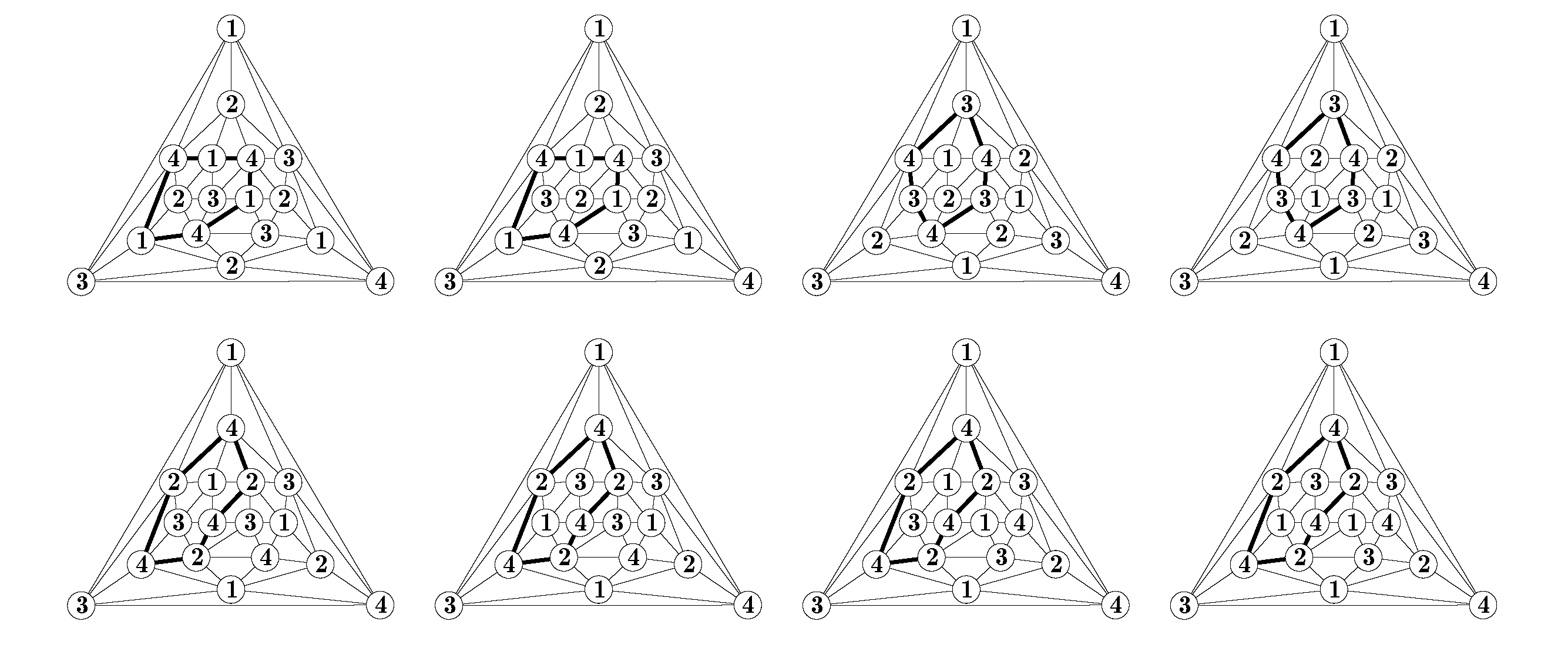}\\
  \caption{A pure-type UBCMPG and all its 4-colorings}\label{fig4}
\end{figure}

Let $G$ be a 4-chromatic maximal planar graph with $\delta(G)\geq 4$, and $f\in C_4^0(G)$ be a cycle-coloring. If $C^2(f)$ contains no UB-cycle, then we call $f$ a \emph{cyclic cycle-coloring}. Let $C_{4U}^0(G)$ denote the set of UBC-colorings of $G$,  $C_{4T}^0(G)$ denote the set of tree-colorings of $G$, and  $C_{4C}^0(G)$ denote the set of cyclic cycle-colorings of $G$. It is evident that $C_{4U}^0(G)\cap C_{4T}^0(G)=\emptyset, C_{4U}^0(G)\cap C_{4C}^0(G)=\emptyset$, $C_{4T}^0(G) \cap C_{4C}^0(G)=\emptyset$, and
\begin{equation}\label{equadd-1}
C_4^0(G)=C_{4U}^0(G)\cup C_{4T}^0(G)\cup C_{4C}^0(G)
\end{equation}

Based on equation (\ref{equadd-1}), we divide UBCMPGs $G$ into the following types.

\begin{itemize}
 \item {Pure-type}: $C_4^0(G)=C_{4U}^0(G)$;
\vspace{0.1cm}

 \item{Tree-type}: $C_4^0(G)=C_{4U}^0(G)\cup C_{4T}^0(G)$;
\vspace{0.1cm}

 \item{Cycle-type}: $C_4^0(G)=C_{4U}^0(G)\cup C_{4C}^0(G)$;
\vspace{0.1cm}

 \item{Hybrid-type}: $C_4^0(G)=C_{4U}^0(G)\cup C_{4T}^0(G)\cup C_{4C}^0(G)$.
\end{itemize}

 Note that there are graphs belonging to each of the above four types.
The 17-order graph $G$ shown in Figure \ref{fig4} is a pure-type UBCMPG, where $C_4^0(G)$ contains eight pairs of complement UBC-colorings. As for pure-type UBCMPGs, we propose the following conjecture.

\begin{conjecture}\label{conjecture1} 
The graph shown in Figure \ref{fig4} is the uniquely pure-type UBCMPG.
\end{conjecture}

The graph  shown in Figure \ref{newfig2-1} is a tree-type UBCMPG of order 8, which contains one pairs of complement UBC-colorings and one tree-coloring. The graph  shown in Figure \ref{fig5} is a tree-type UBCMPG of order 12, which contains two pairs of complement UBC-colorings and two tree-colorings.

\begin{figure}[H]
  \centering
  \includegraphics[width=12.3cm]{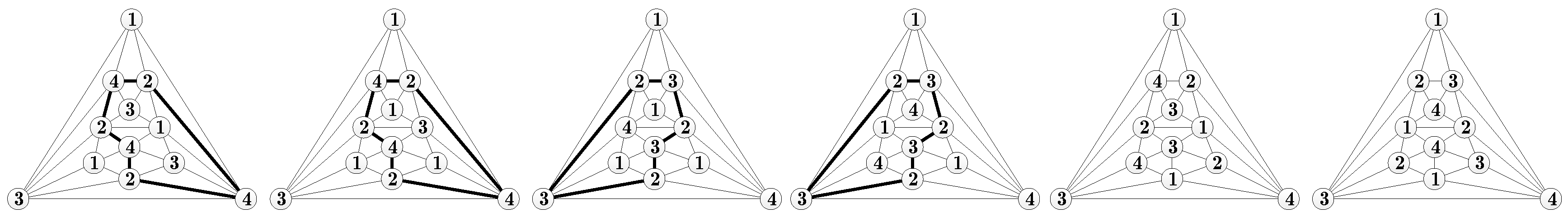}\\
  \caption{A tree-type UBCMPG of order 12 and all of its 4-colorings}\label{fig5}
\end{figure}

{\bf Remark 4} Note that for any $n\geq 2$, there exists a  tree-type UBCMPG $G$ of order $4n$. And $|C_4^0(G)|=2^{n-1}+2^{n-2}$, where the numbers of tree-colorings and UBC-colorings are $2^{n-2}$ and $2^{n-1}$, respectively \cite{r23}.

\subsection{$\sigma$-reconfiguration graph of MPGs}

Let $G$ be a 4-chromatic planar graph, and assume $C_4^0(G)$=$\{f_1,f_2,\ldots,f_n\}$.  The \emph{$\sigma$-reconfiguration graph} of $G$, written as $G^{\sigma}_4$, is a graph with vertex set $V(G^4_{\sigma})$=$\{f_1,f_2,$ $\ldots,f_n\}$, where  two vertices $f_i$ and $f_j$ are adjacent if and only if they are complement colorings, i.e.,
$\sigma(f_i,C)=f_j$ for some bichromatic cycle $C\in C^2(f_i)$, $i,j=1,2,\ldots,n$, $i\neq j$. By this definition,  $G^{\sigma}_4$ can be viewed as a weighted graph, in which each edge $e=f_if_j$ is weighted by the bichromatic $C$ such that $\sigma(f_i,C)=f_j$. If we consider only the topological structure, then the weights of $G^{\sigma}_4$ can be neglected \cite{r12,r20}. Please see the following examples for the illustration of $\sigma$-reconfiguration graphs.

 \begin{figure}[H]
  \centering
  \includegraphics[width=10cm]{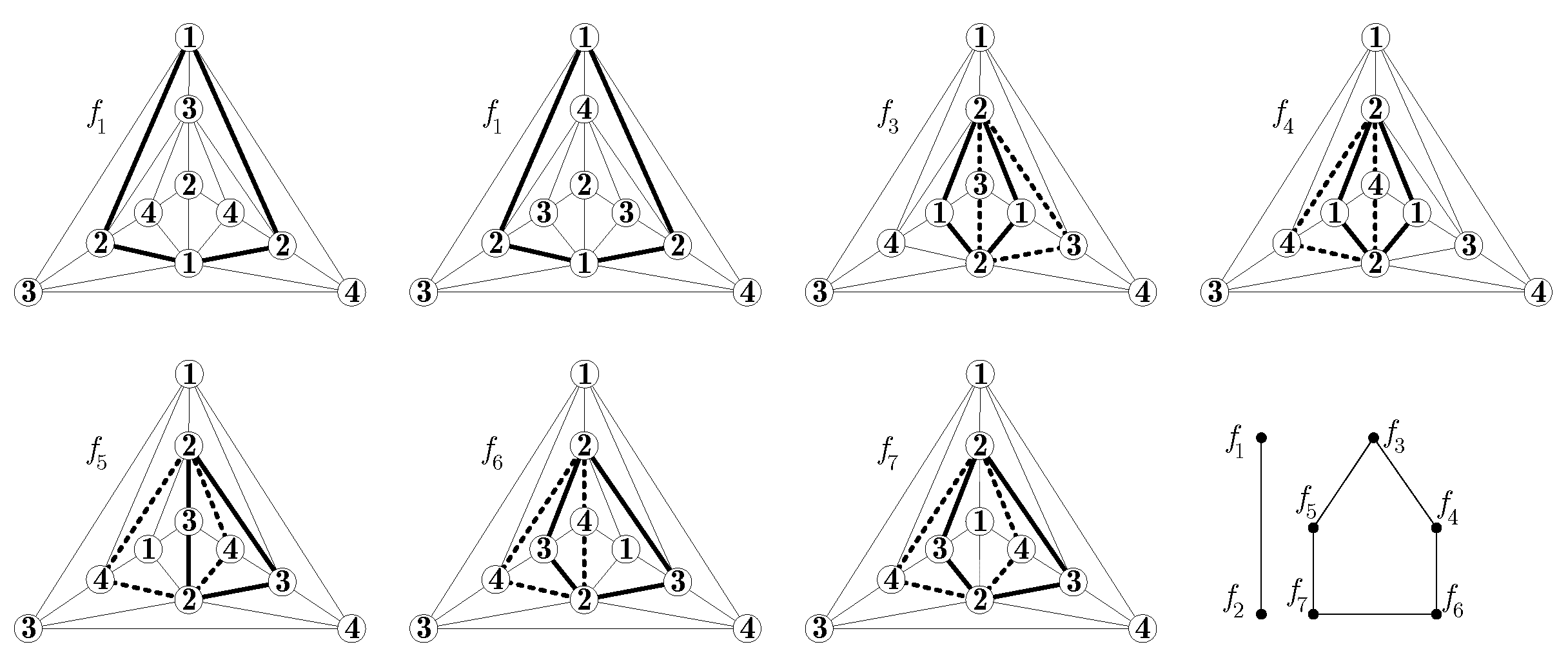}\\
  \caption{A cycle-type UBCMPG of order 10, all of its 4-colorings, and its $\sigma$-reconfiguration graph}\label{fig6}
\end{figure}

 \begin{figure}[H]
  \centering
  \includegraphics[width=13cm]{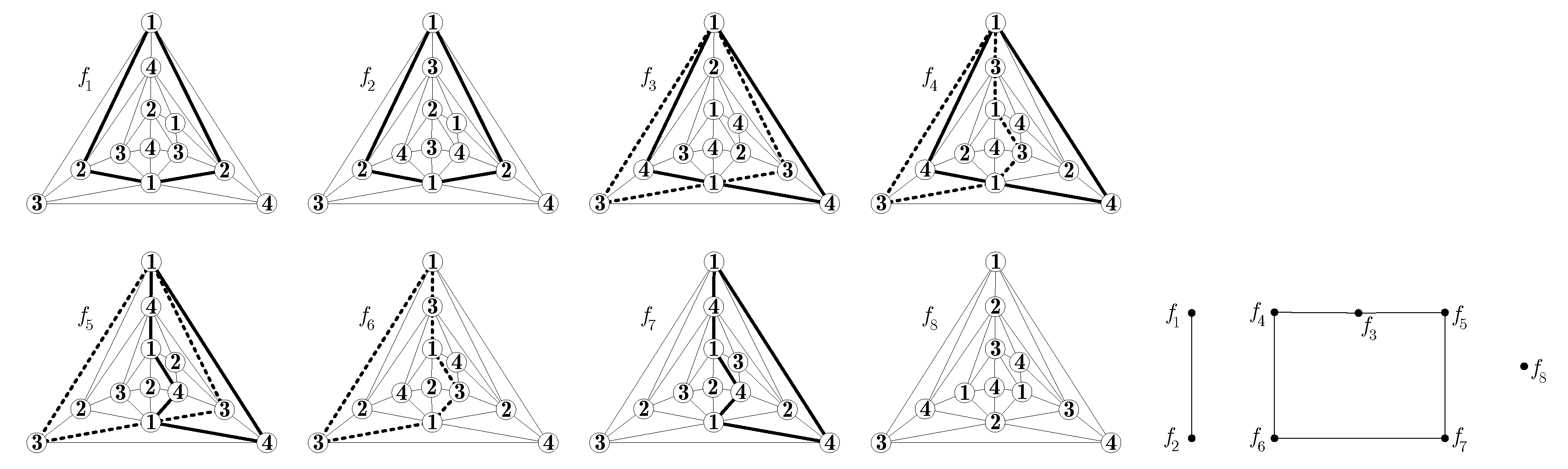}\\
  \caption {A hybrid-type UBCMPG, all of  its 4-colorings, and its $\sigma$-reconfiguration graph}\label{fig7}
\end{figure}

\section{Base-Modules}

\subsection{Definitions and Types}\label{sec4-1}
Let $G^C$ be an SMPG with respect to a cycle $C$, where $|V(C)|=\ell$. If there is another SMPG  $G_1^C$ with respect to $C$ such that $G^C\cap G_1^{C}=C$ and $G^C\cup G_1^{C}=G$ is a UBCMPG with respect to $C$, then we call $G^C$ and $G_1^C$ \emph{$\ell$-base-modules with respect to $C$} (or \emph{$\ell$-base-modules} for short). Let $f$ be a  UBC-coloring with respect to $C$ of $G=G^C\cup G_1^{C}$. We refer to  the restricted coloring of $f$ to $G^C$ and  $G_1^{C}$ as a \emph{module-coloring} of $G^C$ and  $G_1^{C}$, respectively.
Figure \ref{newfig2-1} (d) shows the minimum base-module, which is denote by $B^4$.

In this paper, we only consider 4-base-modules. We will discuss  $\ell$-base-modules ($\ell\geq 5$) and UBCMPGs in other articles.  4-base-modules can be divided into three categories in terms of their module-colorings. Let $G^{C_4}$ be a 4-base-module with respect to 4-cycle $C_4$ and $f$ a module-coloring of $G^{C_4}$.

\begin{itemize}
   \item If $C_4$ is the unique bichromatic cycle of $f$, then $G^{C_4}$ is called a  \emph{tree-type 4-base-module};
   \item If $f$ contains a UB-cycle $C'$ such that $C'\neq C_4$, then  $G^{C_4}$ is  a \emph{cycle-type 4-base-module}. Figure \ref{new1} (a) is a cycle-type 4-base-module, which contains a UB-cycle in the interior of $C_4=v_1v_2v_3v_4v_1$ (marked with  dashed bold lines);
   \item Let $C'(\neq C_4)$ be a bichromatic cycle of $f$. If $C'$ is not a UB-cycle of $f$, then $C'$ is called a \emph{cyclic-cycle} of $G^{C_4}$. If $\mathbb{C}=C^2(F^f(G^{C_4}))\setminus \{C_4\}$ is not empty and contains only cyclic-cycles, then  $G^{C_4}$ is called a  \emph{cyclic-cycle-type 4-base-module}. 
   \noindent Since  the subgraph of $G^{C_4}$ induced by the set of vertices belonging to a cycle $C'\in \mathbb{C}$ and its interior is  a SMPG, denoted by $G^{C'}$, it follows that  the union of $G^{C'}$ over all $C'\in \mathbb{C}$ is one or more SMPGs. Observe that each such SMPG $G^C$ is the union of some $G^{C^1}, G^{C^2}, \ldots, G^{C^k}$, where $C^i\in \mathbb{C}$ for $i=1,2,\ldots,k$; that is, $C=C^1\cup C^2\cup \ldots \cup C^k$. We call $G^C$ a \emph{family of cyclic-cycles} of $\mathbb{C}$, and  $C$ a \emph{shell} of $G^{C_4}$. Figure~\ref{new1}~(b) is a cyclic-cycle-type 4-base-module,
       in which the shell is marked with dashed bold lines. One can readily check that there is only one shell and one family of cyclic-cycles for this 4-base-module.
\end{itemize}

 \begin{figure}[H]
  \centering
 \includegraphics[width=7cm]{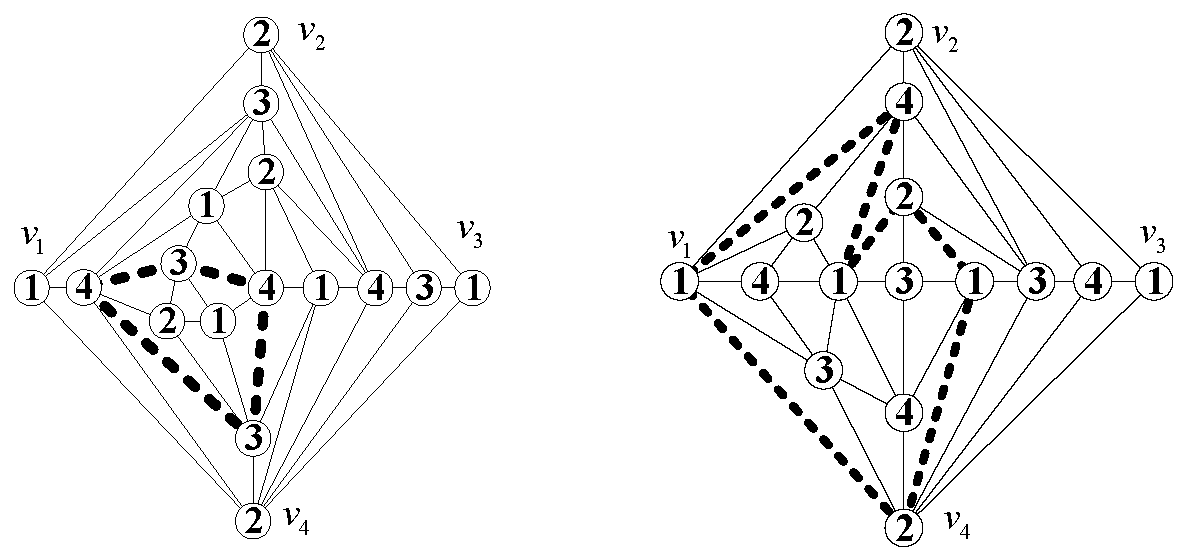}\\
  \caption {Two examples of cycle-type 4-base-module and cyclic-cycle-type 4-base-module: (a) cycle-type; (b) cyclic-cycle-type}\label{new1}
\end{figure}

\subsection{Properties of 4-base-modules}

Let $G^{C_4}$ be a 4-chromatic SMPG. If no specified note, we always let $C_4=v_1v_2v_3v_4v_1$.  Let $F_2(G^{C_4})(\subset C_4^0(G^{C_4}))$ be the set of 4-colorings of $G^{C_4}$ such that $C_4$ is colored with exact two colors. If $f\in F_2(G^{C_4})\neq \emptyset$, without loss of generality, we always assume $f(v_1)=f(v_3)=1, f(v_2)=f(v_4)=2$. Based on this agreement, we denote by $P_{1i}^{f}(v_1,v_3)$  the set of $1i$-path (under $f$) from $v_1$ to $v_3$ and by $P_{2i}^{f}(v_2,v_4)$  the set of $2i$-path (under $f$) from $v_2$ to $v_4$, where $i\in \{3,4\}$. We refer to the paths in $P_{1i}^{f}(v_1,v_3)$ and $P_{2i}^{f}(v_2,v_4)$ for $i=3,4$ as \emph{$ji$-endpoint-paths of $f$}, where $j=1,2$, and use  $\ell^{1i}$ and $\ell^{2i}$ to denote a path in $P_{1i}^{f}(v_1,v_3)$ and $P_{2i}^{f}(v_2,v_4)$, respectively. In particular, when $G^{C_4}$ is a 4-base-module and $f$ is a module-coloring of $G^{C_4}$,  these  endpoint-paths are called \emph{module-paths of $f$}.
Based on these notations, we have the following result.

\begin{theorem}\label{thm3.3}
Let $G^{C_4}$ be a 4-chromatic SMPG, $C_4=v_1v_2v_3v_4v_1$. If $F_2(G^{C_4})\neq \emptyset$, then for any $f\in F_2(G^{C_4})$, exact two of $P_{13}^{f}(v_1,v_3)$, $P_{14}^{f}(v_1,v_3)$, $P_{23}^{f}(v_2,v_4)$ and $P_{24}^{f}(v_2,v_4)$ are nonempty.
\end{theorem}
\begin{proof}
The conclusion follows directly from the fact that $P_{13}^{f}(v_1,v_3)=\emptyset$ if and only if $P_{24}^{f}(v_2,v_4)\neq \emptyset$, and $P_{14}^{f}(v_1,v_3)=\emptyset$ if and only if $P_{23}^{f}(v_2,v_4)\neq \emptyset$. \qed
\end{proof}

 Based on Theorem \ref{thm3.3}, the 4-colorings in $F_2(G^{C_4})$ can be classified into two classes: \emph{cross-coloring} and \emph{shared-endpoint-coloring}.
 Let $f\in F_2(G^{C_4})$.  If $P^f_{1i}(v_1,v_3)\neq \emptyset$ and $P^f_{2i}(v_2,v_4)\neq \emptyset$ for some $i\in \{3,4\}$, then $f$ is called a \emph{cross-coloring} of $G^{C_4}$; see Figure \ref{fig3-2}(a) for an example of \emph{ cross-colorings};  if $P^f_{1i}(v_1,v_3)=\emptyset$ for every $i\in \{3,4\}$ or $P^f_{1i}(v_1,v_3)\neq\emptyset$ for every $i\in \{3,4\}$, then  $f$ is called a \emph{shared-endpoint-coloring} of $G^{C_4}$ on $\{v_2,v_4\}$ (when $P^f_{1i}(v_1,v_3)=\emptyset, i=3,4$) or a \emph{shared-endpoint-coloring} of $G^{C_4}$ on $\{v_1,v_3\}$ (when $P^f_{1i}(v_1,v_3)\neq \emptyset, i=3,4$), respectively; see Figure \ref{fig3-2}(b) for an example of \emph{shared-endpoint-colorings} on $\{v_2,v_4\}$.

 \begin{figure}[H]
  \centering
 \includegraphics[width=6cm]{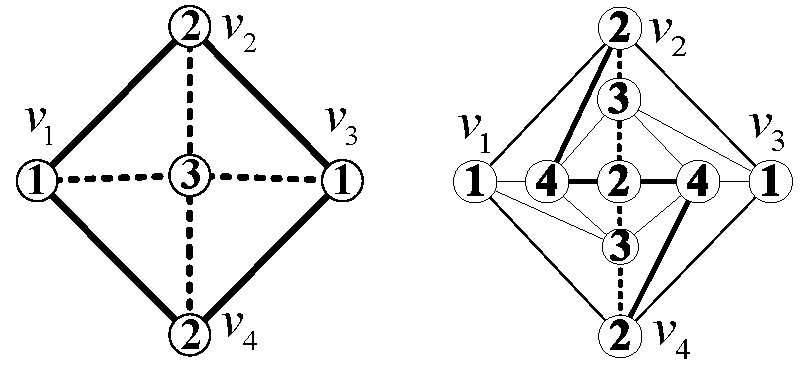}\\
  (a) $f$ \hspace{2cm} (b)  $f_1$
  \caption {Illustration for  cross-colorings and  shared-endpoint-colorings}\label{fig3-2}
\end{figure}

\begin{theorem} \label{thm3.4}
Let  $G^{C_4}$ be a 4-chromatic SMPG such that $F_2(G^{C_4})\neq \emptyset$, where $C_4=v_1v_2v_3v_4v_1$. Then, $G^{C_4}$ is a 4-base-module if and only if there exists a $f_0\in F_2(G^{C_4})$ such that  $F^{f_0}(G^{C_4})$ contains only shared-endpoint-colorings on either $\{v_2,v_4\}$ or $\{v_1,v_3\}$.
\end{theorem}

\begin{proof}
(Necessity) Suppose that $G^{C_4}$ is a 4-base-module, and let $G_1^{C_4}$ be an arbitrary SMPG with respect to $C_4$ such that $G_1^{C_4}\cap G^{C_4}=C_4$ and $G=G_1^{C_4}\cup G^{C_4}$ is a UBCMPG with respect to $C_4$. Clearly, $F_2(G_1^{C_4})\neq \emptyset$, i.e., there exists a coloring $f_1\in F_2(G_1^{C_4})$ such that $|f_1(C_4)|=2$, where  $f_1(v_1)=f_1(v_3)=1$ and $f_1(v_2)=f_1(v_4)=2$. If for any $f\in F_2(G^{C_4})$,  $F^{f}(G^{C_4})$ always contains a cross-coloring $f_c$, then $f_c\cup f_1$ or $f_c\cup \sigma(f_1, C_4)$ contains a bichromatic cycle of $G$ that is intersect with $C_4$. By Theorem \ref{thm2.1}, $G$ is not an  UBCMPG, a contradiction. Therefore, there exists a $f^0\in F_2(G^{C_4})$ such that  $F^{f^0}(G^{C_4})$ contains no cross-coloring, i.e., $F^{f^0}(G^{C_4})$ contains only   shared-endpoint-colorings on $\{v_2,v_4\}$ or $\{v_1,v_3\}$ (by Theorem \ref{thm3.3}). Suppose that there are two $f', f''\in F^{f^0}(G^{C_4})$ such that $f'$ is a shared-endpoint-coloring  on $\{v_1,v_3\}$ and $f''$ is a shared-endpoint-coloring on $\{v_2,v_4\}$ (see Figure \ref{fig3-3} for an illustration of this case). Analogously, we see that $f_1 \cup f'$ or $f_1 \cup f''$ contains a bichromatic cycle of $G$ that intersects with $C_4$, and also a contradiction.

\begin{figure}[H]
  \centering
 \includegraphics[width=7cm]{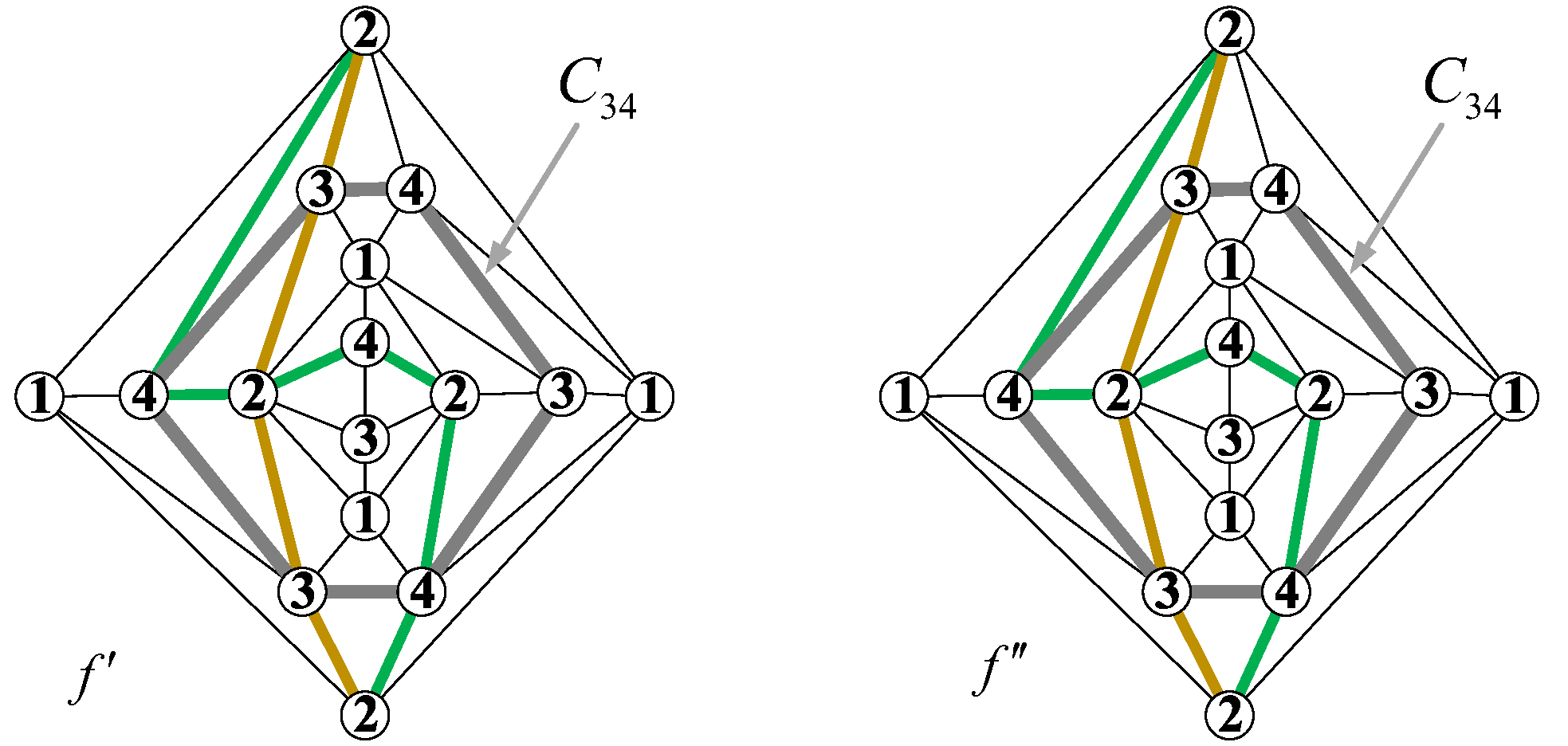}\\
  \caption {An illustration for  the proof of Theorem \ref{thm3.4}}\label{fig3-3}
\end{figure}

(Sufficiency) Let $f^0\in F_2(G^{C_4})$ such that $F^{f_0}(G^{C_4})$ contains only shared-endpoint-colorings $f$ on $\{v_2,v_4\}$ (see Figure \ref{fig3-4} (a)). Now, let $g$ be the 4-coloring of $B^4$ shown in Figure \ref{fig3-4} (b), and let $G=G^{C_4}\cup B^4$. For any $f\in F^{f_0}(G^{C_4})$, let $f'$ be the extended  coloring of $f$ to $G$ such that $f'(B^4)=g(B^4)$. By the assumption, we see that $f$ contains 23-endpoint-path and 24-endpoint-path. However, $B^4$ contains neither 23-endpoint-path nor 24-endpoint-path. Therefore, in $G$, $f'$ contains no $2i$-cycle $(i=3,4)$ that intersects with $C_4$. Additionally, the fact that in $B^4$ $g$ contains both 13-endpoint-path and 14-endpoint-path while in $G^{C_4}$ $f$ contains no 13-endpoint or 14-endpoint, implies that  in $G$, $f'$ contains no $1i$-cycle  $(i=3,4)$ that intersects with $C_4$. So, by Theorem \ref{thm2.1}, $G$ is a UBCMPG  with respect to $C_4$. \qed

\begin{figure}[H]
  \centering
 \includegraphics[width=6cm]{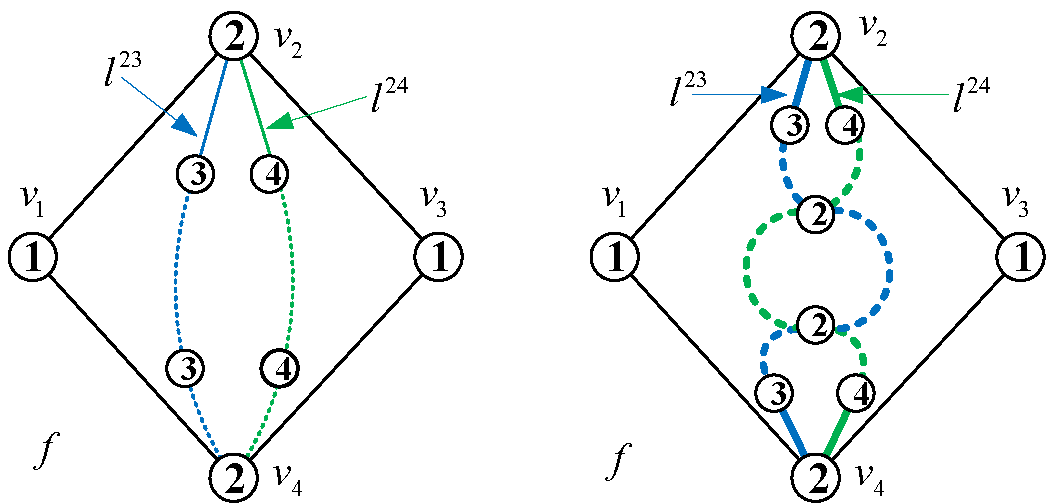}\hspace{1cm}
  \includegraphics[width=2.5cm]{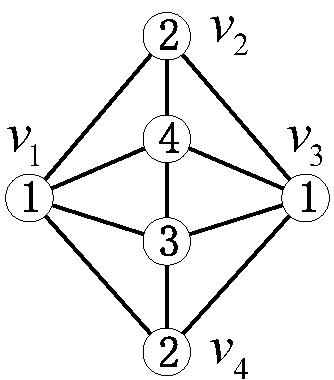}\\
  \hspace{0.5cm}(a) A diagram of shared-endpoint-coloring \hspace{1.5cm} (b)a 4-coloring $g$ of $B^4$
  \caption {An illustration for  the proof of Theorem \ref{thm3.4}}\label{fig3-4}
\end{figure}
 \end{proof}

{\bf Remark 5}
For a given 4-base-module $G^{C_4}$ with respect to $C_4$, let $f$ be
a module-coloring of $G^{C_4}$. In the following, if no specified note, we always assume that a module-path of $f$ is a 23-module-path or a 24-module-path of $f$.

\begin{figure}[H]
  \centering
 \includegraphics[width=7cm]{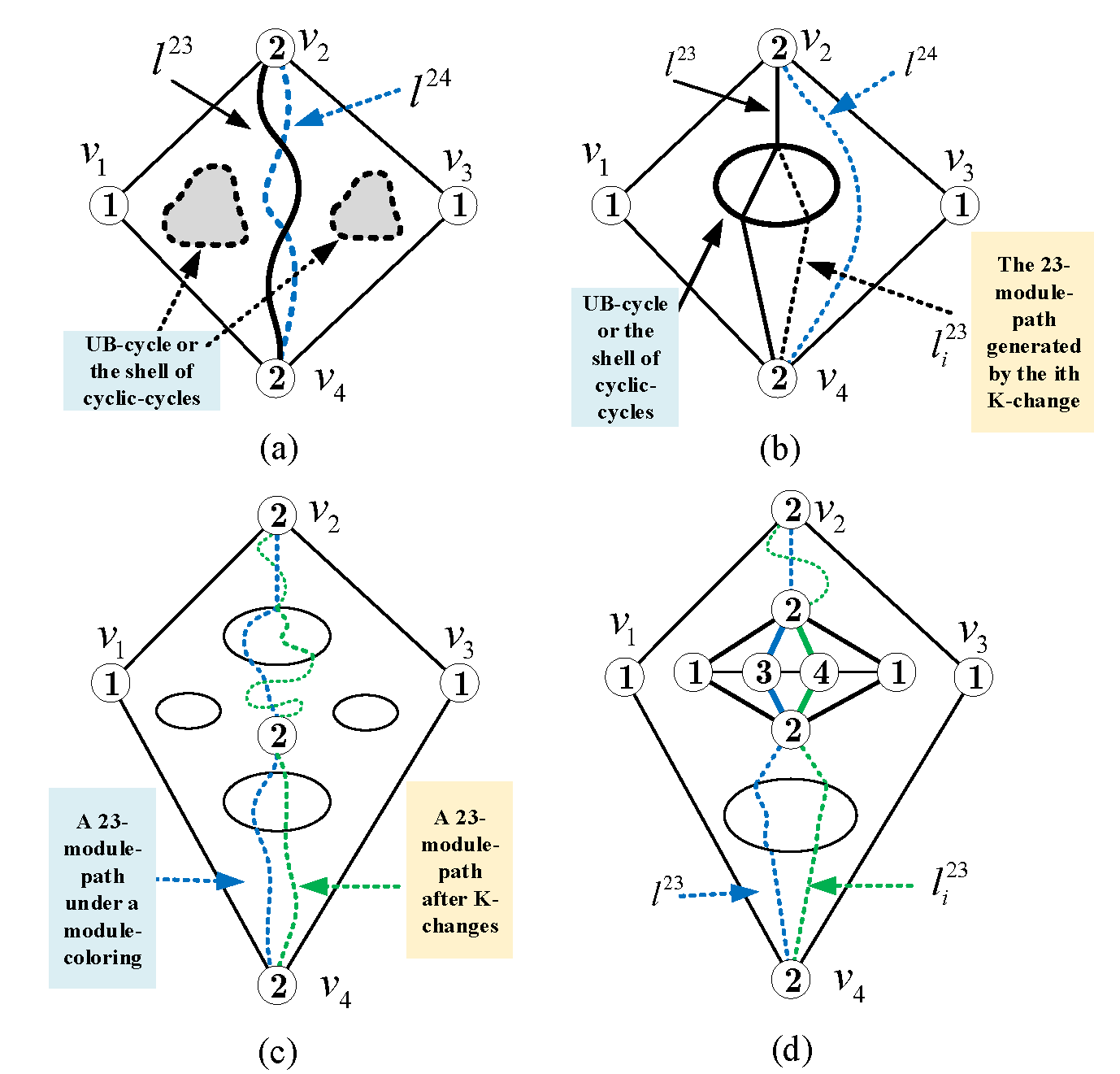}\\
  \caption {The structure and types of cycle-type and cyclic-type 4-base-modules. (a)  SMP-type 4-base-module, (b) one UB-cycle or one family of cyclic-cycles module-path, (c) more than one UB-cycle or one family of cyclic-cycles module-paths, (d) a specific example of UB-cycle}\label{fig3-5}
\end{figure}

According to Theorem \ref{thm3.4}, we have the following corollary.

\begin{corollary}\label{cor3.3}
Suppose that $G^{C_4}$ is a tree-type 4-base-module. Then, there exists a coloring  $f_0\in F_2(G^{C_4})$ such that $F^{f_0}(G^{C_4})=\{f_0\}$ and $f_0$ contains a unique 23-module-path and a unique 24-module-path.
\end{corollary}

Cycle-type 4-base-modules and cyclic-cycle-type 4-base-modules can further be divided into two classes in terms of module-paths: single module path type (SMP-type) and multi-module path type (MMP-type).

Let $G^{C_4}$ be a cycle-type or cyclic-cycle-type 4-base-module and $f_0$ be a module-coloring of  $G^{C_4}$. If all colorings in $F^{f_0}(G^{C_4})$ have the same 23-module-path and 24-module-path, denoted by $\ell^{23}$ and $\ell^{24}$, then $G^{C_4}$ is called a SMP-type 4-base-module. Clearly, if $G^{C_4}$  is an SMP-type 4-base-module, then  no vertex belonging to $\ell^{23}$ or $\ell^{24}$ is in the interior of UB-cycles or cyclic-cycles of $G^{C_4}$; see Figure \ref{fig3-5} (a).

For example, in the cycle-type 4-base-module and the cyclic-cycle-type 4-base-module  shown in Figure \ref{new1}, the number of module-colorings in the equivalence class is 2 and 5, respectively.

Let $G^{C_4}$ be a cycle-type or cyclic-cycle-type 4-base-module. $G^{C_4}$ is called an MMP-type 4-base-module, if there exists a module-coloring $f_1\in F^{f_0}(G^{C_4})$, and under $f_1$ there is a  UB-cycle or a shell $C$ of $G^{C_4}$ such that the interior of $C$ contains a vertex of module-paths.
Figure \ref{fig3-5}(b) presents an example, in which there is only one UB-cycle or one family of cyclic-cycles; Figure \ref{fig3-5}(c) presents an example, in which there are more than one UB-cycle or one family of cyclic-cycles;  Figure \ref{fig3-5}(d) gives a specific example of UB-cycle based on Figure \ref{fig3-5}(c).

In the above, we discussed the characteristics of cycle-type and cyclic-cycle-type 4-base-modules. We will give a more detailed discussion in Section \ref{sec5}.

\subsection{Extending wheel operations and 4-base-modules}
\label{sec:ecsystem}

In \cite{r18}, the author introduced the extending and contracting system (EC-system) $<K_4; \Phi=\{\zeta^+_2,\zeta^+_3,\zeta^+_4,\zeta^+_5, \zeta^-_2,\zeta^-_3,\zeta^-_4,\zeta^-_5\}>$, in which $K_4$ is the starting graph, and $\zeta^+_i$ and $\zeta^-_i$ are a pair of operators (called extending and contracting $i$-wheel operation, respectively), $i=2,3,4,5$. Because we will use this system to prove our subsequent conclusions, here we give a description of these operations.

 \begin{figure}[H]
  \centering
  \includegraphics[width=6.5cm]{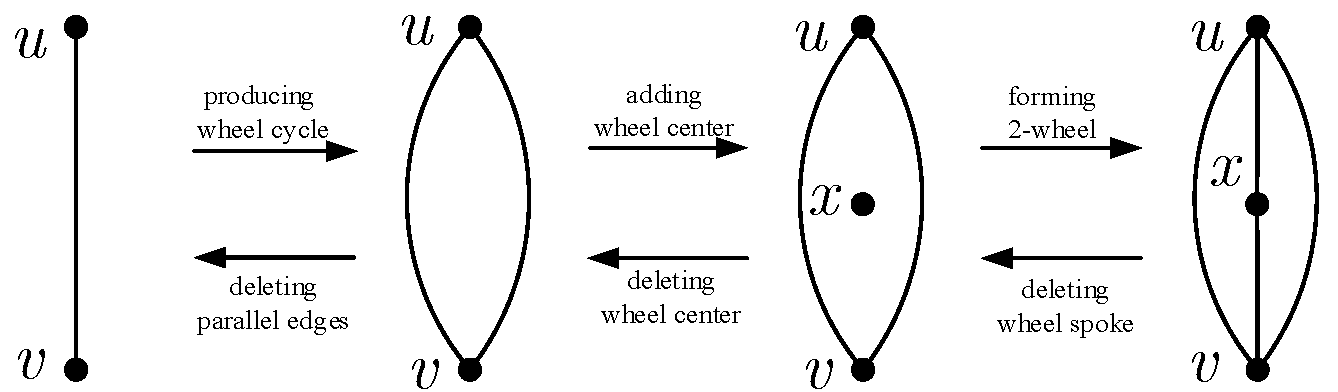} \hspace{0.2cm}
   \includegraphics[width=4cm]{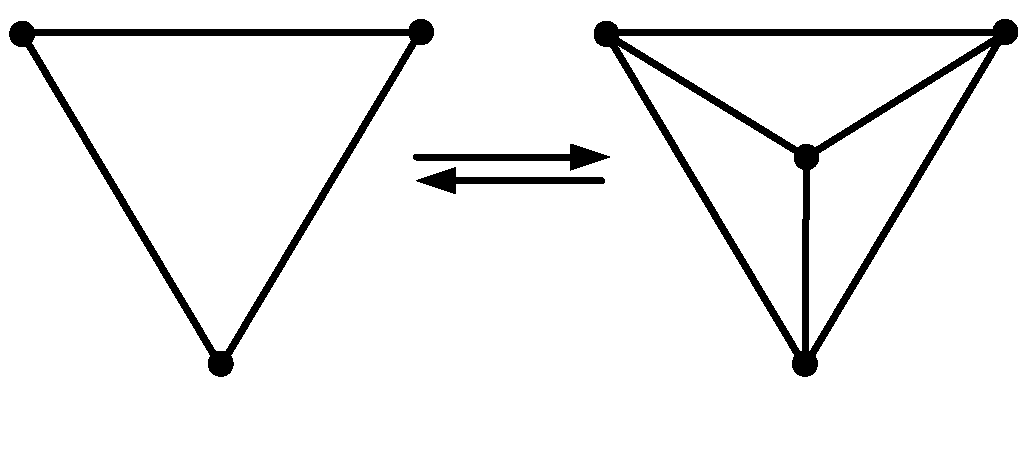}\\
   \hspace{2.3cm}(a)\hspace{6cm} (b)
  \caption {(a) E2WO and C2WO ~~(b) E3WO and C3WO }\label{23-wheel}
\end{figure}

\emph{The extending $2$-wheel operation} (E2WO):  adding a new edge first between two adjacent vertices $u,v$ (this yields two parallel edges between $u$ and $v$), and then adding a new vertex $x$ on the face bounded by these two parallel edges and connecting $x$ to $u$ and $v$ (this generates a $2$-wheel). \emph{The contracting $2$-wheel operation} (C2WO): for a given 2-wheel $x$-$uvu$, removing first its center $x$ and the edges $xu$ and $xv$, and then one parallel edge. The whole process is shown in Figure \ref{23-wheel}(a). We use $\zeta^+_2$ and $\zeta^-_2$ to denote the E2WO and C2WO, respectively.

\emph{The extending $3$-wheel operation} (E3WO), denoted by $\zeta^+_3$, refers to the transformation from triangle to a 3-wheel, and its inverse process is \emph{the contracting $3$-wheel operation} (C3WO), denoted by $\zeta^-_3$; see Figure \ref{23-wheel}(b).

\begin{figure}[H]
  \centering
  \includegraphics[width=11.5cm]{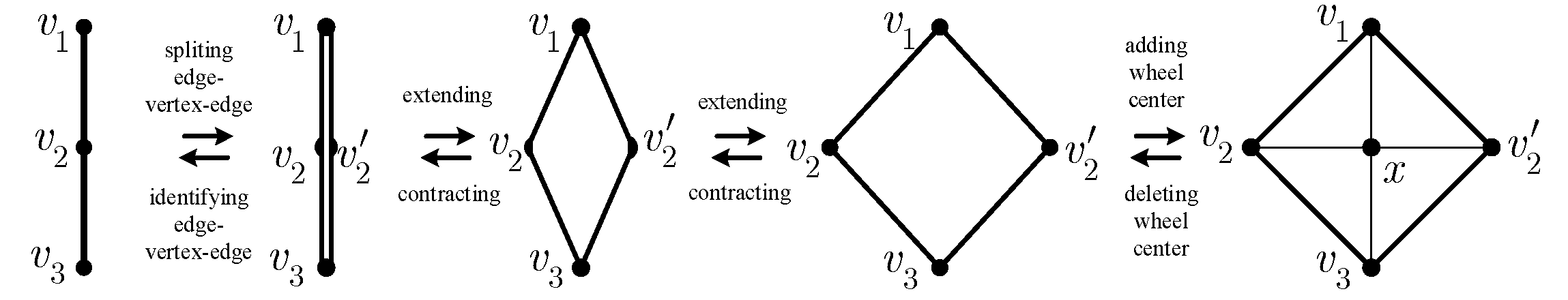}
  \caption {E4WO and C4WO}\label{4-wheel}
\end{figure}

\begin{figure}[H]
  \centering
  \includegraphics[width=12.5cm]{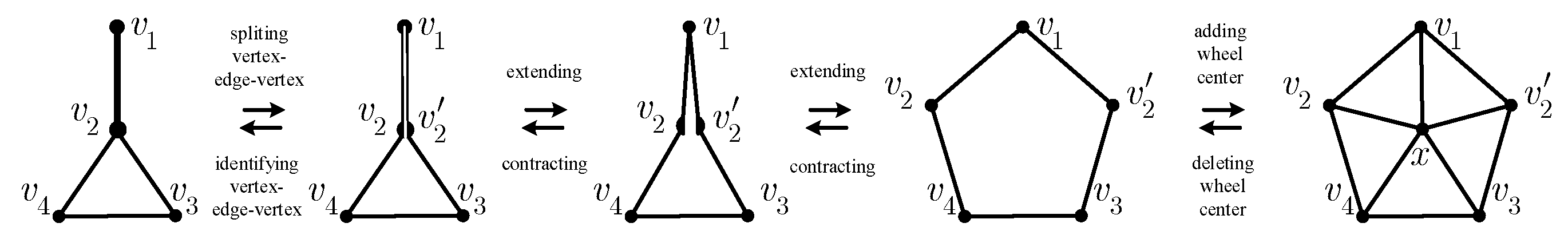}
  \caption {E5WO and C5WO}\label{5-wheel}
\end{figure}

\emph{The extending $4$-wheel operation} (E4WO), denoted by $\zeta^+_4$, refers to the transformation from a path $P_3=v_1v_2v_3$ of length 2 to a 4-wheel, and its inverse process is \emph{the contracting $4$-wheel operation} (C4WO), denoted by $\zeta^-_4$; see Figure \ref{4-wheel}. Note that in the process of $\zeta^+_4$,  edge $v_1v_2$, vertex $v_2$ and edge $v_2v_3$ are split into $v_1v_2$ and $v_1v'_2$, $v_2$ and $v'_2$, and $v_2v_3$ and $v'_2v_3$, respectively, and  edges incident with $v_2$ lied on left side of $P_3$ (in the original graph) are incident with $v_2$ while edges incident with $v_2$ lied on right side (in the original graph) of $P_3$ are incident with $v'_2$; in the process of $\zeta^-_4$, $v_2$ and $v'_2$ are identified into a new vertex incident with all edges incident with  $v_2$ and $v'_2$ in the original graph. $v_2$ and $v'_2$ are called \emph{contracted vertices}.

\emph{The extending $5$-wheel operation} (E5WO), denoted by $\zeta^+_5$, refers to the transformation from a funnel (the first graph from the left shown in Figure \ref{5-wheel}) to a 5-wheel, and its inverse process is \emph{the contracting $5$-wheel operation} (C5WO), denoted by $\zeta^-_5$; see Figure \ref{5-wheel}. Note that in the process of $\zeta^+_5$, vertex $v_2$  are split into  $v_2$ and $v'_2$, edge $v_1v_2$ into $v_1v_2$ and $v_1v'_2$,  and edges incident with $v_2$ lied on left side of path $v_1v_2v_4$ (in the original graph) are incident with $v_2$ while edges incident with $v_2$ lied on right side of path $v_1v_2v_4$ (in the original graph)  are incident with $v'_2$; in the process of $\zeta^-_5$, $v_2$ and $v'_2$ are identified into a new vertex incident with all edges incident with  $v_2$ and $v'_2$ in the original graph. $v_2$ and $v'_2$ are called \emph{contracted vertices}.

{\bf Remark 6}. {In this paper, the C$i$WO ($i=2,3,4,5$) is always assumed to be implemented under a given 4-coloring. When $i=4$, the two vertices of the 4-wheel assigned with the same color are contracted vertices. In addition, based on a 4-coloring,  if the object of E$i$WO (i=2,3,4,5) are colored with at most three colors, then in the resulting graph we color the wheel center (the new vertex) with a color not assigned to vertices of wheel cycle and remain other vertices' color unchanged; in particular, when $i=4,5$, we color $v'_2$ with the color assigned to $v_2$.

Let $G^C$ be a 4-chromatic SMPG, $f\in C_4^0(G^C)$, and $W_4=u$-$u_1u_2u_3u'_2u_1$ be a 4-wheel subgraph of $G^C$, where $u$ is the wheel center of $W_4$ and $C_4^1=u_1u_2u_3u'_2u_1$ is the cycle of $W_4$. If $|\{u_2,u'_2\}\cap V(C)|\leq 1$ and  $f(u_2)=f(u'_2)$, then we call $W_4$ a \emph{contractible 4-wheel} with respect to $f$; see Figure \ref{fig3-6} (a).

 \begin{figure}[H]
  \centering
  \includegraphics[width=7cm]{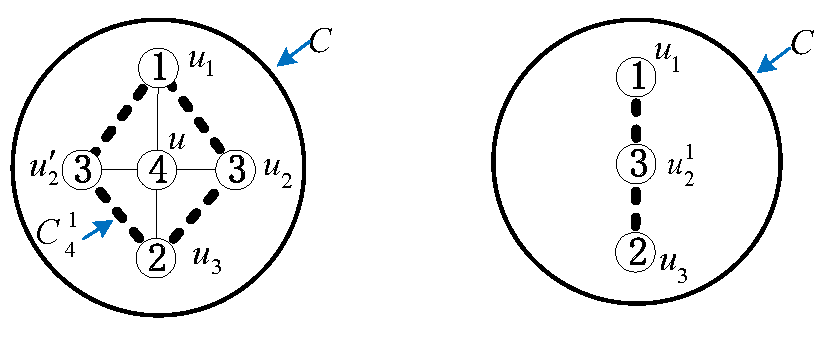}\\
  (a) $G^C$ \hspace{3.2cm}(b) $H^C$
  \caption {Illustration for contractible 4-wheels}\label{fig3-6}
\end{figure}

 \begin{figure}[H]
  \centering
  \includegraphics[width=12cm]{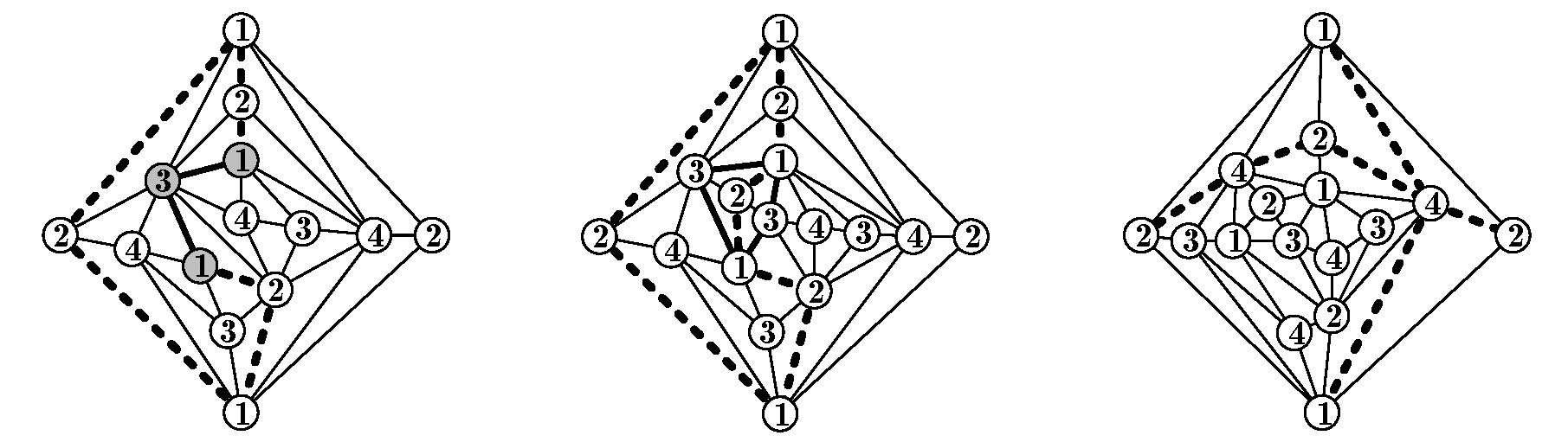}\\
  (a)   \hspace{3.5cm}(b)  \hspace{3.5cm}(c)
  \caption {An example for Remark 7}\label{fig3-7}
\end{figure}

\begin{theorem}\label{thm3-6}
 Let $G^C$ be a  4-base-module,  $f$  a module-coloring of $G^C$, and $W_4=u$-$u_1u_2u_3u'_2u_1$  a contractible 4-wheel with respect to $f$ such that $f(u_2)=f(u'_2)$, where  $C_4^1=u_1u_2u_3u'_2u_1$ is the cycle of $W_4$.  Then, the resulting graph, denoted by $H^C$, obtained from $G^C$ by implementing a C4WO on $W_4$ ($u_2$ and $u'_2$ are contracted vertices)  is a 4-base-module.
\end{theorem}

\begin{proof}
Since $G^C$ is a  4-base-module, there exists another 4-base-module $G_1^C$ such that $G=G^C\cup G_1^C$ is a UBCMPG with respect to $C$. Denote also by $f$ the extended coloring of $f$ to $G$. Then, $f$ is a UBC-coloring of $G$ with respect to $C$. Therefore, for any $f_0\in F^{f}(G), C\in C^2(f_0)$.

Let $H$ be the resulting graph obtained from $G$ by implementing a C4WO on $W_4$ ($u_2$ and $u'_2$ are contracted vertices).  Observe that  $f(u_2)=f(u'_2)$. We let $g$ be  the  coloring of $f$ restricted to $H$, i.e., $g(v)=f(v)$ for every $v\in V(H)\cap V(G)$ and $g(u')=f(u'_2)$ where $u'$ is the new vertex obtained by contracting $u_2$ and $u'_2$ (see Figure \ref{fig3-6} (b)). Then, for every 4-coloring  $g_0\in F^g(H)$, $|g_0(C)|=2$,  since $C^2(g)\subseteq C^2(f)$. Therefore, $H$ is a UBCMPG with respect to $C$ and  hence $H^C$ is a 4-base-module. \qed
\end{proof}

{\bf Remark 7} Note that the reverse of Theorem \ref{thm3-6} is not always true. For example, the graph $H$ shown in Figure \ref{fig3-7} (a) is a 4-base-module. We conduct E4WO for $H$ on the path of length 2, $P$ (marked by the solid bold lines in the graph) and denote by $H'$ the resulting graph (see Figure \ref{fig3-7} (b)). Observe that $H'$ contains cross-coloring (see Figure \ref{fig3-7} (c)). Therefore, $H'$ is not a 4-base-module.

\section{Decycle Theorem}\label{sec5}

Let $G^{C_4}$ be a 4-base-module and $f$ be a module-coloring of $G^{C_4}$, where $C_4=v_1v_2v_3v_4v_1$,  $f(v_1)=f(v_3)=1$, $f(v_2)=f(v_4)=2$, and $P_{2i}^f(v_2,v_4)\neq \emptyset$ for $i=3,4$. If there exists a 4-coloring $f^*\in C_4^0(G^{C_4})$ such that $f^*(v_2)\neq f^*(v_4)$, then we call $f^*$ a \emph{decycle coloring} of $G^{C_4}$. We use $F^*(G^{C_4})$ to denote the set of decycle colorings of $G^{C_4}$. If $F^*(G^{C_4})\neq \emptyset$, then
$G^{C_4}$ is said to be  \emph{decyclizable}.

\subsection{Decycle coloring of recursive 4-base-modules}

\begin{theorem}\label{thm4.1}
Let $G^{C_4}$ be a 4-base-module, $f$ a module-coloring of $G^{C_4}$, and $W_4$ a contractible 4-wheel with respect to $f$. If the graph  $H^{C_4}$ obtained from $G^{C_4}$ by a C4WO on $W_4=u-u_1u_2u_3u'_2u_1$ is decyclizable,  then $G^{C_4}$ is decyclizable, where the contracted vertices are $u_2$ and $u'_2$.
\end{theorem}

\begin{proof}
Let $C_4=v_1v_2v_3v_4v_1$, and without loss of generality, we assume that $f(v_1)=f(v_3)=1, f(v_2)=f(v_4)=2$, and $P_{2i}^f(v_2,v_4)\neq \emptyset$ for $i=3,4$.
By Theorem \ref{thm3-6},  $H^{C_4}$ is a 4-base-module. Let $L$ be the path of length 2 in $H^{C_4}$ obtained by contracting $W_4$. Then, $G^{C_4}$ can be obtained by conducting an E4WO on $L$. If $H^{C_4}$ is decyclizable, then there exists a 4-coloring $f'\in C_4^0(H^{C_4})$ such that $f'(v_2)\neq f'(v_4)$. Let $f^*$ be the extended coloring of $f'$ to $G^{C_4}$ by letting $f^*(u_2)=f^*(u'_2)=f'(u_2^1)$ and $f^*(u)=c$, where $u_2^1$ is the new vertex obtained by contracting $u_2$ and $u'_2$, and $c\in \{1,2,3,4\}\setminus \{f'(u_1),f'(u_3), f'(u_2^1)\}$. Obviously, $f^*(v_2)\neq f^*(v_4)$. This proves that $G^{C_4}$ is decyclizable. \qed
\end{proof}

Let $G^{C_4}$ be a 4-base-module containing a contractible 4-wheel with respect to a module-coloring of $G^{C_4}$. Observe that the graph $G_1^{C_4}$ obtained from  $G^{C_4}$ by contracting the contractible 4-wheel is also a 4-base-module. If  $G_1^{C_4}$  contains a contractible $k$-wheel under some module-coloring of $G_1^{C_4}$, $k\in\{2,3,4\}$, then we can obtain a 4-base-module $G_2^{C_4}$ by conducting a C$k$WO on the contractible $k$-wheel of $G_1^{C_4}$. In this way, by continuously conducting  C$k$WOs   on contractible $k$-wheels, $G^{C_4}$ can be contracted into a  4-base-module $G_{\ell}^{C_4}$ ($\ell\geq 1$) which contains no contractible $k$-wheel. If $G_{\ell}^{C_4}$ is isomorphic to $B^4$, then we call $G^{C_4}$ a \emph{recursive 4-base-module}.

\begin{figure}[H]
  \centering
  \includegraphics[width=9cm]{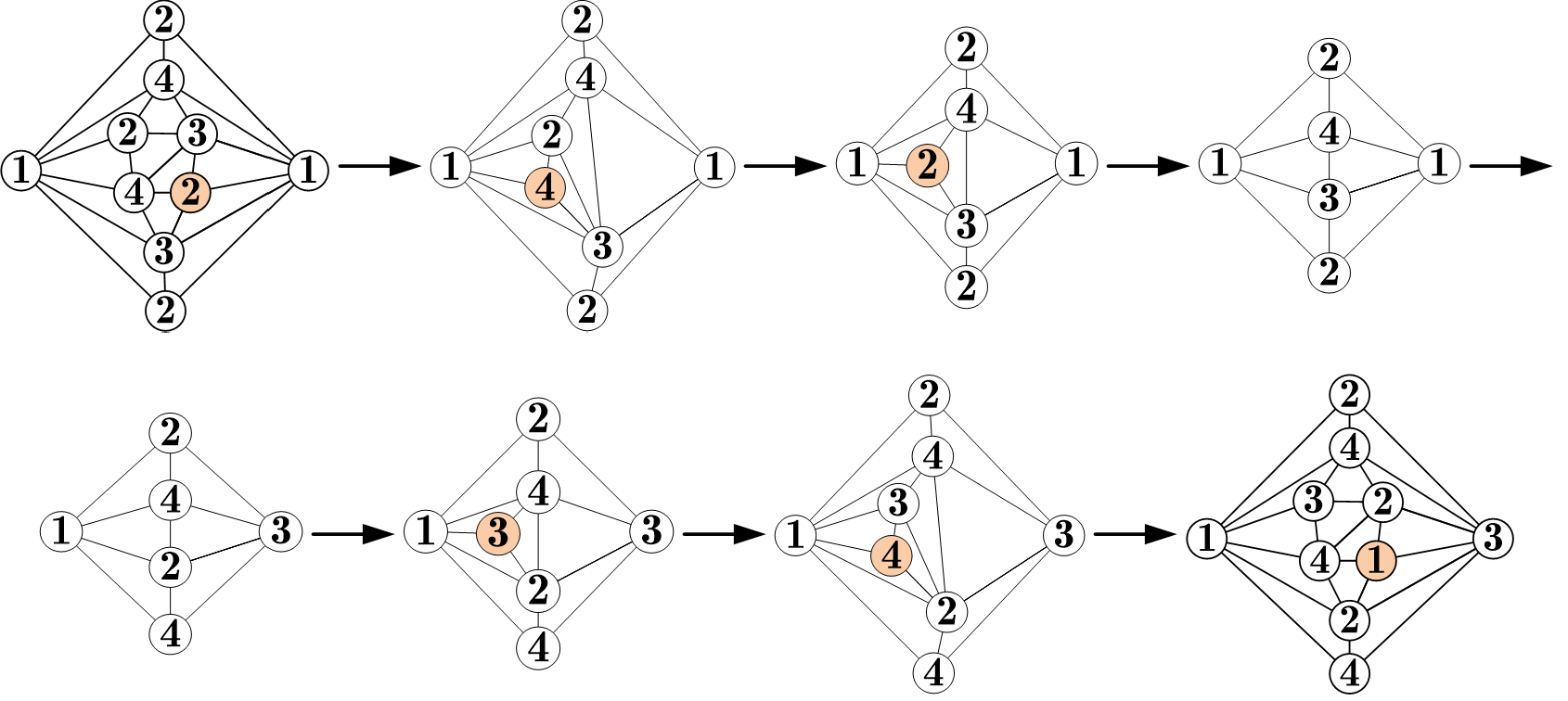}\\
  \caption {Illustration of the decycle process of a recursive 4-base-module}\label{newfig4-1}
\end{figure}

 \begin{figure}[H]
  \centering
  \includegraphics[width=11cm]{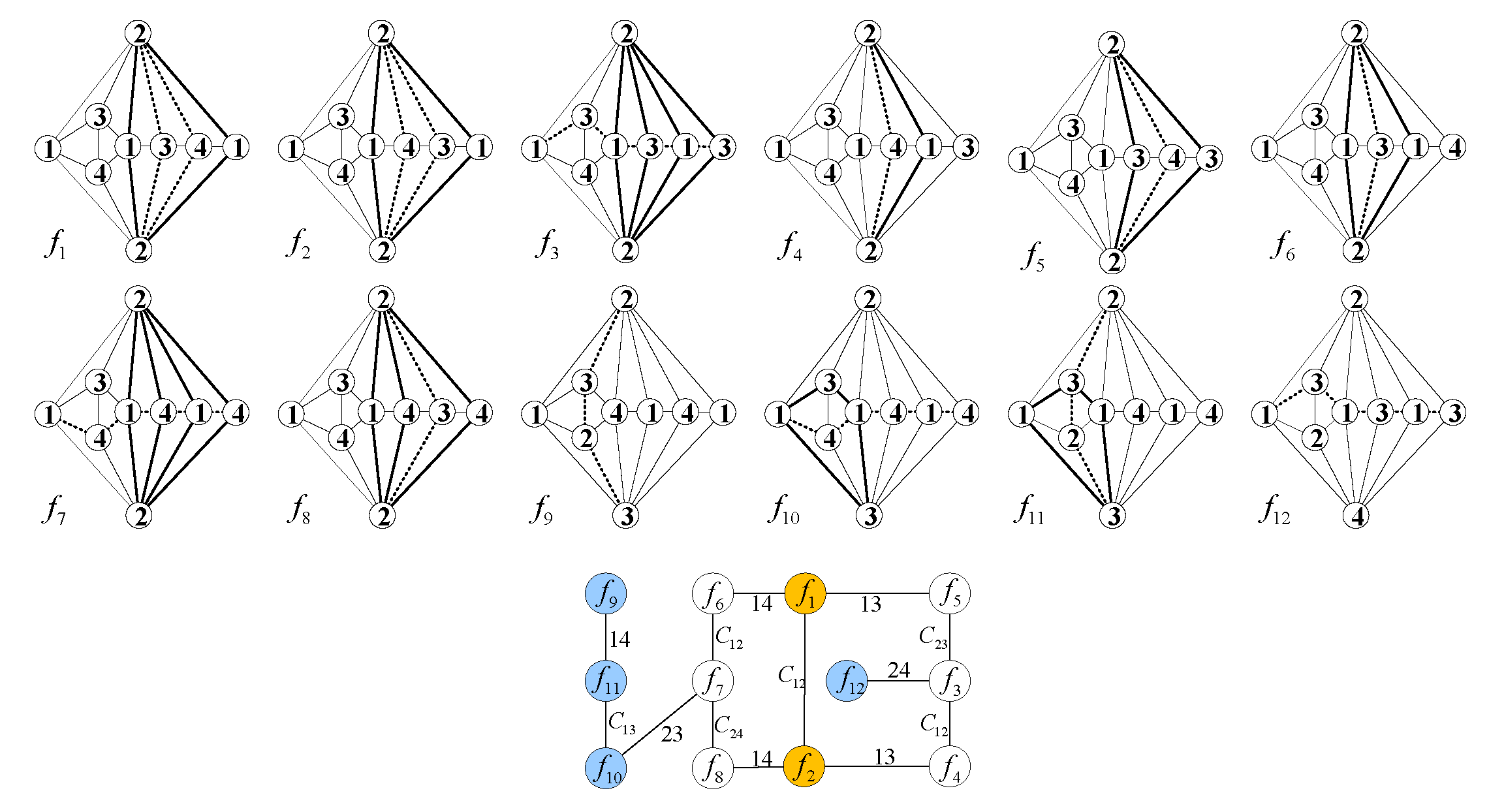}\\
  \caption {A Kempe 4-base-module and its $\sigma$-reconfiguration graph, where the symbol $ij$ labeled at edges means to interchange the colors $i,j$ on an $ij$-component, $ij\in \{13,14,23,24\}$}\label{fig4-3}
\end{figure}

Note that $B^4$ is decyclizable. Therefore, by Theorem \ref{thm4.1}, recursive 4-base-modules are decyclizable. Figure \ref{newfig4-1} illustrates the decycle process of a recursive 4-base-module.

\begin{corollary}\label{cor4.2}
Recursive 4-base-modules are decyclizable.
\end{corollary}

 \subsection{Kempe and non-Kempe 4-base-modules}

Let $G^{C_4}$ be a 4-base-module. If there are $f\in F_2(G^{C_4})$ and $f^*\in F^*(G^{C_4})$  such that $f^*\in F^{f}(G^{C_4})$ (i.e. $f$ and $f^*$ are in the same component of $\sigma$-reconfiguration graph of $G^{C_4}$), then we call $G^{C_4}$ a \emph{Kempe 4-base-module}; otherwise a \emph{non-Kempe 4-base-module}. For instance, the graph  shown in Figure \ref{fig4-3} is a Kempe 4-base-module. This graph  has totally twelve distinct 4-colorings, including two module-colorings ($f_1, f_2$) and four decycle colorings ($f_9,f_{10},f_{11},f_{12}$). Since the  $\sigma$-reconfiguration graph of this graph is connected, every decycle coloring can be obtained from an arbitrary given 4-coloring.  Note that we omit some edges in the $\sigma$-reconfiguration graph  shown in Figure \ref{fig4-3}, under the condition that its connectivity does not get mangled, where the symbol $ij$ labeled at some edges means to interchange the colors $i,j$ on an $ij$-component, $ij\in \{13,14,23,24\}$.

\begin{figure}[H]
  \centering
  \includegraphics[width=12cm]{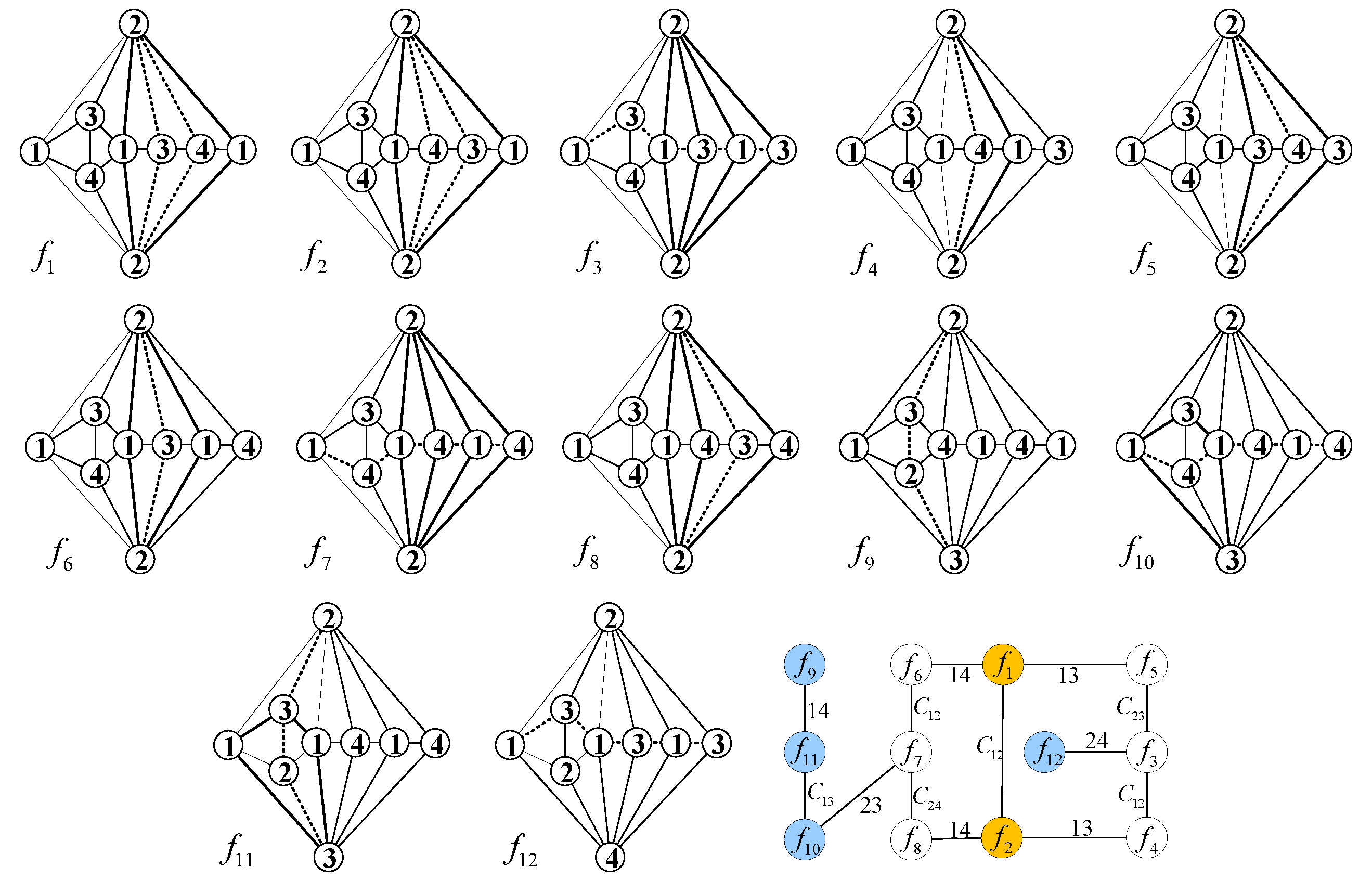}\\
  \caption {A non-Kempe 4-base-module and its $\sigma$-reconfiguration graph}\label{fig4-4}
\end{figure}

As for non-Kempe 4-base-module, we see the graph shown in Figure \ref{fig4-4}.  It has twenty-four distinct 4-colorings, including four module-colorings ($f_1,f_3,f_5,f_{11}$) and twelve decycle colorings ($f_{13}\sim f_{24}$).  Observe that the $\sigma$-reconfiguration graph of the graph has two components: one contains all of the four module-colorings; the other contains only decycle colorings.

\subsection{Module-cycle and non-module-cycle 4-base-modules}

Let $G^{C_4}$ be a 4-base-module with $\delta\geq 4$, where $C_4=v_1v_2v_3v_4v_1$. Suppose that there exists a module-coloring $f$ of $G^{C_4}$  which contains a unique $2i$-module path $\ell^{2i}$, where $i\in \{3,4\}$.  Let $j\in \{3,4\}\setminus \{i\}$. For any two vertices $x_1, x_2\in V(\ell^{2i})$ such that $f(x_1)=f(x_2)=k (\in \{2,i\})$, if $x_1$ and $x_2$ are connected by  a $kj$-path and $1k$-path of $f$ which lies in the two sides (left and right) of $\ell^{2i}$  (see Figure \ref{final4-4} (a)), then we  carry out a $K$-change for the $1j$-component of $f$ that contains vertex $v_1$ and obtain a new 4-coloring $f_1$ of $G^{C_4}$ (see Figure \ref{final4-4} (b)). Clearly, $f_1$ has a $1k$-cycle (or a $kj$-cycle), denoted by $C_{1k}$ (or $C_{kj}$), in the interior of which there exist vertices of $\ell^{2i}$. We then call $C_{1k}$ (or $C_{kj}$) a \emph{module-cycle} or \emph{$1k$-module-cycle} (or \emph{$kj$-module-cycle}) if the resulting 4-coloring (say $f'_1$), obtained from $f_1$ by conducting a $K$-change in the interior of $C_{1k}$ (or $C_{kj}$), contains no $2i$-module-path from $v_2$ to $v_4$, i.e., $f'_1$ contains a $1j$-path from $v_1$ to $v_3$ (see Figure \ref{final4-4} (c)); otherwise, $C_{1k}$ (or $C_{kj}$) is called a \emph{non-module-cycle}. If there exist $k$ and $j$ mentioned above such that $C_{1k}$ or $C_{kj}$ is  a module-cycle, then we say that $\ell^{2i}$ \emph{is  incident to a module-cycle}; otherwise, we say that $\ell^{2i}$ \emph{is not incident to a module-cycle}. If $\ell^{2i}$, for some $i\in \{3,4\}$,  is incident to a module-cycle, then $G^{C_4}$ is called a \emph{module-cycle 4-base-module}; otherwise, a \emph{non-module-cycle 4-base-module}.

 \begin{figure}[H]
  \centering
  \includegraphics[width=8cm]{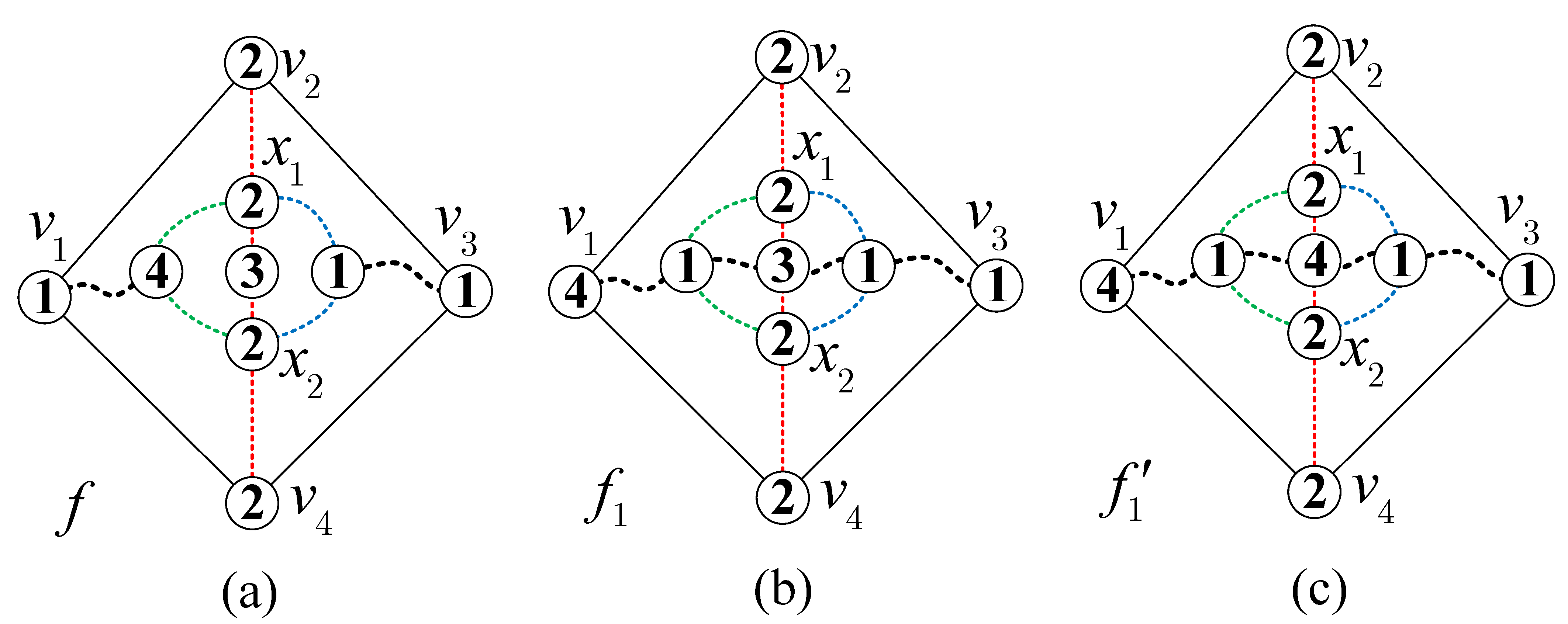}\\
  \caption {Diagram illustration for module-cycle}\label{final4-4}
\end{figure}

\begin{theorem} \label{thm-mcycle}
Let $G^{C_4}$ be a 4-base-module, where $C_4=v_1v_2v_3v_4v_1$. If $G^{C_4}$ has a module-coloring $f$ which contains a unique $ij$-module-path $\ell^{ij}$ for some $i\in \{1,2\}$ and $j\in \{3,4\}$ and $\ell^{ij}$ contains a vertex of degree 4, then $G^{C_4}$ is decyclizable.
\end{theorem}

 \begin{figure}[H]
  \centering
  \includegraphics[width=10cm]{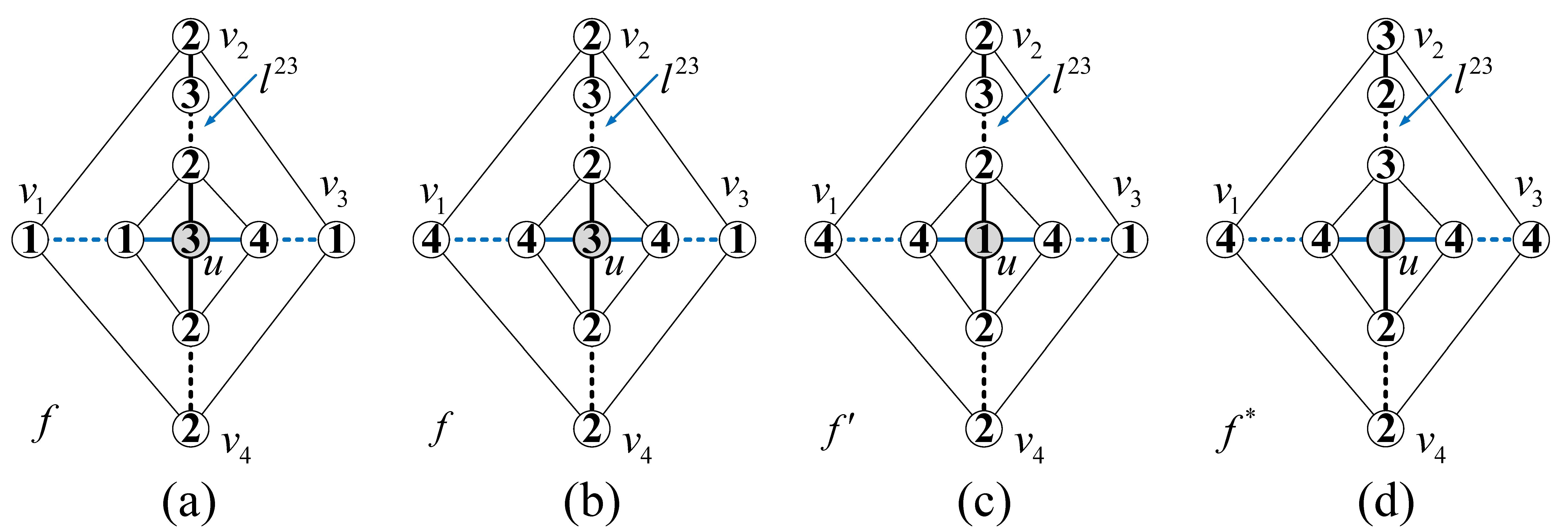}\\
  \caption {Illustration for the proof of Theorem \ref{thm-mcycle}}\label{final4-5}
\end{figure}

\begin{proof}
Without loss of generality, let $f(v_1)=f(v_3)=1,f(v_2)=f(v_4)=2,$ $\ell^{ij}\in P^{f}_{23}(v_2,v_4)$, i.e., $i=2,j=3$. Let $u\in V(\ell^{23})$ be a vertex of degree 4 and $f(u)=3$, and $N_{G^{C_4}}(u)=\{u_1,u_2,u_3,u_4\}$, where $u_1,u_3\in \ell^{23}$. Clearly, $f(u_2)\in \{1,4\}$ and $f(u_4)\in \{1,4\}$. It is enough to consider the case that $f(u_2)=f(u_4)$, since if $f(u_2)\neq f(u_4)$, then we can conduct a $K$-change for the 14-component (of $f$) containing $u_2$ and obtain a new 4-coloring (still denoted by $f$) under which $u_2$ and $u_4$ are colored with the same color; see the first two graphs in Figure \ref{final4-5}. Then, the subgraph induced by $N_{G^{C_4}}(u)$ is a bichromatic cycle $C_{12}$ or $C_{24}$. Now, under $f$, we conduct a $K$-change in the interior of $C_{12}$ or $C_{24}$ and denote by $f'$ the resulting 4-coloring. Clearly, $P^{f'}_{14}(v_1,v_3)\neq \emptyset$, which implies that $v_2$ and $v_4$ are not in the same 23-component of $f'$. Therefore, after conducting a $K$-change for the 23-component of $f'$ containing $v_2$, we obtain a 4-coloring $f^*$ such that $f^*(v_2)\neq f^*(v_4)$ and complete the proof. See Figure \ref{final4-5} for an illustration of this process. \qed
\end{proof}

As for the non-module-cycle, we see Figure \ref{final4-6} (a). Let $f$ be the 4-coloring shown in Figure \ref{final4-6} (a). We first conduct a $K$-change for the 13-component of $f$ containing $v_3$ and denote by $f_1$ the resulting 4-coloring; see Figure \ref{final4-6} (b). Observe that $f_1$ has a 23-cycle $C_{23}$. Then, under $f_1$, we conduct a $K$-change for the 14-component of $f_1$ in the interior of $C_{23}$ and obtain a new 4-coloring $f'_1$. Since $f'_1$ has a 23-module-path, $C_{23}$ is a non-module-cycle.

 \begin{figure}[H]
  \centering
  \includegraphics[width=12cm]{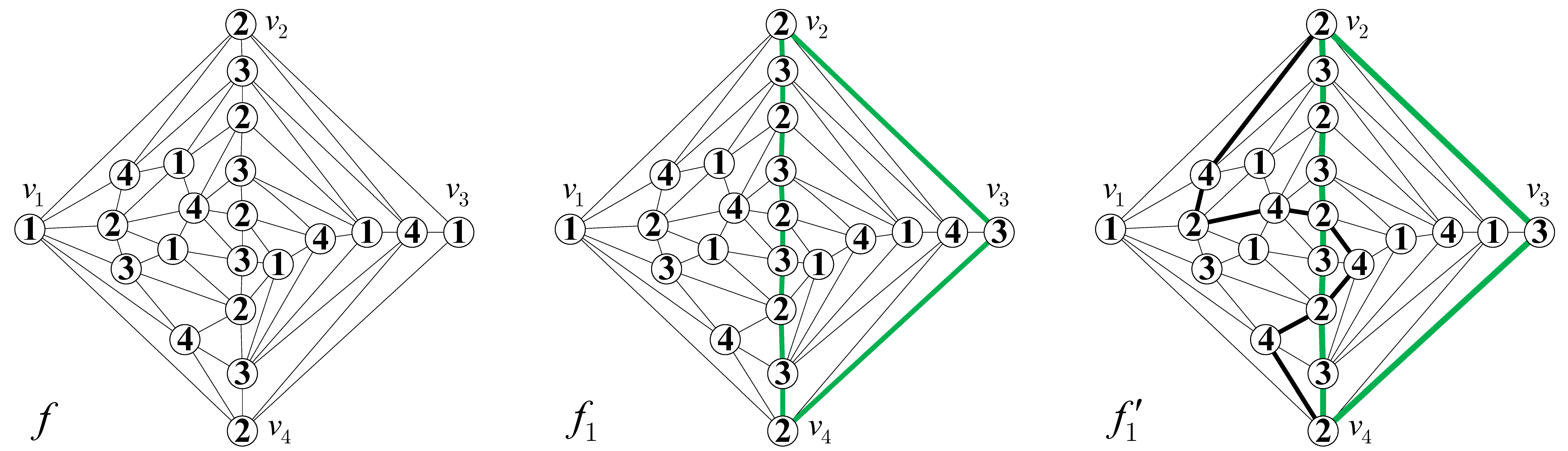}\\
  (a) \hspace{3.5cm} (b) \hspace{4cm} (c)
  \caption {Diagram illustration for non-module-cycle}\label{final4-6}
\end{figure}

\textbf{Remark 8} More generally, let $f$ be a module-coloring of $G^{C_4}$ with a module-path $\ell^{2i}$, where $i\in \{3,4\}$ and $C_4=v_1v_2v_3v_4v_1$.
Suppose that $f$ contains a bichromatic cycle $C$  that intersects with  $\ell^{2i}$. We call $C$ a module-cycle and $G^{C_4}$ a \emph{module-cycle 4-base-module}, if the following holds: conduct a $K$-change in the interior of $C$ and obtain a 4-coloring $f'_1$; if $f'_1$ contains a new bichromatic cycle,  then  conduct a $K$-change in the interior of the new bichromatic cycle and obtained a 4-coloring $f''_1$; we repeatedly carry out this process until the new 4-coloring contains no 23-module-path or 24-module-path.

The following result holds from the above analysis.

\begin{theorem}\label{thmf4-4}
Module-cycle 4-base-module contains a decycle  coloring.
\end{theorem}

\subsection{Parallel 4-base-modules}
Let $G^{C_4}$ be a 4-base-module and $f$ a module-coloring of $G^{C_4}$, where $C_4=v_1v_2v_3v_4v_1$, $f(v_1)=f(v_3)=1, f(v_2)=f(v_4)=2$, and under $f$ there exist unique 23-module-path $\ell^{23}$ and 24-module-path $\ell^{24}$ between $v_2$ and $v_4$. In this section, we consider the case that $d(v_1)=5$ and $d(v_2)=4$. Observe that every module-path  divide $G^{C_4}$ into two SMPGs. Denote by $H_1^{\ell^{23}}$ and $H_2^{\ell^{23}}$ the two SMPGs based on $\ell^{23}$, where the outer-cycle of them are $v_4v_1v_2\cup \ell^{23}$ and $v_4v_3v_2\cup \ell^{23}$, respectively. Similarly, based on $\ell^{24}$, we have $H_1^{\ell^{24}}$ and $H_2^{\ell^{24}}$.

We need to consider two cases in terms of $\ell^{23}$ and $\ell^{24}$. If exact one of $H_1^{\ell^{23}}$ and $H_2^{\ell^{23}}$ contains no vertex in $V(\ell^{24})\setminus \{v_2,v_4\}$ and the other contains all vertices of $\ell^{24}$, then   we say that $\ell^{23}$ and $\ell^{24}$ are \emph{parallel}, denoted by $\ell^{23}|| \ell^{24}$ (see  Figure \ref{newfig4-5} (a)); otherwise,  $\ell^{23}$ and $\ell^{24}$ are \emph{intersecting}, denoted by $\ell^{23}\nparallel  \ell^{24}$ (see  Figure \ref{newfig4-5} (b)).
Let $z_1=v_2, z_2, \ldots, z_t=v_4$ be the vertices in $V(\ell^{23})\cap V(\ell^{24})$. Denote by $P^i_{23}$ and $P^i_{24}$ the sub-paths of $\ell^{23}$ and $\ell^{24}$ from $z_i$ to $z_{i+1}$, respectively. Let $C_i=P^{i}_{23}\cup P^{i}_{24}$, and $H^{C_i}$ be the SMPG with outer-cycle $C_i$ (i.e., the subgraph of $G^{C_4}$ induced by vertices in $C_i$ and the interior of $C_i$), $i=1,2,\ldots, t-1$; see  Figure \ref{newfig4-5} (b). We refer to $H^{C_i}$ as a \emph{parallel block based on $\ell^{23}\nparallel  \ell^{24}$}. 
When $t=2$, $\ell^{23}$ and $\ell^{24}$ are parallel, which shows that the intersecting is a special case of the parallel.

 \begin{figure}[H]
  \centering
  \includegraphics[width=7cm]{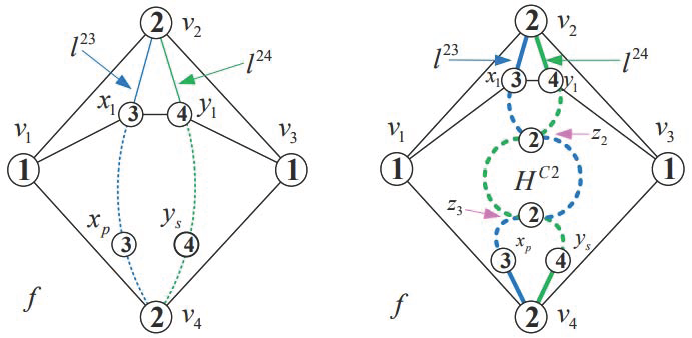}\\
  (a) \hspace{3.2cm} (b)
  \caption {Parallel module-paths and intersecting module-paths}\label{newfig4-5}
\end{figure}

\begin{theorem}\label{thm4-3}
Let $G^{C_4}$ be a 4-base-module, where $C_4=v_1v_2v_3v_4v_1$. Suppose that $G^{C_4}$ has a module-coloring $f$ under which there are parallel 23-module paths $\ell^{23}\in P^f_{23}(v_2,v_4)$ and 24-module path $\ell^{24}\in  P^f_{24}(v_2,v_4)$, where $f(v_2)=f(v_4)=2$ and $f(v_1)=f(v_3)=1$. If $d_{G^{C_4}}(v_2)=4$, then there exists a $f^*\in C_4^0(G^{C_4})$ such that $f^*(v_2)\neq f^*(v_4)$.
\end{theorem}

\begin{figure}[H]
  \centering
  \includegraphics[width=9cm]{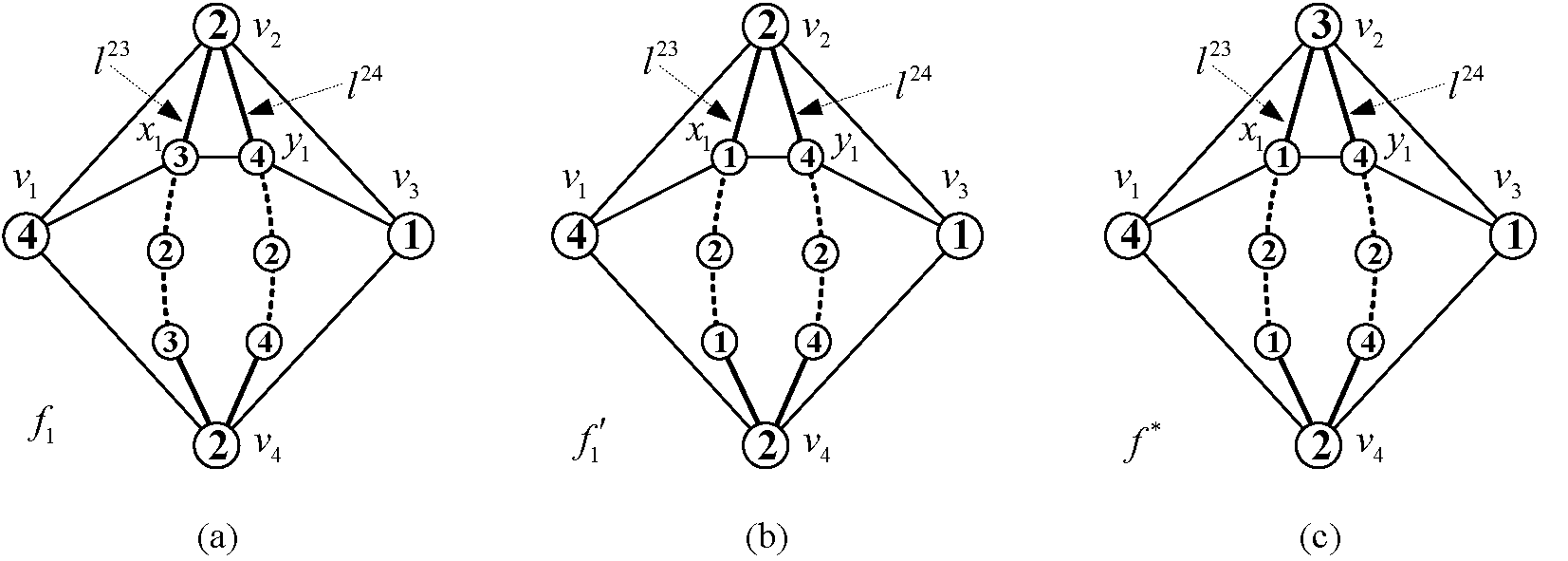}\\
  \caption {Illustration for the proof of Theorem \ref{thm4-3}}\label{figg4-5}
\end{figure}

\begin{proof}
By the definition of $\ell^{23}$ and $\ell^{24}$, $f$ contains at least two 13-components and two 14-components. Note that $G^{C_4}[N_{G^{C_4}}(v_2)]$ is a path, denoted by $P^{v_2}=v_1x_1y_1v_3$, where $f(x_1)=3$ and $f(y_1)=4$; see Figure \ref{newfig4-5} (a).

First, based on $f$, we carry out a $K$-change for the 14-component of $f$ containing $v_1$  and denote by $f_1$ the resulting 4-coloring. Clearly, $f_1$ contains a 24-cycle $C_{24}=\ell^{24}\cup v_4v_1v_2$; see Figure \ref{figg4-5} (a). Then, based on $f_1$, we carry out a $K$-change for the 13-component of $f_1$ in the interior of $C_{24}$ and denote by $f'_1$ the resulting 4-coloring. Clearly, $v_1x_1y_1v_3$ is a 14-path of $f'_1$; see Figure \ref{figg4-5} (b). Finally, based on $f'_1$, we change the color of $v_2$ from 2 to 3, and obtain our desired 4-coloring $f^*$ of $G^{C_4}$; see Figure \ref{figg4-5} (c). \qed
\end{proof}

According to Theorem \ref{thm4-3}, 4-base-modules that has two parallel module-paths are Kempe 4-base-modules.

\subsection{Module-path-related  graph and cycle-related graph}

\subsubsection{Module-path-related graph}

Let $G^{C_4}$ be a 4-base-module, where $C_4=v_1v_2v_3v_4v_1$. Let $f$ be a module-coloring of $G^{C_4}$ with a 24-module-path $\ell^{24}$ of length at least 4, where
$f(v_2)=f(v_4)=2$ and $f(v_1)=f(v_3)=1$. Let $u_1,u_2,\ldots, u_p, p\geq 3$ be the consecutive vertices in $\ell^{24}$ that are colored with 2  under $f$, where $u_1=v_2$, and $u_p=v_4$. Now, we construct a multigraph, denoted by  $H_2^f(\ell^{24})$, whose vertex set is $\{u_1,u_2,\ldots, u_p\}$, and two vertices $u_i$ and $u_j$ ($i\neq j, i,j\in \{1,2,\ldots,p\}$) are adjacent if and only if there is a 12-path (or 23-path)  from $u_i$ to $u_j$ which lies on one side of $\ell^{24}$.  More specifically, if there is a $12$-path  from $u_i$ to $u_j$ on the left (or right) side of $\ell^{24}$, then connect them by a fine solid line (or bold solid line); if there is a $23$-path  from $u_i$ to $u_j$ on the left (or right) side of $\ell^{24}$, then connect them by a fine dashed line (or bold dashed line). Here we note that when there are $k (>1)$  internal vertex disjoint 12-paths (or 23-paths)  from $u_i$ to $u_j$ which lies on one side of $\ell^{24}$, we will connect $u_i$ and $u_j$ with $k$ corresponding solid lines (dashed lines). Analogously, we can define $H_4^f(\ell^{24})$.  $H_2^f(\ell^{24})$ and $H_4^f(\ell^{24})$ are called \emph{24-module-path-related graph}, or simply \emph{$\ell^{24}$-related graph}.

In fact,  $H_4^f(\ell^{24})$  can be obtained from $H_2^f(\ell^{24})$, and vice versa. The proof is similar to that for cycle-related graph, which is shown as follows.

\begin{theorem} \label{thm4.6}
Suppose that $G^{C_4}$ is a tree-type 4-base-module,  and  $f$ is a module-coloring of $G^{C_4}$ such that $f(v_2)=f(v_4)=2$ and  $f(v_1)=f(v_3)=1$, where $C_4=v_1v_2v_3v_4v_1$. If  there are two intersecting module-paths from $v_2$ to $v_4$, 23-module-path $\ell^{23}$ and 24-module-path  $\ell^{24}$, then the following two statements hold.

(1) Conduct a $K$-change for the 13-component of $f$ that contains $v_3$ and denote by $f_1$ the resulting 4-coloring. Then, $f_1$ contains a 12-cycle or 23-cycle that intersects with a 24-module-path.

(2) Conduct a $K$-change for the 14-component of $f$ that contains $v_3$ and denote by $f_1$ the resulting 4-coloring. Then, $f_1$ contains a 12-cycle or 24-cycle that intersects with  a 23-module-path.
\end{theorem}

\begin{proof}
 Here we only give the proof of (1), and (2) can be proved by the same argument.
 Let $u_1,u_2,\ldots, u_p$ be the consecutive vertices in $\ell^{24}$  that are colored with 2 under $f$, where $p\geq 3, u_1=v_2,$ and $u_p=v_4$. Consider $H_2^f(\ell^{24})$. Since $G^{C_4}$ is a tree-type 4-base-module, $(G^{C_4})_{23}^f$ is a tree, and  $(G^{C_4})_{12}^f$ is connected and contains a 12-cycle. So, the subgraph of $H_2^f(\ell^{24})$ induced by solid lines is connected and contains cycle, and the subgraph of $H_2^f(\ell^{24})$ induced by dashed lines is connected. This implies that $|E(H_2^f(\ell^{24}))|\geq 2p-1$. Observe that $H_2^{f_1}(\ell^{24})$ can be  obtained from  $H_2^f(\ell^{24})$ by interchanging the bold solid lines and bold dashed lines (since $f_1$ is obtained from $f$ by conducting a $K$-change for the 13-component on the right side of $\ell^{24}$). Therefore,   $H_2^{f_1}(\ell^{24})$ must contain a cycle that is composed by either bold solid lines and fine solid lines, or bold dashed lines and fine dashed lines. \qed
\end{proof}

As an illustration of Theorem \ref{thm4.6}, we see that the 4-base-module $G^{C_4}$ shown in Figure \ref{figg4-6} (a), where  $f$ is the   module-coloring of $G^{C_4}$ and $f$ is a tree-coloring. Under $f$, $\ell^{24}$ contains five vertices $u_1, \ldots, u_5$ that are colored with 2. Since $(G^{C_4})_{23}^f$ is a tree, the subgraph of $H_2^f(\ell^{24})$ induced by dashed lines is a tree. Since $(G^{C_4})_{12}^f$ is connected and contains a 12-cycle, the subgraph of $H_2^f(\ell^{24})$ induced by solid lines is connected. Therefore, $H_2^{f}(\ell^{24})$  contains a cycle that consists of bold solid lines and fine solid lines (see Figure \ref{figg4-6} (b)). In $H_2^{f}(\ell^{24})$, if we interchange the bold solid lines and bold dashed lines, then we obtain graph $H_2^{f_1}(\ell^{24})$ (where $f_1$ is obtained from $f$ by conducting a $K$-change for the 13-component one the right side of $\ell^{24}$). Clearly, $H_2^{f_1}(\ell^{24})$ contains a cycle that consists of bold dashed lines and fine dashed lines (see Figure \ref{figg4-6} (c)), which corresponds to a 23-cycle of $f_1$ (see Figure \ref{figg4-6} (d)).

\begin{figure}[H]
  \centering
  \includegraphics[width=11cm]{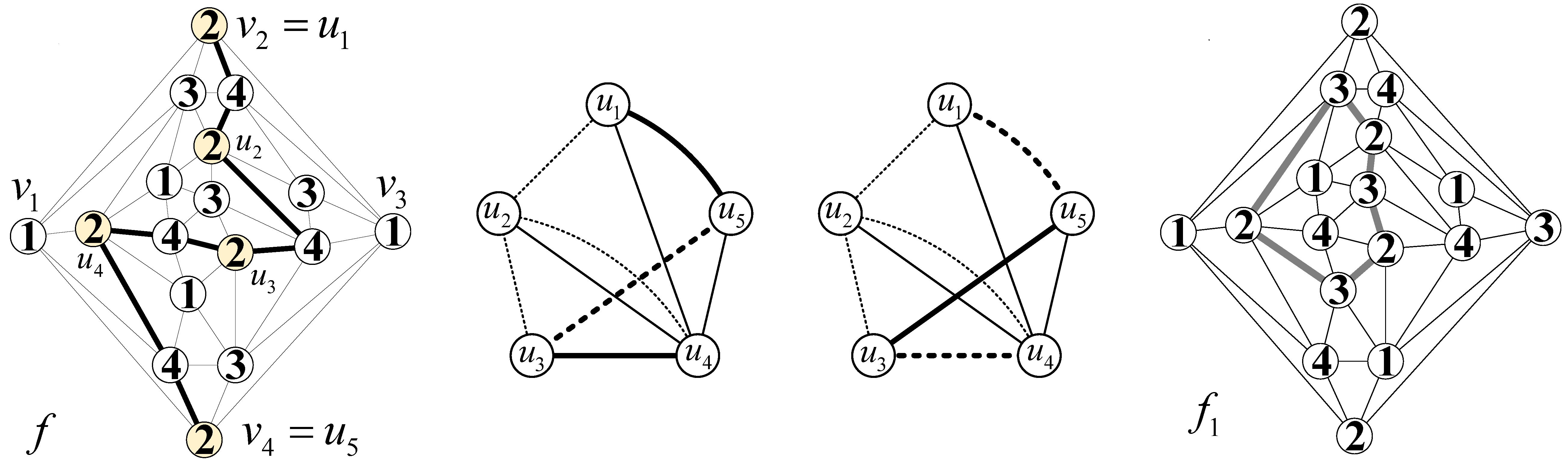}\\
  (a) \hspace{2cm}(b)\hspace{2cm}(c)\hspace{2cm}(d)
  \caption {Illustration for the proof of Theorem \ref{thm4.6}}\label{figg4-6}
\end{figure}

\subsubsection{Cycle-related graph}

Let $G^{C_4}$ be a 4-base-module, where $C_4=v_1v_2v_3v_4v_1$. Suppose that $G^{C_4}$  has a  4-coloring $f'$ (not necessarily  proper) which contains a 34-cycle $C'$ of length at least 4. Denote by $u_1,u_2,\ldots, u_p$  the consecutive vertices in $V(C')$ that are colored with 4 under $f'$, where $p\geq 2$. Now, we construct a multigraph, denoted by  $H_4^{f'}(C')$, whose vertex set is $\{u_1,u_2,\ldots, u_p\}$, and two vertices $u_i$ and $u_j$ ($i\neq j, i,j\in \{1,2,\ldots,p\}$) are adjacent if and only if there is a 14-path or 24-path  from $u_i$ to $u_j$ which lies in either the interior or exterior of $C'$.  More specifically, if there is a $14$-path  from $u_i$ to $u_j$ in the interior (or exterior)  of $C'$, then connect them by a bold solid line (or fine solid line);  if there is a $24$-path  from $u_i$ to $u_j$ in the interior (or exterior)  of $C'$, then connect them by a bold dashed line (or fine dashed line).
Analogously, we can define $H_3^{f'}(C')$ based on vertices in $V(C')$ that are colored with 3. $H_4^{f'}(C')$ and $H_3^{f'}(C')$ are called \emph{cycle-related graph of $C'$}.

To illustrate cycle-related graphs, we consider the 4-base-module $G^{C_4}$ shown in Figure \ref{fig4-11}, where  $f_{14}$ is a pseudo 4-coloring of $G^{C_4}$.   Under $f_{14}$, there are three bichromatic cycles: pseudo 14-cycle, 23-cycle, and 34-cycle of length 8 (denoted by $C'_{34}$). We use four colors red, green, yellow, and blue to mark vertices of $C'_{34}$, and the way how to mark these vertices is as shown in Figure \ref{fig4-11}. Observe that $(G^{C_4})_{24}^{f_{14}}$ is connected and contains no cycle, $(G^{C_4})_{14}^{f_{14}}$ has two components (one is in the interior of 23-cycle, and the other contains a pseudo 14-cycle), and $C'_{34}$ contains a vertex with color 4 that is in the interior of a 23-cycle. So, in $H_4^{f_{14}}(C'_{34})$, there are in total six edges, the subgraph induced by solid lines is a triangle and an isolated vertex, and the subgraph induced by dashed lines is a tree. The structure of $H_4^{f_{14}}(C'_{34})$ is shown in Figure \ref{fig4-11}, in which there is a solid line cycle, and the subgraph induced by bold lines (resp. solid lines) is a tree.  In the following, we prove that $H_3^{f_{14}}(C'_{34})$ can be derived from  $H_4^{f_{14}}(C'_{34})$, and then give a fast derivation method.

\begin{figure}[H]
  \centering
  \includegraphics[width=10cm]{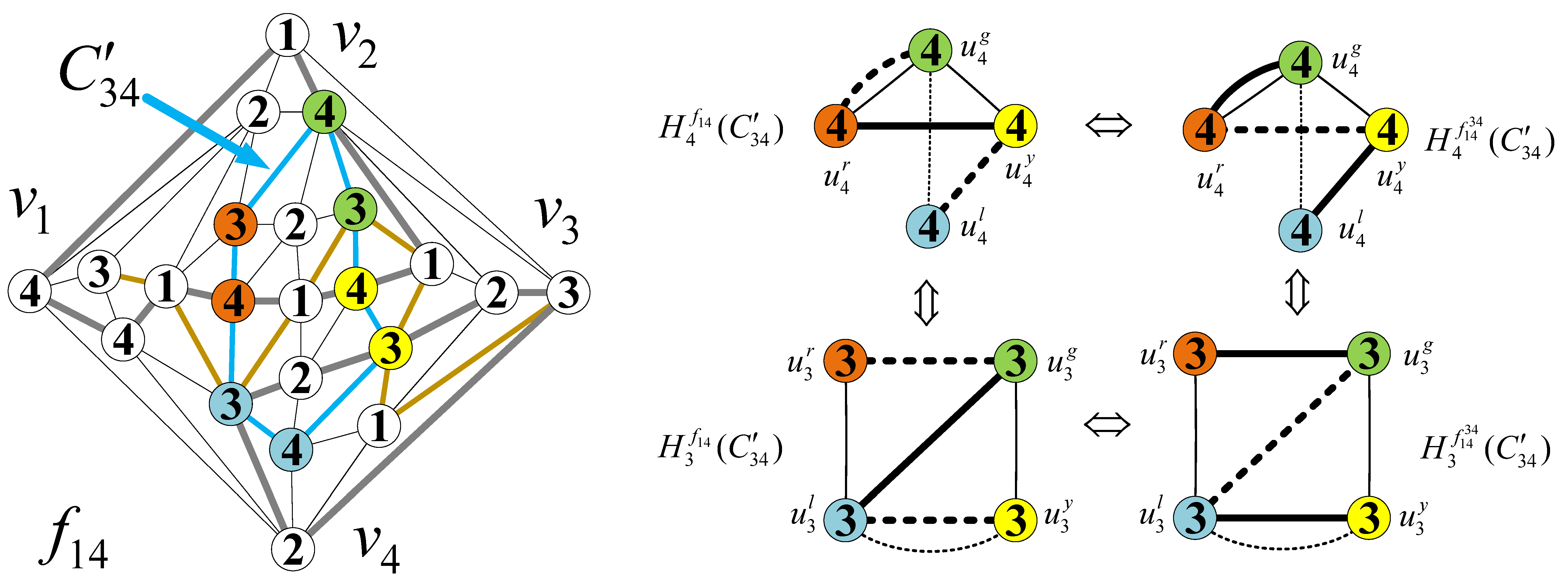}\\
  \caption {Illustration for cycle-related graph}\label{fig4-11}
\end{figure}

Determine the lines between $u_3^r$ (red vertex) and $u_3^g$ (green vertex) in $H_3^{f_{14}}(C'_{34})$. First, since in $H_4^{f_{14}}(C'_{34})$, there is a fine solid line between $u_4^r$(red vertex)  and $u^g_4$ (green vertex), there is a 14-path from  $u_4^r$ to $u^g_4$ in the exterior of $C'_{34}$. This implies that in the exterior of $C'_{34}$, no 23-path connects $u_3^r$ and $u_3^g$, i.e., in $H_3^{f_{14}}(C'_{34})$, no fine dashed line connects $u_3^r$ and $u_3^g$. Second, since in $H_4^{f_{14}}(C'_{34})$, there is a bold dashed line between $u_4^r$ and $u^g_4$,
there is a 24-path from  $u_4^r$ to $u^g_4$ in the interior of $C'_{34}$. This implies that in the interior of $C'_{34}$, no 13-path connects $u_3^r$ and $u_3^g$, i.e., in $H_3^{f_{14}}(C'_{34})$, no bold solid line connects $u_3^r$ and $u_3^g$. Third, since in $H_4^{f_{14}}(C'_{34})$, there is a fine dashed line between $u_4^g$ and $u^g_{\ell}$ (blue vertex),
there is a 24-path from  $u_4^r$ to $u^g_4$ in the exterior of $C'_{34}$. This implies that in the exterior of $C'_{34}$, no 13-path connects $u_3^r$ and $u_3^g$, i.e., in $H_3^{f_{14}}(C'_{34})$, no fine solid line connects $u_3^r$ and $u_3^g$. Based on the above analysis, we see that under $f_{14}$ there is a 23-path from $u_3^r$ to $u_3^g$  in the interior of $C'_{34}$, i.e., in $H_3^{f_{14}}(C'_{34})$,  there are only bold dashed lines between $u_3^r$ and $u_3^g$.

The above gives a method to determine the lines between $u_r^3$ and $u^g_3$ in $H_3^{f_{14}}(C'_{34})$, based on $H_4^{f_{14}}(C'_{34})$. Since under $f_{14}$ the vertex $u_4^g$  in $C'_{34}$  is adjacent to both $u_3^r$ and $u_3^g$, and in $H_4^{f_{14}}(C'_{34})$ the set of lines incident with  $u_4^g$
 includes fine solid lines, fine dashed lines, and bold dashed lines, it follows that in $H_3^{f_{14}}(C'_{34})$ there exist only bold dashed lines between $u_3^r$ and $u_3^g$.

Indeed, the above analysis gives a  method of constructing  $H_3^{f_{14}}(C'_{34})$ from $H_4^{f_{14}}(C'_{34})$.

{\bf Step 1}
In $C'_{34}$,  we first use a straight line (the gray line in Figure \ref{fig4-12}) connecting $u_i$ and $u_j$, by which  the vertices in $C'_{34}$ colored with 4 can be divided into two classes.
 Here we use the 4-base-module $G^{C_4}$  shown in Figure \ref{fig4-11} to illustrate this step. Under $f_{14}$, consider the lines between $\{u_3^r\}$ and $\{u_3^g, u_3^y\}$. In $C'_{34}$, use one straight line connecting  $u_3^r$ to $u_3^g$ (and $u_3^r$ to $u_3^y$); see  Figure \ref{fig4-12} (a) and (b). The corresponding straight lines in $H_4^{f_{14}}(C'_{34})$  are shown in the first two graphs in Figure   \ref{fig4-12} (c), and  other straight lines can be obtained  in a similar way (see the last four graphs in Figure   \ref{fig4-12} (c)).

\begin{figure}[H]
  \centering
  \includegraphics[width=10cm]{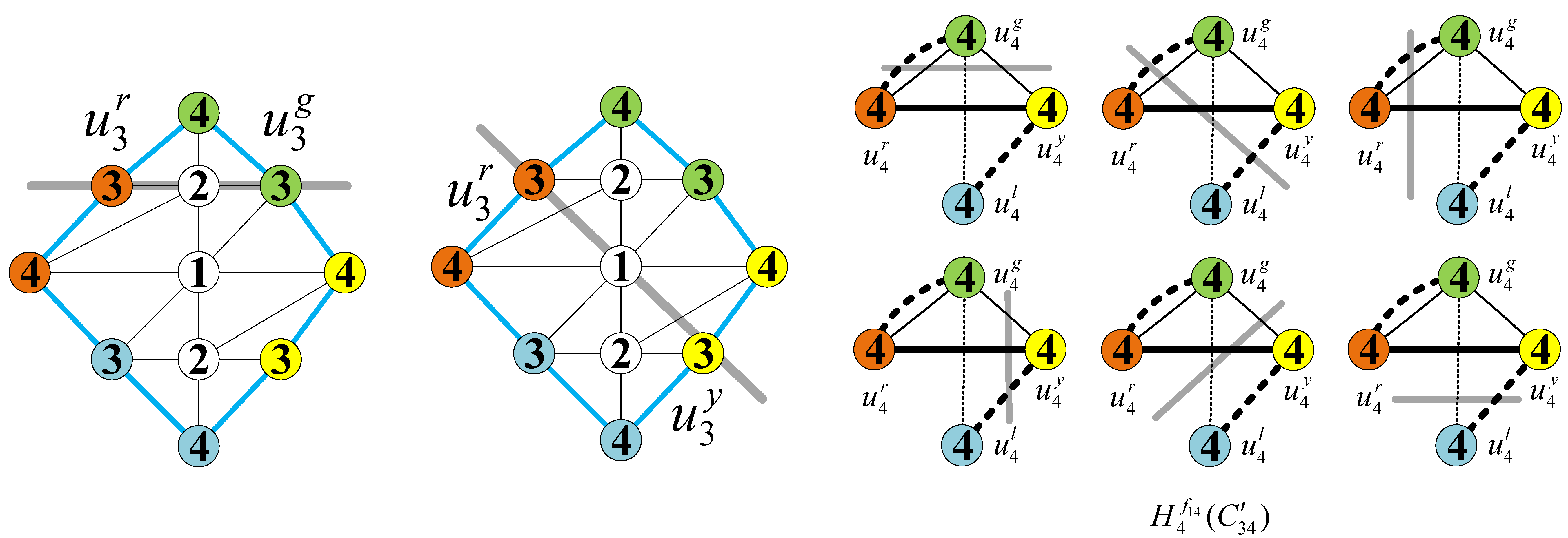}\\
  (a)\hspace{2cm}(b) \hspace{4cm} (c)
  \caption {Illustration for the construction of $H_3^{f_{14}}(C'_{34})$ from $H_4^{f_{14}}(C'_{34})$}\label{fig4-12}
\end{figure}

{\bf Step 2} Determine the set (denoted by $L$) of lines in $H_4^{f_{14}}(C'_{34})$ that are crossed by the straight lines in $H_4^{f_{14}}(C'_{34})$. Consider the first graph shown in Figure \ref{fig4-12} (c); we see that the $L$=\{fine solid lines, fine dashed lines, bold dashed lines\}.

{\bf Step 3} Determine the lines between $u_i$ and $u_j$.
\begin{itemize}
  \item  If $|L|=4$, then no line connects $u_i$ and $u_j$.
  \item If $|L|=3$, when fine dashed (or fine solid) line is not in $L$, connect $u_i$ and $u_j$ with a fine solid (or fine dashed) line; when bold dashed (or bold solid) line is not in $L$, connect $u_i$ and $u_j$ with a bold solid (or bold dashed) line.
  \item If $|L|=2$, when \{fine dashed line, fine solid line\} $\cap L =\emptyset$ (or  \{bold dashed line, bold solid line\} $\cap L =\emptyset$), connect $u_i$ and $u_j$ with  fine solid lines and fine dashed lines (or bold solid lines and bold dashed lines);  when \{fine dashed line, bold dashed line\} $\cap L =\emptyset$ (or  \{fine solid line, bold solid line\} $\cap L =\emptyset$), connect $u_i$ and $u_j$ with  fine solid lines and bold solid lines (or fine dashed lines and bold dashed lines); when \{fine dashed line, bold solid line\} $\cap L =\emptyset$ (or  \{fine solid line, bold dashed line\} $\cap L =\emptyset$), connect $u_i$ and $u_j$ with  fine solid lines and bold dashed lines (or fine dashed lines and bold solid lines).
  \item If $|L|=1$, when fine dashed (or solid) line  is  in $L$, connect $u_i$ and $u_j$ with  fine dashed (or solid) lines and bold (dashed and solid) lines; when bold dashed (or solid) line  is  in $L$, connect $u_i$ and $u_j$ with  bold dashed (or solid) lines and fine (dashed and solid) lines.
  \item If $|L|=0$, then connect $u_i$ and $u_j$ with  bold  (dashed and solid) lines and fine (dashed and solid) lines.
\end{itemize}
Based on Figure \ref{fig4-12} (c), there are six straight lines, by which we obtain the $H_3^{f_{14}}(C'_{34})$ shown in Figure \ref{fig4-11}, e.g., in $H_3^{f_{14}}(C'_{34})$, there is no line between $u_3^r$ and $u_3^y$, and a bold dashed line between $u_3^r$ and $u_3^g$.

By this way,  other lines in $H_3^{f_{14}}(C'_{34})$ can be determined similarly. So far, we prove that $H_3^{f_{14}}(C'_{34})$ can be constructed from $H_4^{f_{14}}(C'_{34})$, and give a fast construction method.

\subsection{Tree 4-base-module}\label{sec3-6}

\begin{theorem}\label{thm4-7}
 Suppose that $G^{C_4}$ is a 4-base-module with $\delta\geq 4$, $d_{G^{C_4}}(v_2)=4$, $d_{G^{C_4}}(v_1)\in \{5,6\}$, and $f$ is a module-coloring of $G^{C_4}$ for which $f(v_1)=f(v_3)=1$ and $f(v_2)=f(v_4)=2$, where $C_4=v_1v_2v_3v_4v_1$. If  $f$ is a tree-coloring and under $f$ there are two intersecting module-paths between $v_2$ and $v_4$, 23-module-path $\ell^{23}$ and 24-module-path $\ell^{24}$, then $G^{C_4}$ has a 4-coloring  $f^*\in C_4^0(G^{C_4})$ satisfying
\begin{equation}\label{equ4.1}
f^*(v_2)\neq f^*(v_4)
\end{equation}
\end{theorem}

\begin{proof}
Since $f$ is a tree-coloring, we have that $|P_{23}^f(v_2,v_4)|=1$ and $|P_{24}^f(v_2,v_4)|=1$, i.e., $P_{23}^f(v_2,v_4)=\{\ell^{23}\}$ and $P_{24}^f(v_2,v_4)=\{\ell^{24}\}$. Therefore, both $G_{13}^f$ and $G_{14}^f$ has two components. Let $P^{v_2}=G^{C_4}[N_{G^{C_4}}(v_2)]$. Then $P^{v_2}$ is a path of length 3. Let $P^{v_2}=v_1x_1y_1v_3$. Without loss of generality, we assume $f(x_1)=3$ and $f(y_1)=4$ (see Figure \ref{newfig4-5} (b)).

In $H_2^{\ell^{24}}$,  conduct a $K$-change for the 13-component containing $v_3$, and denote by $f_1$ the resulting coloring. Clearly, the following relations hold:
\begin{equation} \label{equ4.2}
\left\{ \begin{array}{lll}
(G^{C_4})_{12}^{f_1} & =& (H_1^{\ell^{24}})_{12}^{f} \cup (H_2^{\ell^{24}})_{23}^{f}\\
(G^{C_4})_{14}^{f_1} & =& (H_1^{\ell^{24}})_{14}^{f} \cup (H_2^{\ell^{24}})_{34}^{f}\\
(G^{C_4})_{23}^{f_1} & =& (H_1^{\ell^{24}})_{23}^{f} \cup (H_2^{\ell^{24}})_{12}^{f}\\
(G^{C_4})_{34}^{f_1} & =& (H_1^{\ell^{24}})_{34}^{f} \cup (H_2^{\ell^{24}})_{14}^{f}\\
(G^{C_4})_{13}^{f_1} & =& (G^{C_4})_{13}^{f}\\
(G^{C_4})_{24}^{f_1} & =& (G^{C_4})_{24}^{f}
\end{array}\right.
\end{equation}

For $f_1$, one of the following two cases will happen. One case is that there is a 23-path from  $z_2$ to $v_4$, and we call $f_1$  a \emph{23-vertical-axis} coloring; the other case is that there is an 14-path from  $v_1$ to $y_1$, and we call $f$  an \emph{14-sloped-axis} coloring. In what follows, we analyze these two cases in detail.

(1) $f_1$ is a 23-vertical-axis coloring

Based on Equation (\ref{equ4.2}),  a 23-vertical-axis coloring $f_1$ can be constructed as follows. Since $f_1$ is a 23-vertical-axis coloring, there is a 23-path of $f_1$ from  $z_2$ to $v_4$. So, under $f$, there are some 23-paths in $H_1^{\ell^{24}}$  and 12-paths in $H_2^{\ell^{24}}$ such that under $f_1$
the union of these paths is a 23-path from  $z_2$ to $v_4$,  denoted by $P^{z_2v_4}$.  Let $\ell_1^{23}=P_{23}^{z_2v_4} \cup P_{23}^{1}$. We call $\ell_1^{23}$ a \emph{23-vertical-axis} based on $f_1$, and  $G^{C_4}$  a \emph{vertical-axis 4-base-module}.

Note that $P_{23}^{z_2v_4}$ can be divided into two classes: one contains $v_3$ and the other does not contain $v_3$. Please see Figure \ref{fig4-13}(b) for the illustration of structures of these two classes of 23-vertical-axis under $f_1$ (marked by black bold lines).

Now, under $f_1$, conduct a $K$-change for the 14-component containing $v_1$ and denote by $f'_1$ the resulting coloring (see Figure \ref{fig4-13}(c)). Based on $f'_1$, change the color of $v_2$ from $2$ to $1$, and the resulting coloring (still denoted by $f^*$) satisfies Equation (\ref{equ4.1}).

\begin{figure}[H]
  \centering
  \includegraphics[width=10cm]{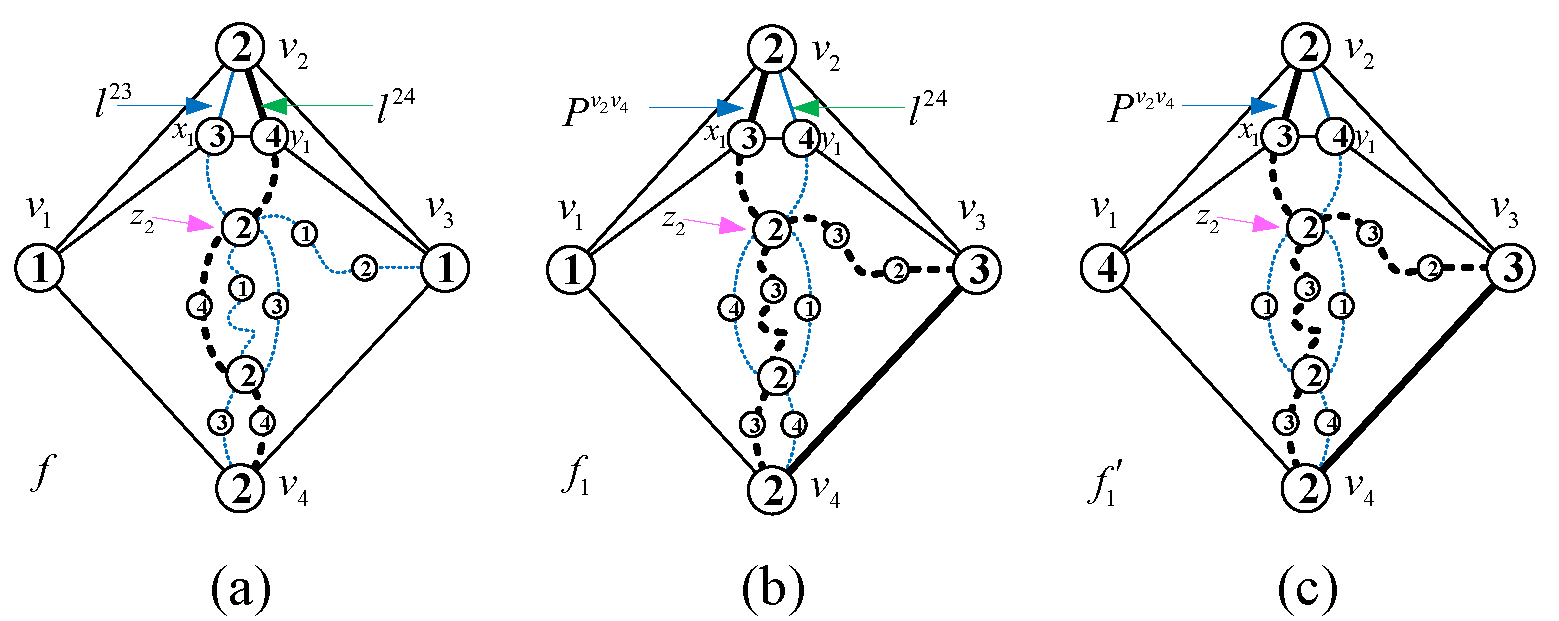}\\
  \caption {Illustration for structure of the 23-vertical-axis coloring}\label{fig4-13}
\end{figure}

\begin{figure}[H]
  \centering
  \includegraphics[width=10cm]{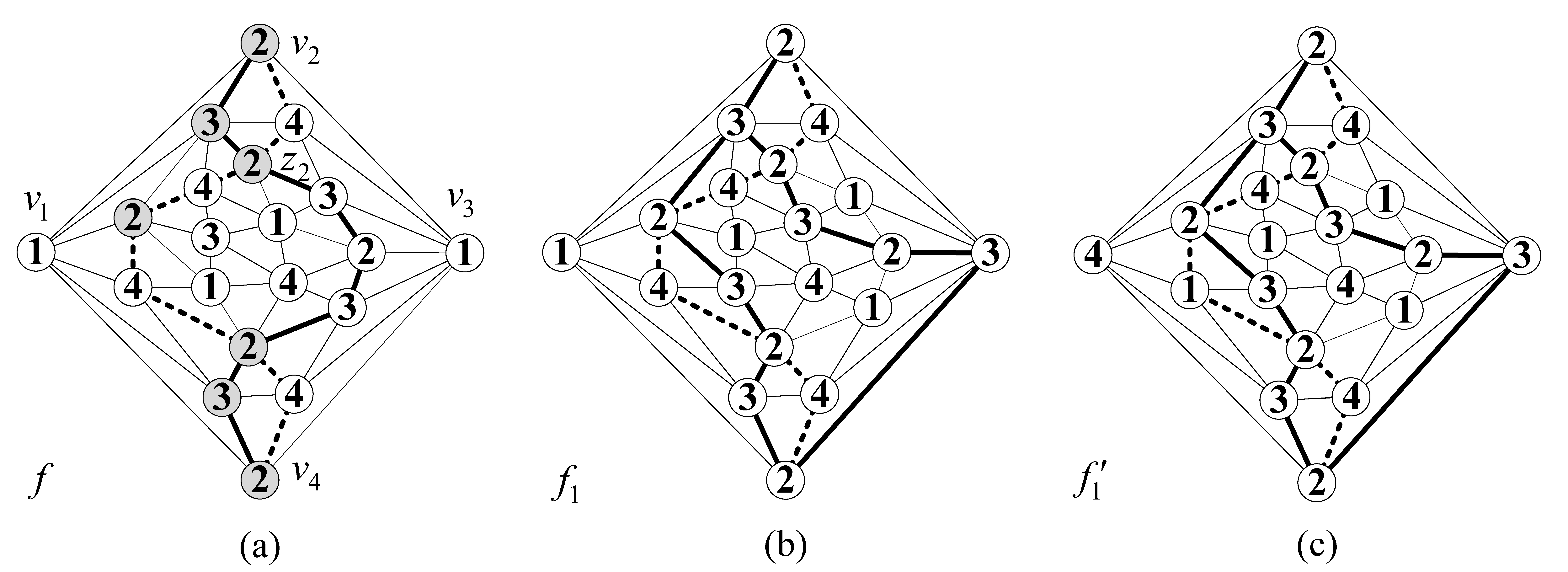}\\
  \caption {An example of 23-vertical-axis 4-base-modules}\label{fig4-14}
\end{figure}

Figure \ref{fig4-14} gives an instance of 23-vertical-axis 4-base-module. Denote by $G^{C_4}$ and $f$ the 4-base-module and the module-coloring shown in Figure \ref{fig4-14}(a), respectively. Under $f$, $H_1^{\ell^{24}}$ has two 23-components,  whose vertices are marked by grey color;  $H_1^{\ell^{24}}$ has two 12-paths,  the union of which makes up two 23-vertical-axes under $f_1$ (see the bold lines in Figure \ref{fig4-14} (b)). Under $f_1$,  conduct  a $K$-change for the 14-component containing $v_1$ and denote by $f'_1$ the resulting coloring (see Figure \ref{fig4-14}(c)). Based on $f'_1$, change the color of $v_2$ from $2$ to $1$, and  obtain the decycle coloring $f^*$.

(2) $f_1$ is a 14-sloped-axis  coloring

Since $f_1$ is an 14-sloped-axis  coloring,  $f_1$ contains an 14-path $P^{v_1y_1}$ from $v_1$ to $y_1$. So,  $H_1^{\ell^{24}}$ has one or more 14-paths, $H_2^{\ell^{24}}$ has one or more  34-paths, and the union of these paths is a path $P^{v_1y_1}$ from $v_1$ to $y_1$ (Note that under $f_1$, $P^{v_1y_1}$ is a 14-path. So, it is also denoted by $P^{v_1y_1}_{14}$). Colorings $f$ and $f_1$ are shown in Figure \ref{fig4-15} (a) and (b), respectively. Based on $f_1$, change the color of $v_1$ from 1 to 4 and change the color of $v_2$ from 2 to 1, and denote by $f_2$ the resulting pseudo 4-coloring. Under $f_2$, let $C_{14}^{f_2}=P_{14}^{v_1y_1}\cup y_1v_2v_1$. Clearly,  $C_{14}^{f_2}$ is a pseudo 14-cycle (denoted by $C_{14}$ in the following); see Figure \ref{fig4-15} (c).

\begin{figure}[H]
  \centering
  \includegraphics[width=10cm]{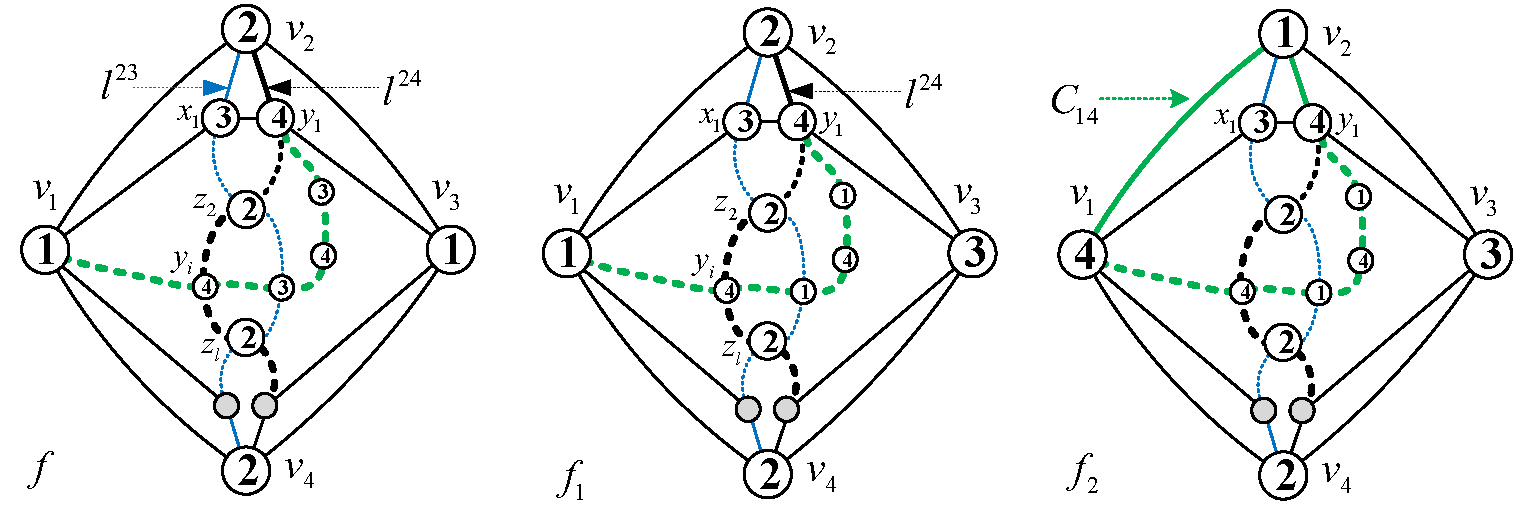}\\
  (a)\hspace{3cm}(b) \hspace{3cm} (c)
  \caption {Illustration for structure of the 14-sloped-axis coloring}\label{fig4-15}
\end{figure}

If $d_{G^{C_4}}(v_1)=5$, then under $f_2$ there is exact one pseudo 44-edge, and there are two situations in terms of $f$  by which $G^{C_4}$ can be divided into two classes correspondingly: 55-324 type and 55-343 type. Figure \ref{fig4-16} gives an illustration for these cases (under module-coloring $f$). In the following, we will prove that these two classes of 4-base-modules have decycle colorings.

\begin{figure}[H]
  \centering
  \includegraphics[width=6cm]{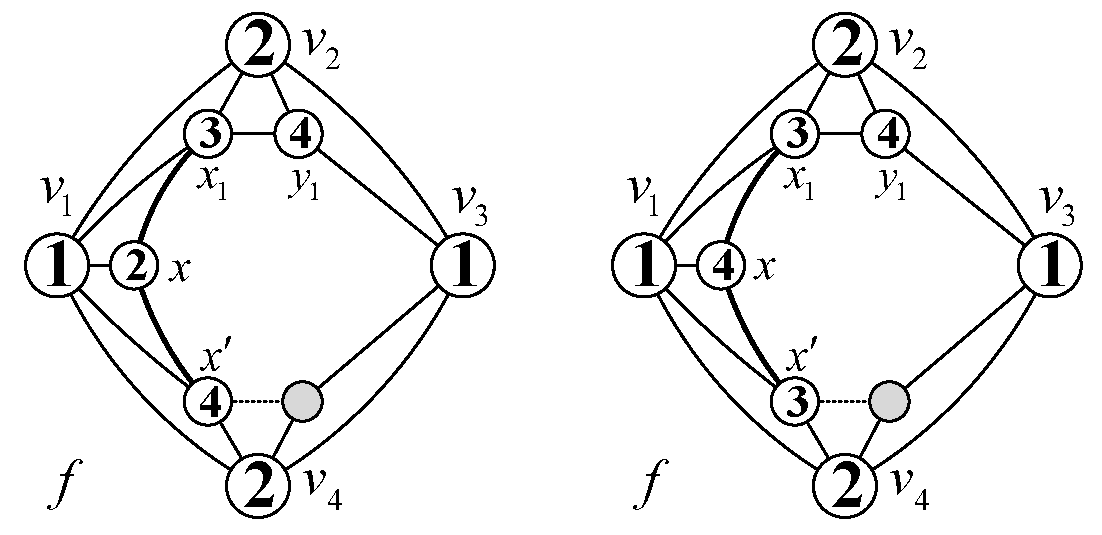}\\
  (a)55-324 type\hspace{2cm}(b) 55-343 type
  \caption {Illustration for two classes of 4-base-modules}\label{fig4-16}
\end{figure}

{\bf Claim 1.} \textbf{55-324 type 4-base-modules are decyclizable}.

{\bf Proof of Claim 1.} The graph shown in Figure \ref{fig4-16}(a) is a 55-324 type 4-base-module. The graph shown in Figure \ref{fig4-17} (a) gives its structure under $f_2$, in which there is a unique pseudo 44-edge $v_1x'$. Observe that $v_1x'\in E(C_{14})$ and $v_1x'$ is 22-type. If $(G^{C_4})_{34}^{f_2}$ contains no odd-length cycle, then $(G^{C_4})_{34}^{f_2}-v_1x'$ is disconnected. Then,
the pseudo 44-edge can be eliminated by conducting a $K$-change for  the 34-component of $(G^{C_4})_{34}^{f_2}-v_1x'$ containing $v_1$, and the proof is completed \footnote{Note that when a pseudo $ii$-edge $z_1z_2$ is not on an odd-length $ij$-cycle, we can eliminate this pseudo edge in the following way. Let $H$ be the $ij$-component containing $z_1z_2$. We conduct a $K$-change for the $ij$-component of $H-z_1$ containing $z_2$. For convenience, we refer to such a $K$-change as a \emph{pseudo $ii$-edge $K$-change for $H$}.}
 So, in what follows, we assume that under $f_2$, the pseudo 44-edge $v_1x'$ is on an odd-length 34-cycle $C_{34}$.  We further consider two cases.

{\bf Case 1.} Under $f_2$, $C_{34}$ and $C_{14}$ share a common pseudo 44-edge, and  $C_{34}$ and $C_{14}$ are nonintersecting; see Figure \ref{fig4-17} (b). For this case, we can obtain a decycle coloring by the following way. First, conduct a $K$-change for the 12-component in the interior of $C_{34}$ and denote by $f_{34}$ the resulting coloring. Since both $v_1x'$ and $v_1x_1$ are edges of $C_{34}$, we have $f_{34}(x)=1$; see Figure \ref{fig4-17} (c). Furthermore, under $f_{34}$, conduct a $K$-change for the 23-component in the interior of $C_{14}$   and denote by $f_{34}^{14}$ the resulting pseudo 4-coloring; see Figure \ref{fig4-17} (d). Observe that $f_{34}^{14}(x)=1$ and $f_{34}^{14} (x_1)=2$. Therefore, under $f_{34}^{14}$, we can obtain a decycle coloring $f^*$ by changing the color of $v_1$ from $4$ to $3$.

\begin{figure}[H]
  \centering
  \includegraphics[width=10cm]{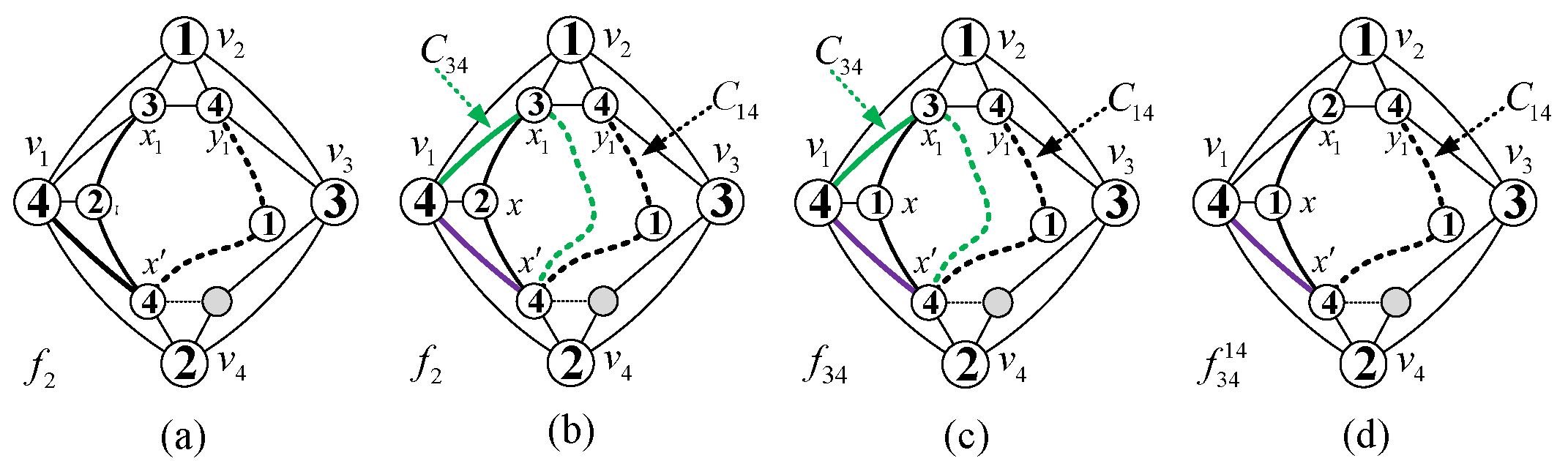}\\
  \caption {Illustration for the proof of 55-324 type 4-base-modules when $C_{14}$ contains $C_{34}$ under $f_2$}\label{fig4-17}
\end{figure}

Figure \ref{fig4-18} gives an example of  4-base-modules for case 1, in which the module-coloring $f$ of $G^{C_4}$ is not displayed. By conducting a $K$-change for the components in the interior of $C_{34}$ we obtain a new pseudo 4-coloring $f_{34}$. Clearly, $f_{34}$ contains $C_{14}$. Then, under $f_{34}$ conduct a $K$-change for the 23-component in the interior of $C_{14}$ and denote by $f_{34}^{14}$ the resulting pseudo 4-coloring. Observe that under $f_{34}^{14}$, no neighbor of $v_1$ is colored with color 3. Therefore, we can obtain a decycle coloring  $f^*$ by changing the color of $v_1$ from $4$ to $3$.

\begin{figure}[H]
  \centering
  \includegraphics[width=8cm]{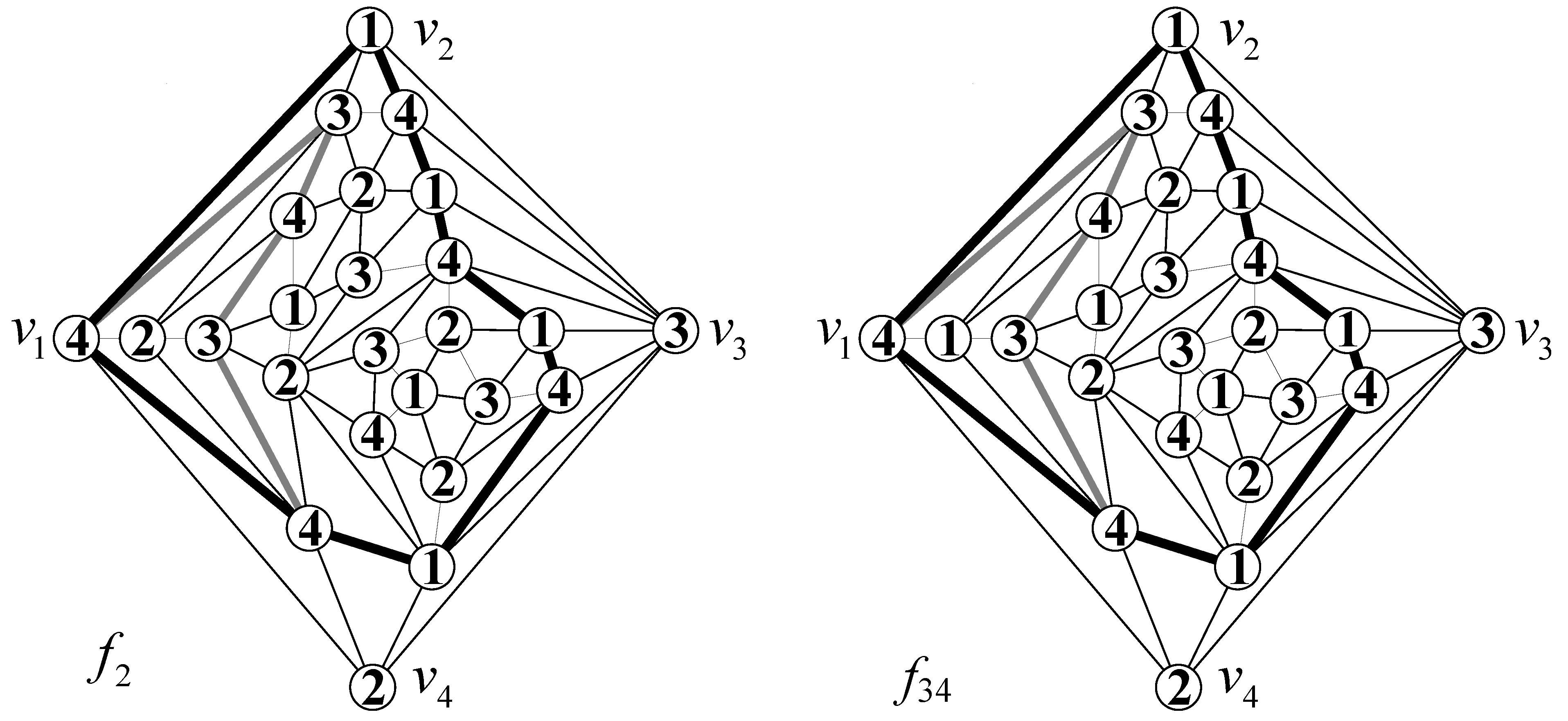}\\
  \caption {An example of 4-base-modules for case 1}\label{fig4-18}
\end{figure}

{\bf Case 2.} Under $f_2$, $C_{14}$ and $C_{34}$ are intersecting. By Theorem \ref{thm4.6},  $f_2$ contains either a 23-cycle $C_{23}$ or an 12-cycle $C_{12}$. So, $C^2(f_2)$ contains at least three bichromatic cycles such that one of them is a proper bichromatic cycle, and two of them are pseudo bichromatic  cycles $C_{14}$ and $C_{34}$. Without loss of generality, we assume the proper bichromatic cycle is  $C_{23}$ (the case for $C_{12}$ is similar). Then, $C_{23}$ is a non-module-cycle of $f_1$ (otherwise, $G^{C_4}$ is a module-cycle 4-base-module). In other words,  under $f_2$, there are 21-path, 31-path, 24-path, and 34-path in the interior of $C_{23}$.
Now, we give the detailed steps to construct a decycle coloring.

{\bf Step 1. $C_{23}$ is a good-cycle}

 Consider $C_{23}$; conduct a $K$-change for the 14-component in the interior (or exterior) of $C_{23}$  and denote by $f_{23}$ the new coloring. If $f_{23}$ contains no odd-length 34-cycle, then the pseudo 44-edge can be eliminated by conducting a pseudo 44-edge $K$-change for the 34-component of $f_{23}$, and we call $C_{23}$ a \emph{good-cycle} \footnote{In general, let $f'$ be a pseudo 4-coloring, under which the pseudo edge is on the pseudo $it$-cycle $C'$. If there is a bichromatic cycle $C$ that intersects with $C'$, then a new pseudo 4-coloring $f''$ can be obtained by conducting a $K$-change for the components in the interior (or exterior) of $C$. If  $f''$ contains no odd-length $it$-cycle, then we  call $C$ a good-cycle (which can eliminate the pseudo $it$-edge), $i,t\in \{1,2,3,4\}$.}. As an example for good-cycles, we see the 4-base-module (denoted by $G^{C_4}$) shown  Figure \ref{fig4-19} (a). Under $f_2$, $(G^{C_4})_{34}^{f_2}$ is connected and contains the cycle $C_{34}$, and  $(G^{C_4})_{13}^{f_2}$ is connected; see Figure \ref{fig4-19} (b). So, we have the cycle-related graph $H_3^{f_2}(C_{23})$ of $C_{23}$ and the resulting graph $\overline{H}_3^{f_2}(C_{23})$  by interchanging the bold  lines in $H_3^{f_2}(C_{23})$; see Figure \ref{fig4-19} (c). Since there is no cycle consisting of dashed lines in  $H_3^{f_2}(C_{23})$, $f_{23}$ contains no odd-length 34-cycle, and so $C_{23}$ is a good-cycle. In the  cycle-related graph $H_3^{f_2}(C_{23})$, dashed lines represent  34-paths and solid lines represent 13-paths. Obviously, if  $C_{23}$ is a good-cycle, then $C_{23}$ and $C_{34}$ are intersecting.

If $C_{23}$ is not a good-cycle, then switch to Step 2.

\begin{figure}[H]
  \centering
  \includegraphics[width=8cm]{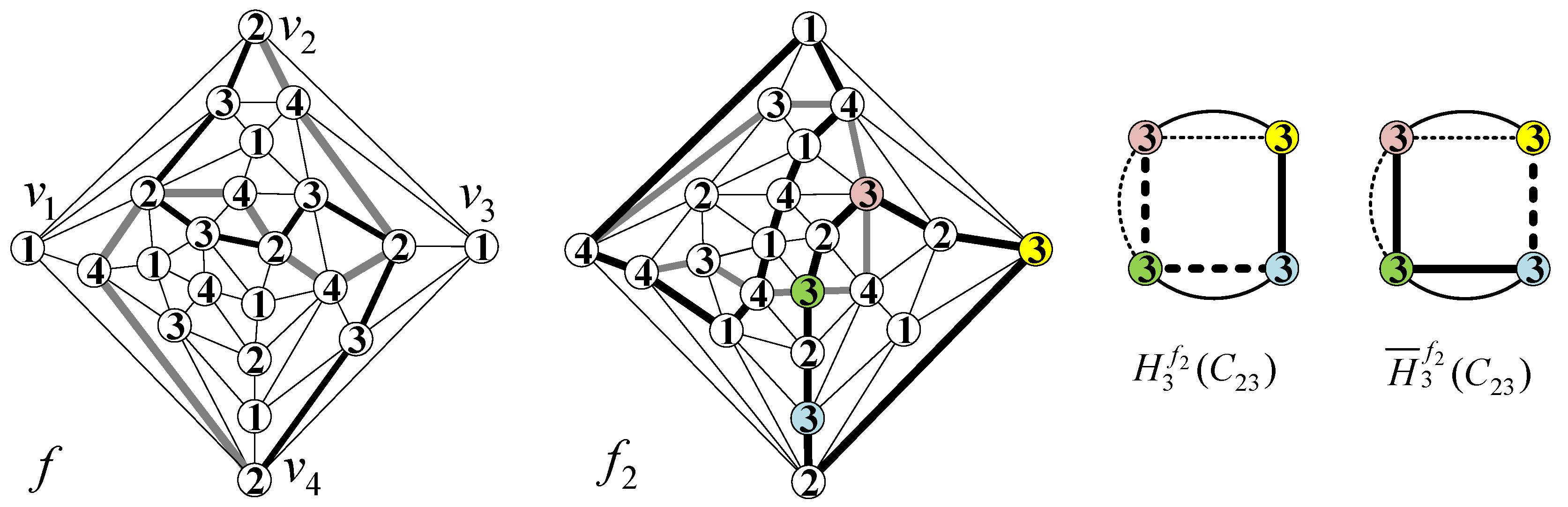}\\
  (a) \hspace{2cm}(b)\hspace{2cm}(c)
  \caption {An example of 4-base-module such that $C_{23}$ is a good-cycle, and its  cycle-related graph}\label{fig4-19}
\end{figure}

{\bf Step 2. $f_{14}$ or $f_{14}^{23}$ contains a 24-big-cycle}

A \emph{24-big-cycle} is a special good-cycle which contains vertices $v_1$ and $x_1$. In the following, we discuss the case that $f_{14}$ or $f_{14}^{23}$ contains a 24-big-cycle.

Consider $C_{14}$; conduct a $K$-change for the 23-component in the interior (or exterior) of $C_{14}$   and denote by $f_{14}$ the new pseudo 4-coloring. We assert that under $f_{14}$, there is a 24-cycle or 34-cycle that intersects with $C_{14}$.
Let $u_1,u_2,\ldots, u_{p}, p\geq 2$ be the sequence of vertices in $V(C_{14})$ that are colored with 4 under $f_2$, where  $u_1=y_1$ and $u_p$ denotes both $v_1$ and $x'$ (Observe that $v_1x'$ is a pseudo 44-edge of $f_2$; here we use one vertex $u_p$ to represent $v_1$ and $x'$). Now, consider the cycle-related graph $H_4^{f_2}(C_{14})$ with vertex set  $\{u_1,u_2,\ldots, u_{p}\}$ and a dashed (or solid) line between $u_i$ and $u_j$ if and only if there is a 34-path (or 24-path). Since the subgraph induced by vertices colored with 3 or 4 under $f_2$ is connected and contains a pseudo 34-cycle $C_{34}$, it follows that in $H_4^{f_2}(C_{14})$ the subgraph induced by the set of dashed lines is a spanning subgraph of  $H_4^{f_2}(C_{14})$ and contains a cycle, i.e., this subgraph contains $p$ lines. In addition, the  subgraph induced by vertices colored with 2 or 4 under $f_2$ is  a tree. So,  in $H_4^{f_2}(C_{14})$ the  subgraph induced by the set of solid lines is a
tree, and hence $|E(H_4^{f_2}(C_{14}))|=2p-1$. In $H_4^{f_2}(C_{14})$, if we interchange the bold dashed lines and bold solid lines, then there is a cycle consisting of solid lines (or dashed lines). This  shows that $f_{14}$ must contains a 24-cycle $C'_{24}$ or  a 34-cycle $C'_{34}$ that intersects with $C_{14}$.

If $C'_{24}$ or  $C'_{34}$ is a good-cycle, then the result holds.

In the following, we further discuss the case that $C'_{24}$ is a 24-big-cycle.  One example 24-big-cycle  satisfies the following conditions: $C'_{24}$ contains $y_1$; there is a 24-path $P_{24}^{y_1u}$  from $y_1$ to $u$ (in the exterior of $C_{14}$), where  $u$ and $u'$ are the endpoints of an 12-path or 24-path on $C_{23}$; in the exterior of $C_{14}$, there is a 24-path between $u'$ and $v_4$; see Figure \ref{fig4-20} (a). So, when there is a 24-path between $u$ and $u'$ in the interior of $C_{23}$, $f_{14}$ contains a 24-big-cycle $C'_{24}$, i.e.,

 \begin{equation} \label{equ4.3}
 C'_{24}=v_4v_1\cup v_1x_1 \cup x_1y_1 \cup P_{24}^{y_1u} \cup P_{24}^{uu'} \cup P_{24}^{u'v_4}
 \end{equation}

\noindent when  there is an 12-path between $u$ and $u'$ in the interior of $C_{23}$, $f_{14}^{23}$ contains a 24-cycle $C'_{24}$, i.e.,

 \begin{equation} \label{equ4.4}
 C'_{24}=v_4v_1\cup v_1x_1 \cup x_1y_1 \cup P_{24}^{y_1u} \cup P_{24}^{uu'} \cup P_{24}^{u'v_4}
 \end{equation}

\begin{figure}[H]
  \centering
  \includegraphics[width=9cm]{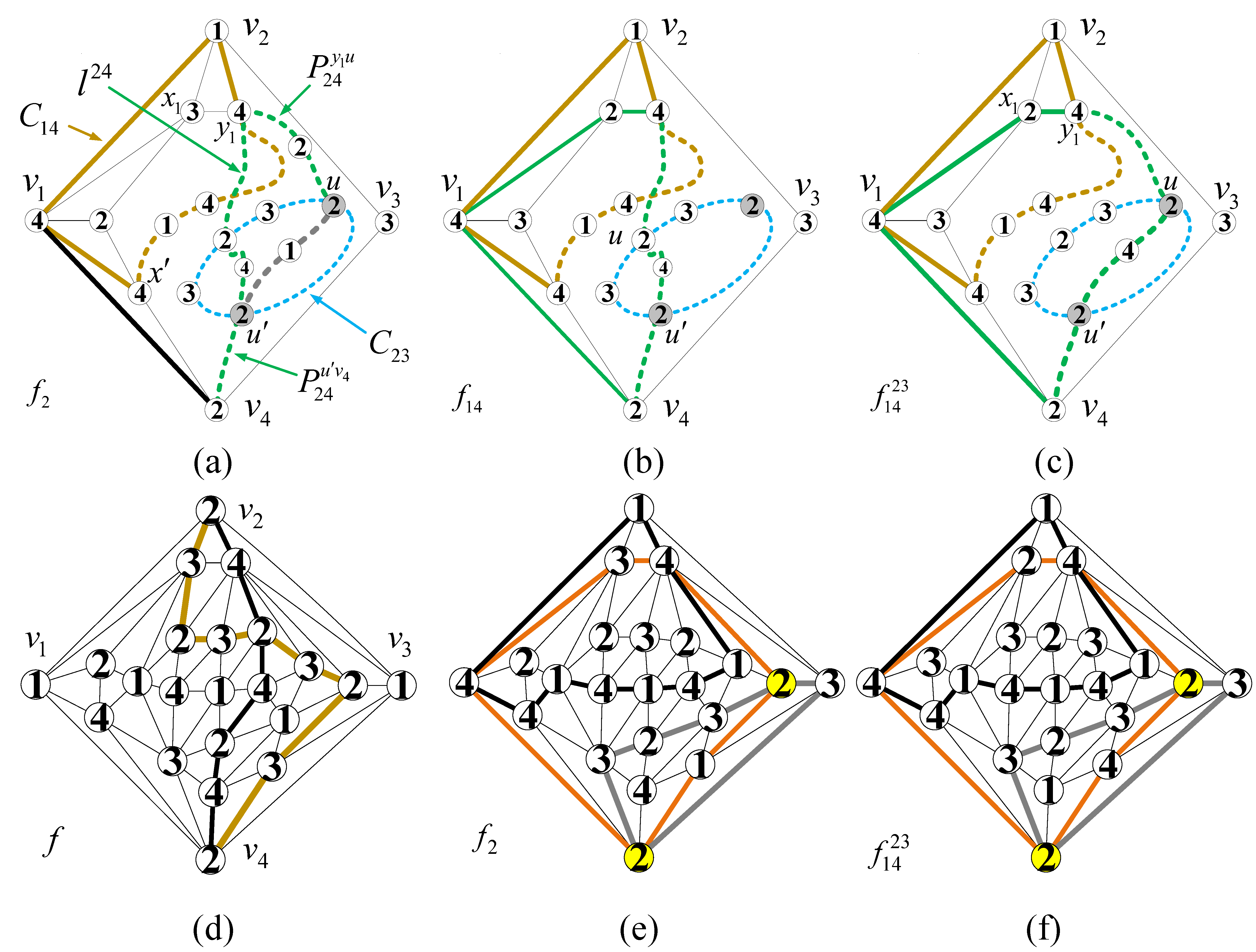}\\
  \caption {The 24-big-cycle $C'_{24}$ in $f_{14}$ or $f_{14}^{23}$ is a good-cycle}\label{fig4-20}
\end{figure}

We now showcase how to obtain a decycle coloring based on  $C'_{24}$. Under $f_2$, conduct a $K$-change  in the interior (or exterior) of $C_{14}$ and obtain a new coloring $f_{14}$. If $f_{14}$ contains a 24-big-cycle $C'_{24}$ (see Figure \ref{fig4-20} (b)), then conduct a $K$-change  in the interior of $C'_{24}$  and obtain a new coloring $f_{14}^{24}$. Under $f_{14}^{24}$, we change the color of $v_1$ from 4 to 3 and obtain  a  decycle coloring $f^*$. For the case that $C'_{24}$ is constructed under $f_{14}^{23}$ by Equation (\ref{equ4.4}), the proof is analogous; see Figure \ref{fig4-20} (c). Figures \ref{fig4-20} (d),(e), and (f) give an example for this case, where $C'_{24}$ belongs to the case in Equation (\ref{equ4.4}).

In the above, we have discussed the case that the given 24-big-cycle $C'_{24}$ involves only one 34-path $v_1x_1y_1$ in the interior of $C_{14}$. In what follows, we consider the case that under $f_2$, the 34-path consists of more than one 34-paths in the interior of $C_{14}$. Since $C'_{24}$ contains vertex $v_1$, we can view $v_1$ as one endpoint of the 34-path in the interior of $C_{14}$, and denote by $y'_1$ the other endpoint of the 34-path; see Figure \ref{fig4-21}(a), in which the 34-path is denoted by green dashed lines. Under $f_2$, conduct a $K$-change in the interior of $C_{14}$ and denote by $f_{14}$ the resulting coloring (see Figure \ref{fig4-21}(b)). Clearly, under $f_{14}$, the path denoted by green dashed lines becomes a 24-path. Figure \ref{fig4-21}(c) gives an example for this case, in which only pseudo coloring $f_2$ is given while  the labels of vertices, the coloring $f$, and other  colorings are omitted.  Based on $f_2$, it is easy to show that this example belongs to the situation we considered.

\begin{figure}[H]
  \centering
  \includegraphics[width=10cm]{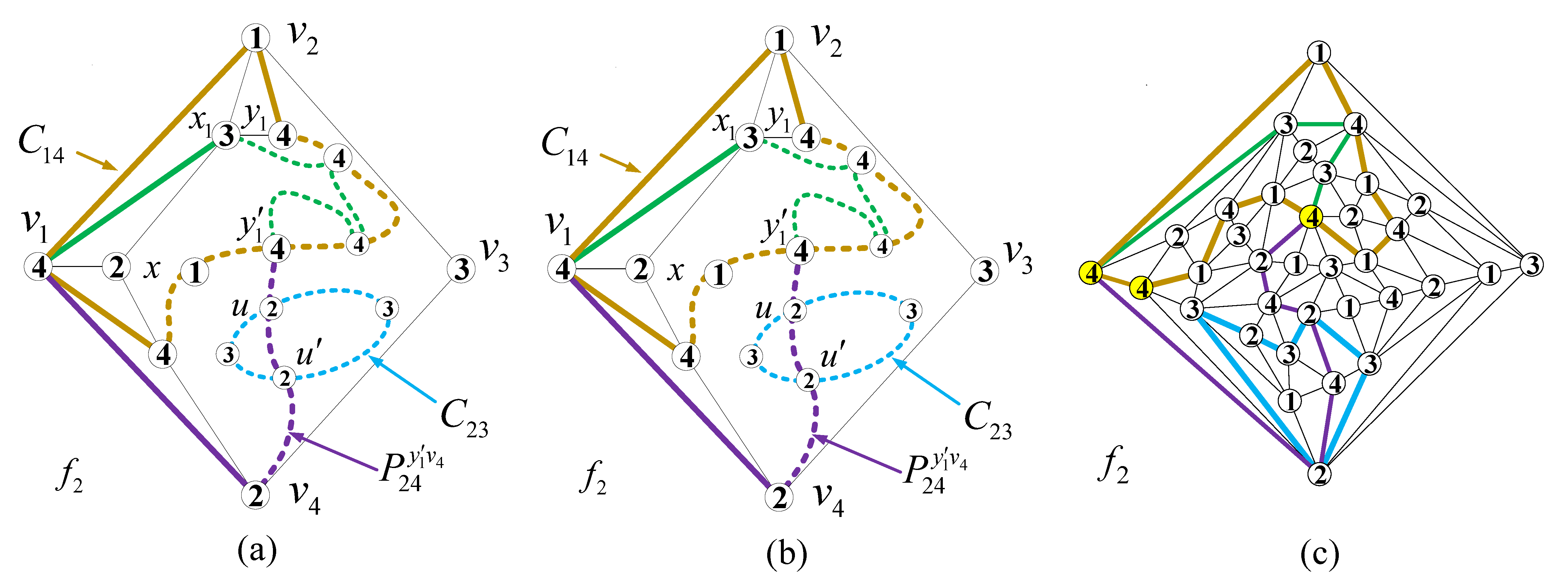}\\
  \caption {The 24-big-cycle $C'_{24}$ in $f_{14}$ or $f_{14}^{23}$ consists of more than one 34-paths of $f_2$ in the interior of $C_{14}$}\label{fig4-21}
\end{figure}

{\bf Step 3. Quasi 13-horizontal-axis decycle coloring}

For convenience,  we introduce a definition which is named as a \emph{pocket}. Let $G$ be a 4-chromatic MPG (or SMPG) with $\delta\geq 4$, and $f\in C_4^0(G)$ be a 4-coloring of $G$. Suppose that $G$ contains a subgraph $H^{C'}$ such that $H^{C'}$ is an SMPG with at least one internal vertex and the outer cycle $C'$ of $H^{C'}$ consists of two paths $P=t_1t_2\ldots t_n$ and $P'=t_1st_n$. If $f(P)=\{i,j\}, i,j\in \{1,2,3,4\}, i\neq j$, and $f(s)\in \{1,2,3,4\}\setminus f(P)$, then we call $H^{C'}$ an \emph{$ij$-pocket with mouth $t_1st_n$}.

If neither $f_{14}$ nor  $f_{14}^{23}$ contains a 24-big-cycle that  intersects with $C_{14}$, then there must be a 13-path from $x$ to $v_3$, which is called a \emph{quasi 13-horizontal-axis} and is denoted by $\ell_{13}^{xv_3}$. Let $\kappa_{13}$ be the subgraph induced by the set of vertices belonging to the cycle  $C=\ell_{13}^{xv_3} \cup v_2v_3 \cup xx_1v_2$ and its interior. Clearly, $\kappa_{13}$ is an 13-pocket with mouth $xx_1v_2$. Let $z$ be a vertex in  $\ell_{13}^{xv_3}$. Under $f_{14}$ or  $f_{14}^{23}$, if  there exists an 12-path from $x_1$ to $z$  in the interior of $\kappa_{13}$ and a 14-path from $z$ to $x'$ in the exterior of $\kappa_{13}$, then we call $z$ a \emph{bridge-vertex} on $\ell_{13}^{xv_3}$;  if  there is a 12-path from $x_1$ to $z$  in the interior of $\kappa_{13}$ but no 14-path from $z$ to $x'$ in the exterior of $\kappa_{13}$, then we call $z$ a \emph{semi-bridge-vertex} on $\ell_{13}^{xv_3}$.

Note that if  the pseudo 4-coloring $f_{14}$ or $f_{14}^{23}$ (of the 4-base-module $G^{C_4}$)  contains no  24-big-cycle, then there must be a quasi 13-horizontal-axis $\ell_{13}^{xv_3}$. In what follows, we will give a detailed discussion for the cases that $\ell_{13}^{xv_3}$ contains a bridge-vertex, contains a semi-bridge-vertex, and contains neither bridge-vertices nor semi-bridge-vertices. We always assume that these cases are under $f_{14}$ (the cases under $f_{14}^{23}$ are similar).

\textbf{Case 2.1. $\ell_{13}^{xv_3}$ contains a bridge-vertex.}

Let $z$ be a vertex of $\ell_{13}^{xv_3}$  such that $f_{14}(z)=1$, there is an 12-path $P_{12}^{x_1z}$ from $x_1$ to $z$ in the interior of $\kappa_{13}$, and there is an 14-path $P_{14}^{zx'}$ from $z$ to $x'$ in the exterior of $\kappa_{13}$ (see Figure \ref{fig4-41} (a)). Then, we can obtain a  decycle coloring $f^*$ as follows: first, under $f_{14}$, conduct a $K$-change for the 24-component in  $\kappa_{13}$, and denote by $f_{14}^{13}$ the new pseudo 4-coloring (see Figure \ref{fig4-41} (b)), i.e., interchange the colors (2 and 4) of vertices in $\{x_1\} \cup$ $Int(\kappa_{13})$, where $Int(\kappa_{13})$ is the set of internal vertices of  $\kappa_{13}$. Clearly, the 12-path $P_{12}^{x_1z}$ becomes an 14-path of $f_{14}^{13}$, denoted by $P_{14}^{x_1z}$. Thus, $C'_{14}=P_{14}^{x_1z}\cup P_{14}^{zx'} \cup v_1x' \cup v_1v_2\cup v_2x_1$ is an 14-cycle of $f_{14}^{13}$ (see Figure \ref{fig4-41} (b)); second, under $f_{14}^{13}$, conduct a $K$-change for the 23-component in the interior  of $C'_{14}$  and denote by $f_{13}^{14}$ the new pseudo 4-coloring (see Figure \ref{fig4-41} (c)). Note that no vertex adjacent to $v_1$ is colored with 3 under $f_{13}^{14}$. Therefore, based on $f_{13}^{14}$, we can obtain a decycle coloring $f^*$ by changing the color of $v_1$ from 4 to 3.  Figure \ref{fig4-41} (d) give an example for this case.

{\bf Remark 9.} If under $f_{14}$ the 13-component containing $\ell_{13}^{xv_3}$  includes an 13-cycle, and $C_{13}$ contains a bridge-vertex, then by conducting a $K$-change for the 24-component in the interior of $\kappa_{13}$ and simultaneously a $K$-change for the 24-component in the interior of $C_{13}$, we can obtain a pseudo 14-cycle whose interior contains vertex $x$; see Figure \ref{fig4-41} (e)$\sim$ (f) for an example ($f_{14}$ and $f_{14}^{23}$).

\begin{figure}[H]
  \centering
  \includegraphics[width=10cm]{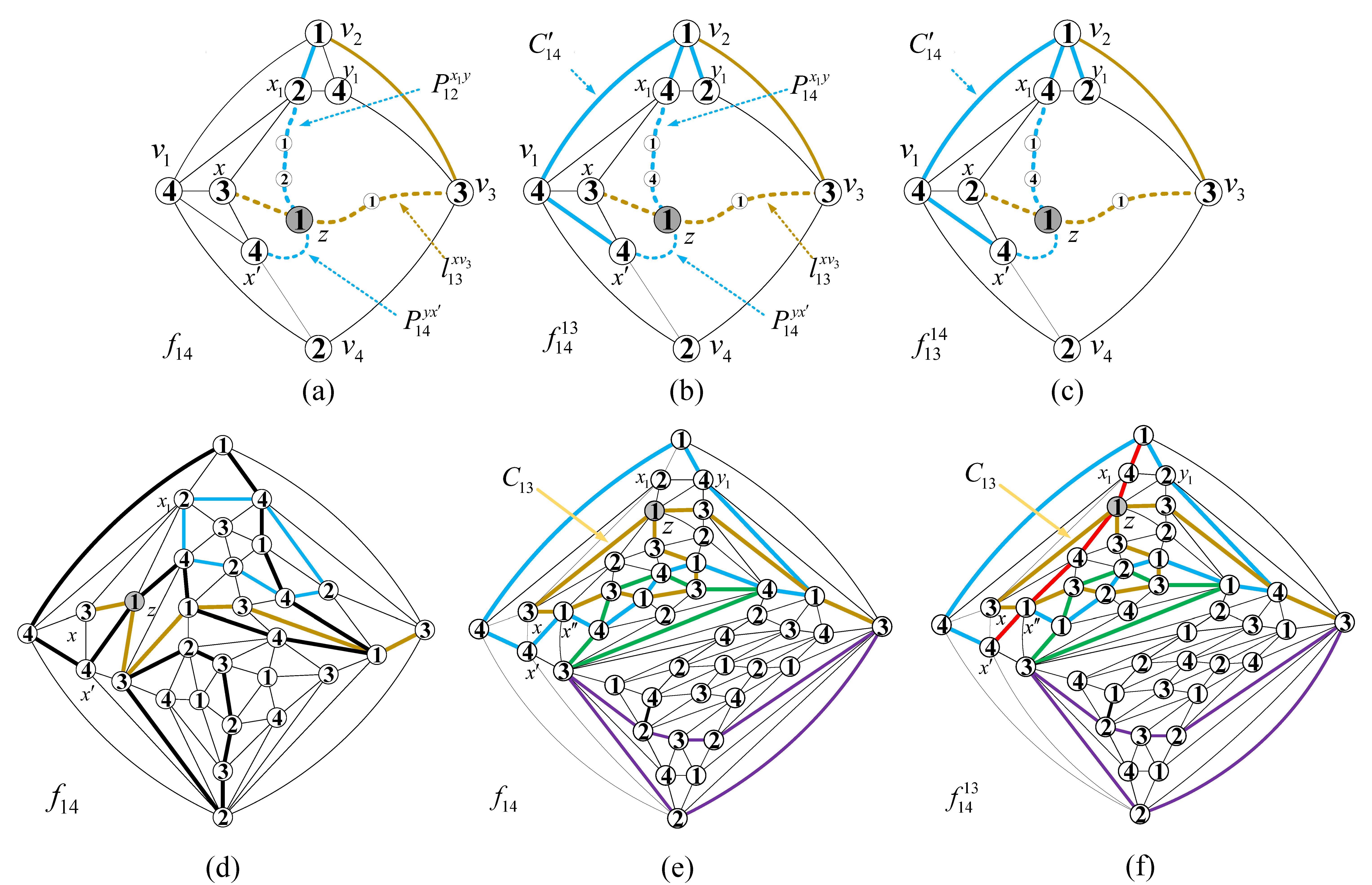}\\
  \caption {Illustration for the process of generating quasi 13-horizontal-axis colorings}\label{fig4-41}
\end{figure}

\textbf{Case 2.2. $\ell_{13}^{xv_3}$ contains no bridge-vertex but contains a semi-bridge-vertex.}

 Under $f_{14}$, let $t$ be a vertex in the interior of $\kappa_{13}$ that is colored with 4. According to the structure of vertices adjacent to $t$ on $\ell_{13}^{xv_3}$, we consider the following two cases.

\emph{Case 2.2.1}. The vertex $t$ is adjacent to a 2-length path $zz''z'$ in $\ell_{13}^{xv_3}$, where  $z$ and $z'$ are colored with 1 , $z''$ is colored with 3, and $z$ and $z'$ belong to the same 12-component in the exterior of $\kappa_{13}$.

In this case, we see that $f_{14}$ has an 12-path, denoted by $P_{12}^{zz'}$, which is in the exterior of $\kappa_{13}$. Clearly, $P_{12}^{zz'}\cup ztz'$ is an 12-pocket with mouth $ztz'$. Since under $f_{14}$ the subgraph induced by the set of vertices colored with 2 and 4 is connected,  there is a 24-path $P_{24}^{tv_1}$ from $t$ to $v_1$ in the interior of $\kappa_{13}$. Consider a vertex $t'\in  V(P_{12}^{zz'})$ with $f_{14}(t')=2$, such that there is a 23-path $P_{23}^{z''t'}$ from $t'$ to $z''$ and a 24-path $P_{24}^{t'v_1}$ from $t'$ to $v_1$; see Figure \ref{fig4-42} (a). Now, under $f_{14}$, conduct a $K$-change for the 24-component in the interior  of $\kappa_{13}$  and denote by $f_{14}^{13}$ the new pseudo 4-coloring. Clearly,  $C'_{12}=P_{12}^{zz'}\cup ztz'$ is an 12-cycle of $f_{14}^{13}$ (see Figure \ref{fig4-42} (b)). Under $f_{14}^{13}$, conduct a $K$-change for the 34-component in the interior  of $C'_{12}$  and denote by $f_{13}^{12}$ the new pseudo 4-coloring. Then,
the 23-path $P_{23}^{z''t'}$ of $f_{14}^{13}$ becomes a 24-path of $f_{13}^{12}$, denoted by $P_{24}^{z''t'}$  (see Figure \ref{fig4-42} (c)). So, $P_{24}^{tv_1}, P_{24}^{t'v_1}, tz''$, and $P_{24}^{z''t'}$ make up a 24-big-cycle of $f_{13}^{12}$, denoted by $C'_{24}$. With the method mentioned above, we can obtain a  decycle coloring based on $C'_{24}$.  Figures \ref{fig4-42} (d) and (e) give the colorings $f_{14}$ and $f_{14}^{13}$, respectively, of an example for this case.

\begin{figure}[H]
  \centering
  \includegraphics[width=12.3cm]{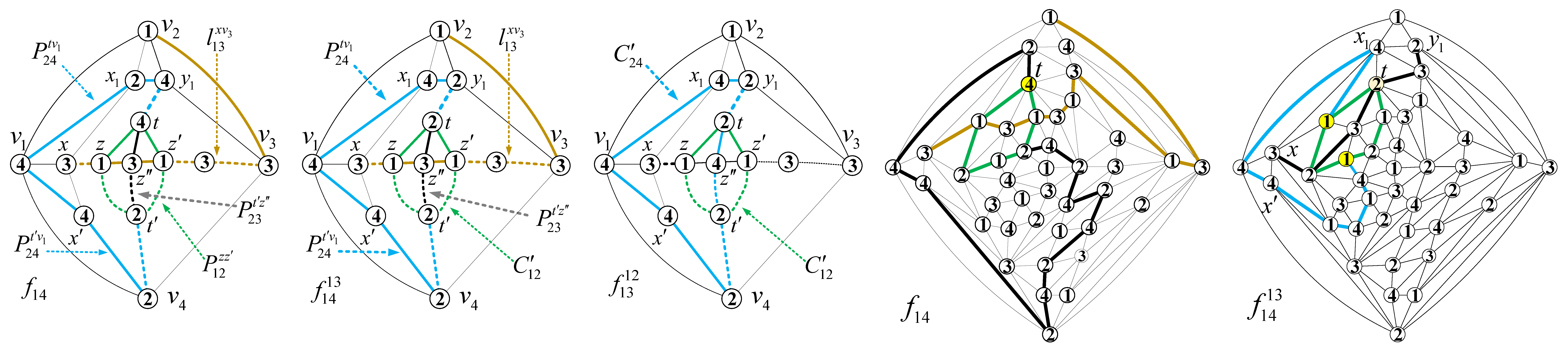}\\
  (a) \hspace{2cm}(b)\hspace{2cm}(c)\hspace{2cm}(d)\hspace{2cm}(e)
  \caption {Illustration for the process of a decycle coloring for case 2.2.1, and an example}\label{fig4-42}
\end{figure}

\emph{Case 2.2.2}. $\kappa_{23}$ contains no such a vertex $t$ that satisfies the condition in Case 2.2.1. Then, $f_{14}$ must contains a 34-cycle $C'_{34}$, and  $C'_{34}$ intersects with $\ell_{13}^{xv_3}$.

\begin{figure}[H]
  \centering
  \includegraphics[width=12.3cm]{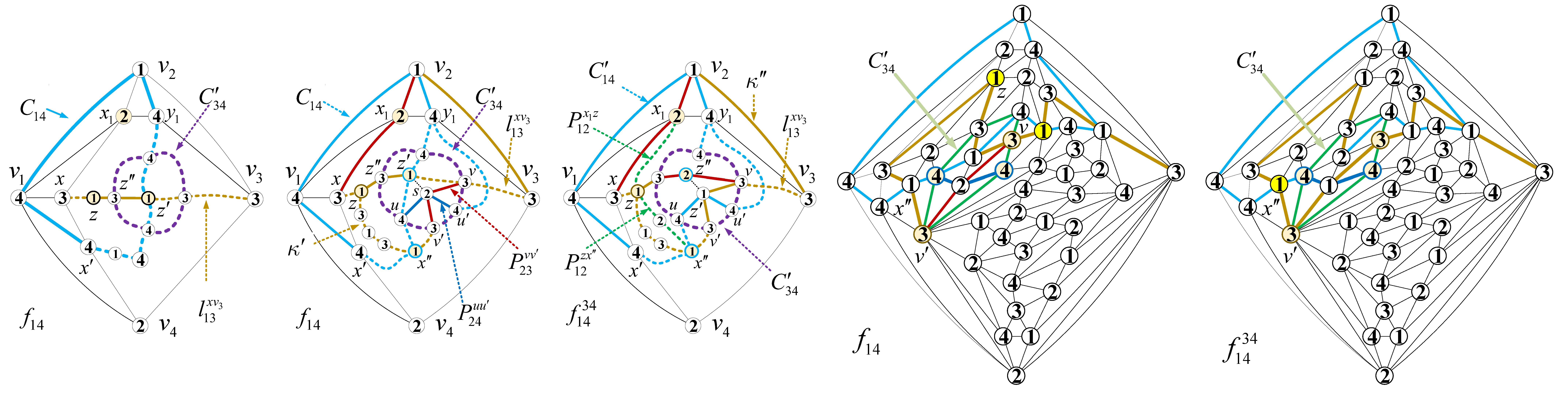}\\
  (a) \hspace{2cm}(b)\hspace{2cm}(c)\hspace{2cm}(d)\hspace{2cm}(e)
  \caption {Illustration for the situation that the quasi 13-horizontal-axis contains no bridge-vertex under $f_{14}$  while a new quasi 13-horizontal-axis contains a bridge-vertex under $f_{14}^{34}$ }\label{fignew4-41}
\end{figure}

Since $\kappa_{13}$ contains no vertex $t$ with the property in case 2.2.1, there must be a path $zz''z'$  in
$\ell_{13}^{xv_3}$, where $z$ and $z'$ are colored with 1, $z''$ is colored with 3, and exactly one of $z$ and $z'$ is in the interior of $C'_{34}$. This implies that $\ell_{13}^{xv_3}$ intersects with $C'_{34}$. See  Figure \ref{fignew4-41}(a) for the sketch of this case.
Conduct a $K$-change for the 12-component in the interior  of $C'_{34}$  and denote by $f_{14}^{34}$ the new pseudo 4-coloring. Then, under $f_{14}^{34}$, there must be a new pseudo 14-cycle $C'_{14}$ and a new  quasi 13-horizontal-axis (otherwise, either there is no odd-length 14-cycle, or there is a 24-big-cycle. Still denote by $\ell_{13}^{xv_3}$ the quasi 13-horizontal-axis). So, there are two vertices $u$ and $u'$ on $C'_{34}$ that are colored by 4. Based on $f_{14}$, there is a 24-path from $u$ to $u'$ in the interior of $C'_{34}$, denoted by $P_{24}^{uu'}$. Clearly, $P_{24}^{uu'}$ is in the either interior or  exterior of $\kappa_{13}$, and there are 14-paths from $u$ to $x'$ and from $u'$ to $y_1$ in the exterior of $C'_{34}$
In addition, there exist two vertices $v$ and $v'$ on $C'_{34}$ that are colored with 3, and   $f_{14}$ contains a 23-path from $v$ to $v'$ in the interior of $C'_{34}$, denoted by $P_{23}^{vv'}$. Here, we assume that $P_{23}^{vv'}$ is in the exterior of $C_{14}$ (the proof for the case that $P_{23}^{vv'}$ is in the interior of $C_{14}$ is similar). There is an 13-path from $v$ to $v_3$ in the exterior of $C'_{34}$, denoted by $P_{13}^{vv'}$. Clearly, $P_{13}^{vv'}$ is a sub-path of the quasi 13-horizontal-axis (under $f_{14}$). Note that there is an 13-path from $v'$ to $x$, denoted by $P_{13}^{v'x}$. Without loss of generality, we assume that $P_{13}^{v'x}$ intersect with  $\ell_{13}^{xv_3}$ at $z$. So, under $f_{14}$, there exists an 13-pocket $\kappa'_{13}$ with mouth $vsv'$ and outer cycle $P_{13}^{vz}\cup P_{13}^{zv'}\cup vsv'$, and hence there are two 13-pockets $\kappa_{13}$ and $\kappa'_{13}$. Let $x'' \in V(\kappa'_{13})\cap V(C_{14})$ be a vertex in $C_{14}$ that has the smallest distance with $v'$ in $G^{C_4}$. Then there must be an 14-path from $x''$ to $x'$, denoted by $P_{14}^{x'x''}$; see Figure \ref{fignew4-41} (b).

Observe that under $f_{14}$ there is no bridge-vertex. Conduct a $K$-change in the interior of $C'_{34}$ and obtain a new pseudo 4-coloring $f_{14}^{34}$. Under $f_{14}^{34}$, a new 13-horizontal-axis is still denoted by $\ell_{13}^{xv_3}$ and a new corresponding 13-pocket is denoted by $\kappa''_{13}$. If $\ell_{13}^{xv_3}$ contains a bridge-vertex, then we can obtain a quasi 13-horizontal-axis decycle coloring. See Figure \ref{fignew4-41} (c) for the schematic diagram for this situation and see  Figure \ref{fignew4-41} $(d)\sim (e)$ for an example.

Under $f_{14}$, since $z$ is a semi-bridge-vertex and $z$ is in the 14-component including $C_{14}$, there must be an 14-path from $z$ to an internal vertex $y_1$ of $\kappa''_{13}$. For this case,
under $f_2$, there is still an 14-path from $z$ to $y_1$. Note that $f_2$ contains only one 34-component, in  which there is a pseudo 34-cycle $C_{34}$ and a 34-horizontal-axis $v_1x_1y_1v_3$. So, there are in total three 12-components.
We  further consider two subcases. For convenience, let $x_2$ be the third vertex of a triangle with an edge $x_1y_1$ in the following proof.

\textbf{(i) $x_2$ is in the exterior of $C_{34}$}

\begin{figure}[H]
  \centering
  \includegraphics[width=8cm]{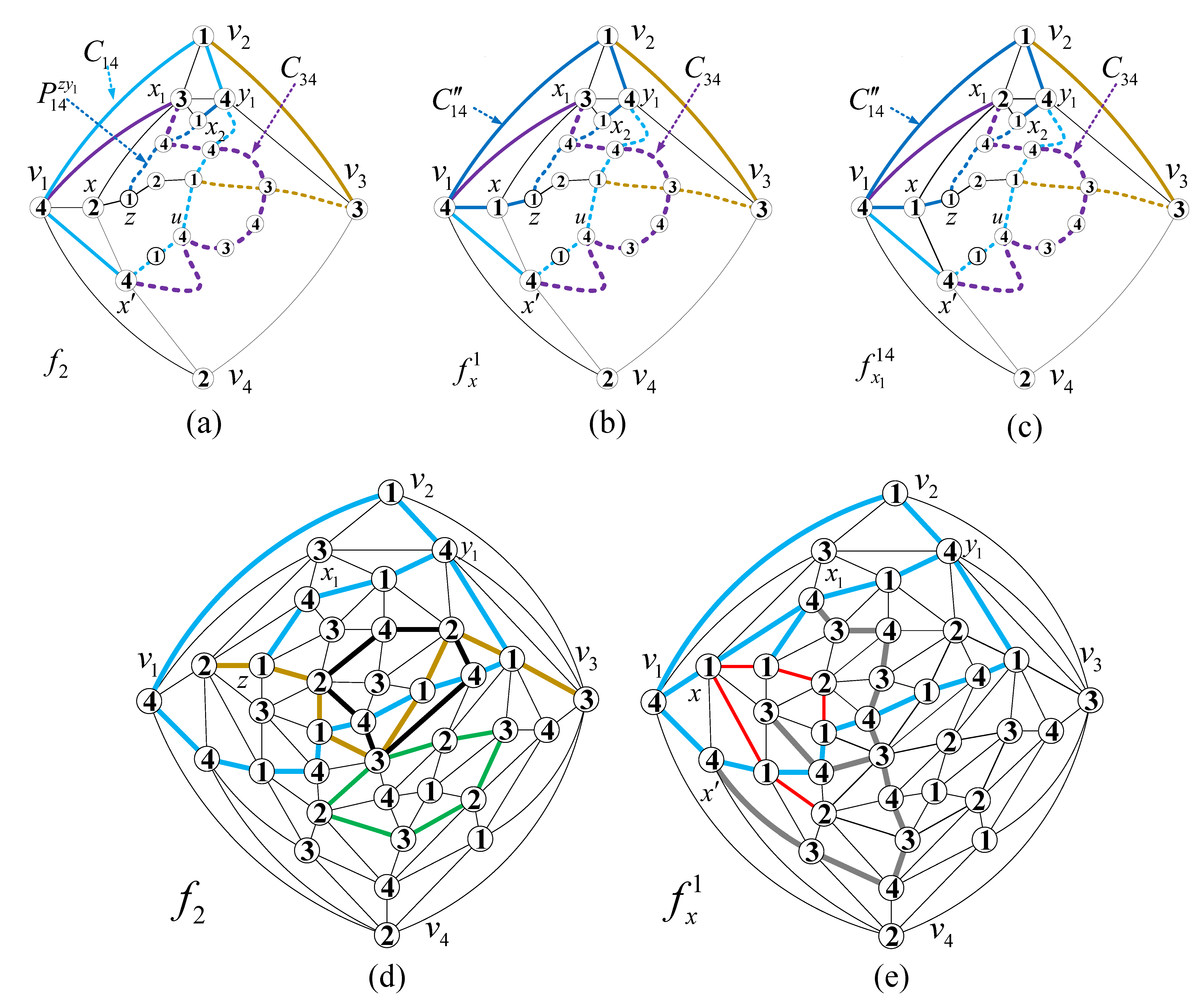}\\
  \caption {Illustration for the situation that $x_2$ is in the exterior of $C_{34}$ and an example}\label{fignew4-42}
\end{figure}

See Figure \ref{fignew4-42} (a) for the diagram of this structure. We can obtain a decycle coloring by the following steps. First, under $f_2$, change the color of $x$ from $2$ to 1, and denote by $f_x^1$ the new pseudo 4-coloring. Clearly, the 12-components remain unchanged from $f_2$ to $f_x^1$. In addition, $f_x^1$ has two pseudo 14-cycles, where one is $C_{14}$ and the other is a new pseudo 14-cycle, denoted by $C''_{14}$ (see Figure \ref{fignew4-42} (b)). Second, conduct a $K$-change for the 23-component in the interior of $C''_{14}$ and obtain a new pseudo 4-coloring, denoted by $f^{14}_{x_1}$ (see Figure \ref{fignew4-42} (c)). Note that under $f^{14}_{x_1}$, no vertex adjacent to $v_1$ is colored with 3. So, we can change the color of $v_1$ from $4$ to 3. Observe that all of the  pseudo 11-edges belong to the 12-component in the interior of the 34-pocket. Therefore, all of these pseudo 11-edges can be eliminated, which results in a decycle coloring $f^*$. See Figure \ref{fignew4-42} (d)$\sim$ (e) for an example of this situation.

\textbf{(ii) $x_2$ is in the interior of $C_{34}$}

See Figure \ref{fignew4-43} (a) for the diagram of this structure. Since both $x$ and $x_2$ are in the interior of $C_{34}$, $f_2$ contains a 12-path $P_{12}^{xx_2}$ from $x$ to $x_2$. We can get a decycle coloring by the following steps.

\begin{figure}[H]
  \centering
  \includegraphics[width=9cm]{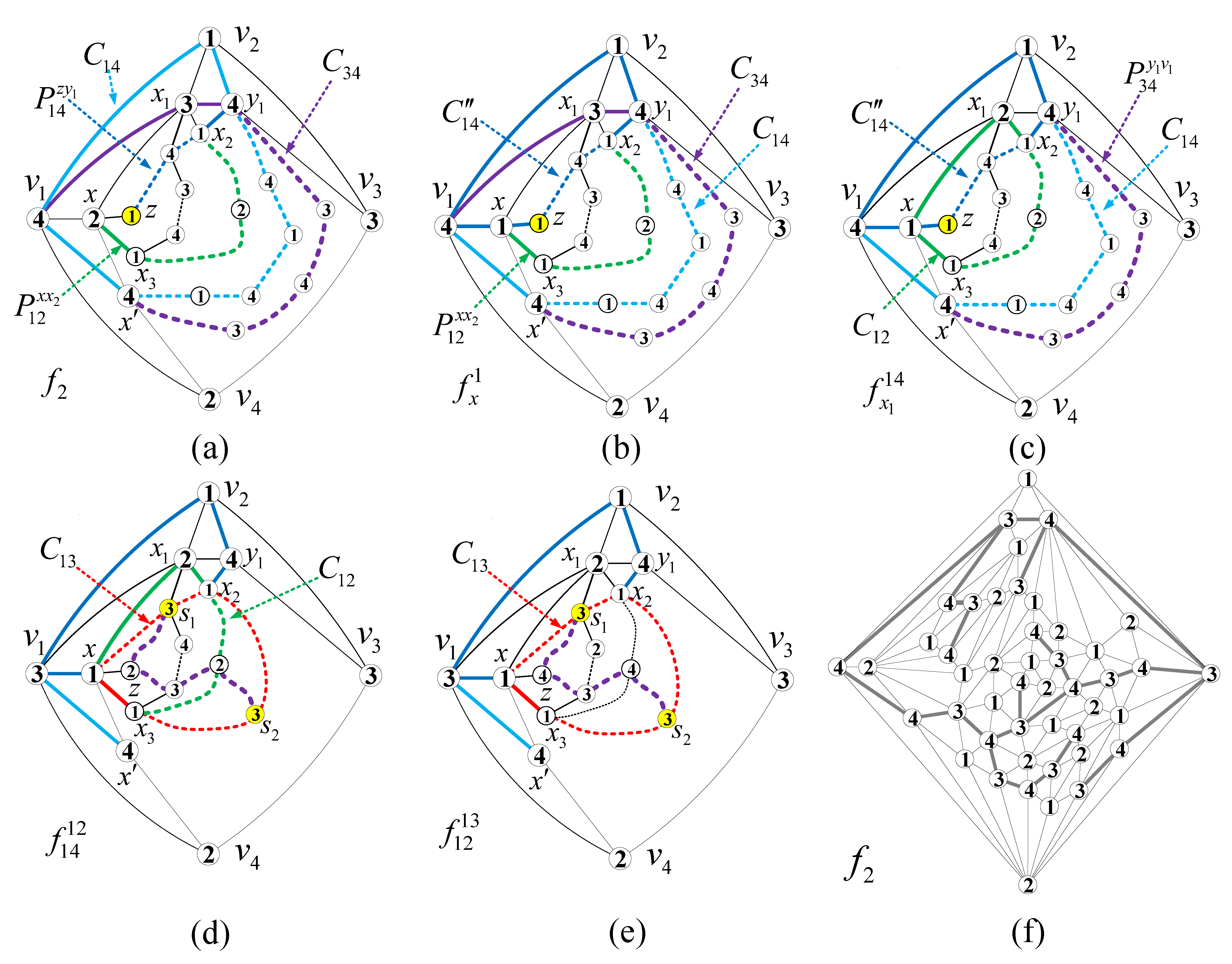}\\
  \caption {Illustration for the situation that $x_2$ is in the interior of $C_{34}$ and an example}\label{fignew4-43}
\end{figure}

First, under $f_2$, change the color of $x$ from $2$ to 1, and denote by $f_x^1$ the new pseudo 4-coloring. Clearly, the 12-components remain unchanged from $f_2$ to $f_x^1$. In addition, $f_x^1$ has two pseudo 14-cycles, where one is $C_{14}$ and the other is a new pseudo 14-cycle, denoted by $C''_{14}$ (see Figure \ref{fignew4-43} (b)).

Second, conduct a $K$-change for the 23-component in the interior of $C''_{14}$ and obtain a new pseudo 4-coloring, denoted by $f^{14}_{x_1}$. The following three statements hold:  no vertex adjacent to $v_1$ is  colored with 3 under $f^{14}_{x_1}$; $f^{14}_{x_1}$ contains an 12-cycle (denoted by $C_{12}$) which consists of $P_{12}^{xx_2}$ and $xx_1x_2$; there may be more that one pseudo 11-edges, in which one pseudo 11-edge is in the pseudo odd-length cycle $C_{12}$ (so it can not be eliminated by one pseudo 11-edge $K$-change) and all of the other pseudo 11-edges are in the interiors of the corresponding 34-pockets; see Figure \ref{fignew4-43} (c).

Third, under $f_{x_1}^{14}$, change the color of  $v_1$ from 4 to 3. Then, the pseudo 11-edges (except the one on $C_{12}$) can be eliminated by using the 34-pockets. Meanwhile, conduct a $K$-change for the 34-component in the interior of $C_{12}$ and obtain a new pseudo 4-coloring, denoted by $f^{12}_{14}$. By using the technique of  cycle-related graphs, we can deduce that $f_{14}^{12}$ contains an 13-cycle that intersects with $C_{12}$, denoted by $C_{13}$. Consider the the cycle-related graph $H_1^{f_{x_1}^{14}}(C_{12})$ of $C_{12}$. Let $u_1=v_2, u_2=x_1, u_3,\ldots,u_p=\{x,v_1\} (p\geq 2)$ be the consecutive vertices along with $C$ that are colored with 1. Observe that $f_{x_1}^{14}$ contains two 14-components (one is in the interior of a 23-cycle $C_{23}$ and the other is in the exterior of $C_{23}$), in which there is an 14-cycle $C''_{14}$ intersecting with $C_{12}$. So, the number of lines in $H_1^{f_{x_1}^{14}}(C_{12})$ that are contributed by 14-components is equal to $p$. Since $f_{x_1}^{14}$ contains exactly one 13-component, we have that the number of lines in $H_1^{f_{x_1}^{14}}(C_{12})$ that are contributed by 13-components is equal to $p-1$. This shows that $H_1^{f_{x_1}^{14}}(C_{12})$ contains in total $2p-1$ lines. Thus, by an analogous proof as that for  $H_4^{f_{2}}(C_{14})$, we have that $H_1^{f_{14}^{12}}(C_{12})$ contains either 13-cycles or 14-cycles, all of which consist of both bold lines and fine lines.  Now, according to the structure of $C_{12}$, we see that $f_{14}^{12}$ contains an 13-cycle $C_{13}$ that intersects with $C_{12}$. Under $f_{14}^{12}$, there is only one 23-component in the interior of $C_{14}$; in particular, there is a 23-path between two vertices $s_1,s_2 \in  V(C_{13})$, where $s_1$ is adjacent to $x_1$ and $s_2$ is the vertex of $C_{13}$ that is colored with 3 and  has the smallest distance with $x_2$; see Figure \ref{fignew4-43} (d).

Four, under $f_{14}^{12}$, conduct a $K$-change for the 24-component in the interior of $C_{13}$ and obtain a new pseudo 4-coloring, denoted by $f^{13}_{12}$; see Figure \ref{fignew4-43} (e). Clearly, $f_{12}^{13}$ contains no pseudo odd-length 12-cycle.  
By conducting a pseudo 11-edge ($xx_3$) $K$-change (in $G^{C_4}$) for the 12-component, we can
obtain a decycle coloring $f^*$. Figure \ref{fignew4-43} (f) gives  an example (a 4-coloring under $f_2$) of this case, in which the labels of vertices are omitted.

\textbf{Case 2.3. $\ell_{13}^{xv_3}$ contains neither bridge-vertex nor semi-bridge-vertex.}

Under $f_{14}$, if the quasi 13-horizontal-axis contains neither a bridge vertex nor a semi-bridge-vertex, then there must be a 34-path from $y_1$ to $x$, denoted by $P_{34}^{y_1x}$. So, $f_{14}$ contains a 34-cycle $C'_{34}$ and also a 34-horizontal-axis; see Figure \ref{fignew4-44} (a).

Note that $f_{14}$ contains two 14-components, where one is in the interior of a 23-cycle $C_{23}$ and the other (including $C_{14}$) is in the exterior of $C_{23}$. Let $s$ be a vertex in  $P_{34}^{y_1x}$ that is adjacent to $x$ and is colored with 4. Since $s$ and $x'$ belong to the same 14-component of $f_{14}$, there is an 14-path $P_{14}^{x's}$ between them. Clearly, the length of $P_{14}^{x's}$ is at least 2. Let $s'$ be the vertex in quasi 13-horizontal axis that is adjacent to $x$.
In what follows, we consider two subcases in terms of the length of $P_{14}^{x's}$.

\textbf{Case 2.3.1.} $P_{14}^{x's}=x's's$

Clearly, $x's's$ be a bichromatic path of $f_{14}$, where $f_{14}(s)=f_{14}(x')=4$ and $f_{14}(s')=1$; see Figure \ref{fignew4-44} (b). For this case, if there does not exist any 12-path between $s'$ and $v_4$, then there must exist a 34-path between $x'$ and $v_3$, which contradicts to the fact that $f_{14}$ contains only one 34-cycle. So, $f_{14}$ contains an 12-path from $s'$ to $v_4$, denoted by $P_{12}^{s'v_4}$.  Since $f_{14}$ contains only one 14-cycle, it follows that $s's$ is in the interior of a 23-pocket; see Figure \ref{fignew4-44} (b). Then, we can obtain a decycle coloring as follows.

First, under $f_{14}$, change the color of $x'$ from 4 to 3 and denote by $f_{x'}^{3}$ the resulting pseudo 4-coloring. Then, $f_{x'}^{3}$ contains a 23-cycle, denoted by $C'_{23}$; see Figure \ref{fignew4-44} (c).

Second, under $f_{x'}^{3}$, conduct  a $K$-change for the 14-component in the interior of $C'_{23}$, and then conduct  a $K$-change for the 13-component containing $x$, denoted by $f_3^{23}$ the resulting 4-coloring; see Figure \ref{fignew4-44} (d). Under $f_3^{23}$, the 13-horizontal-axis is eliminated. So, all of the pseudo 11-edges can be eliminated by conducting pseudo 11-edge $K$-changes for 13-components. This results in a decycle coloring.  Figure \ref{fignew4-44} (e)$\sim$(f) give  an example of this case, in which the labels of vertices are omitted.

 \begin{figure}[H]
  \centering
  \includegraphics[width=9cm]{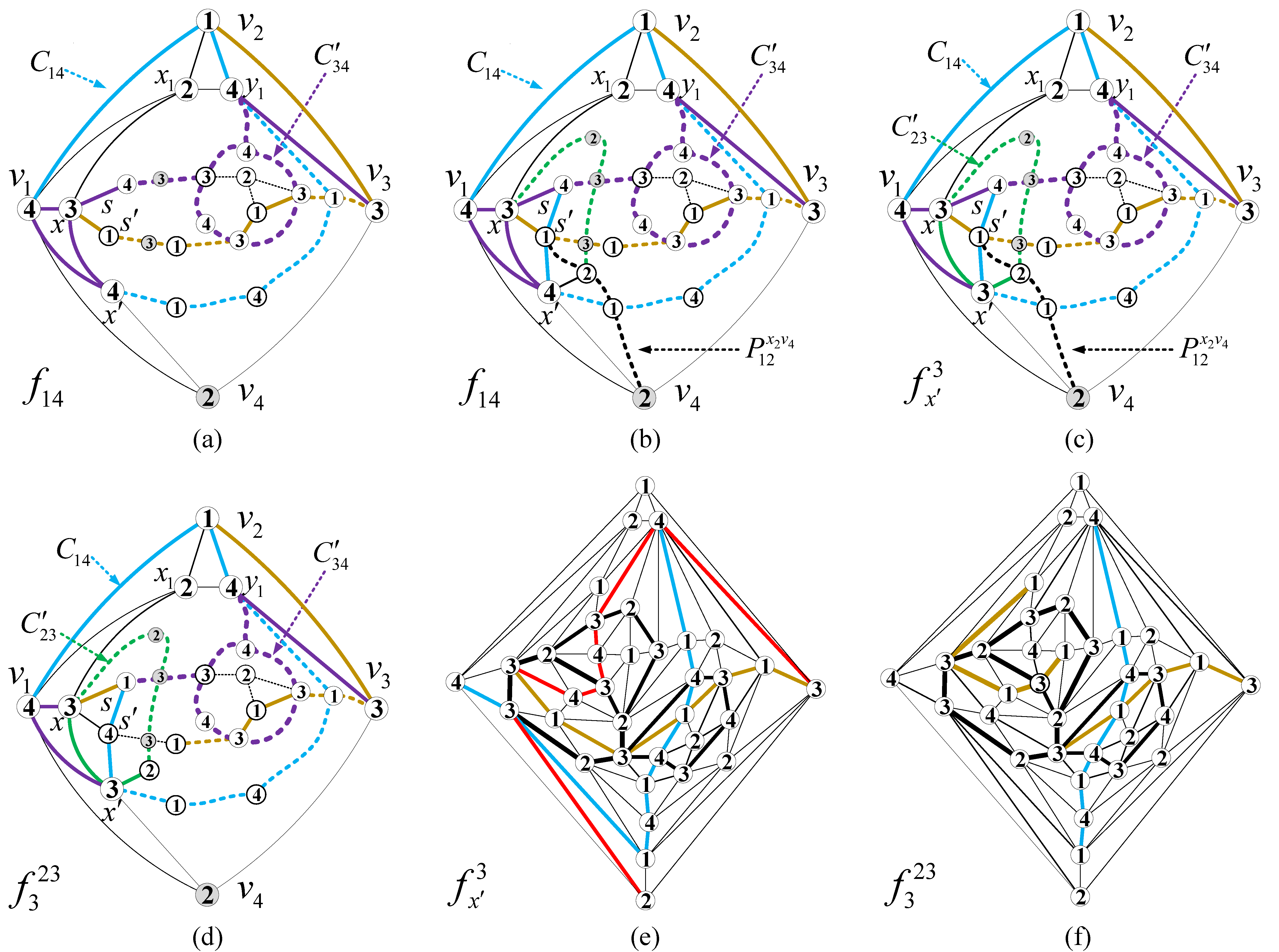}\\
  \caption {Illustration for the process of obtaining a decycle coloring when $\ell_{13}^{xv_3}$ (based on $x's's$ ) contains no bridge-vertex and semi-bridge-vertex, and an example}\label{fignew4-44}
\end{figure}

\textbf{Case 2.3.2.} $P_{14}^{x's} \neq x's's$

In this case, there is an 14-pocket with mouth $sxx'$ and the outer cycle $P_{14}^{x's}\cup sxx'$. We use  $\kappa_{14}$  to denote this 14-pocket.  Clearly, by a similar analysis as Case 2.3.1,  there is an 12-path $P_{12}^{s''v_4}$ from $s''$ to $v_4$, where $s''$ is a neighbor of $x$ that is colored with 2 under $f_{14}$; see Figure \ref{fignew4-45} (a). We can obtain a decycle coloring as follows.

First, under $f_{14}$, change the color of $x$ from 3 to 1 and denote by $f_{x}^{1}$ the resulting pseudo 4-coloring. Then, $\kappa_{14}$ (under $f_{14}$) becomes an 14-cycle of $f_x^1$, denoted by $C'_{14}$; see Figure \ref{fignew4-45} (b).

Second, under $f_{x}^{1}$, conduct  a $K$-change for the 23-component in the interior of $C'_{14}$, and then change the color of $v_1$ from 4 to 3. If there is a pseudo 11-edge in the exterior of $C'_{14}$, then we conduct a $K$-change for the 13-component containing $x$ and denote by $f_1^{14}$ the resulting 4-coloring; see Figure \ref{fignew4-45} (c). Observe that $f_1^{14}$ contains no 12-path from $s''$ to $v_4$.   
So, all of the pseudo 11-edges can be eliminated by conducting pseudo 11-edge $K$-changes for 12-components. This results in a decycle coloring.  Figure \ref{fignew4-45} (d)$\sim$(f) describe   an example of this case, in which we only give the 4-colorings of the 4-base-module under $f_{14}, f_x^1, $ and $f_1^{14}$. This completes the proof of Claim 1. \qed
\end{proof}
 \begin{figure}[H]
  \centering
  \includegraphics[width=10cm]{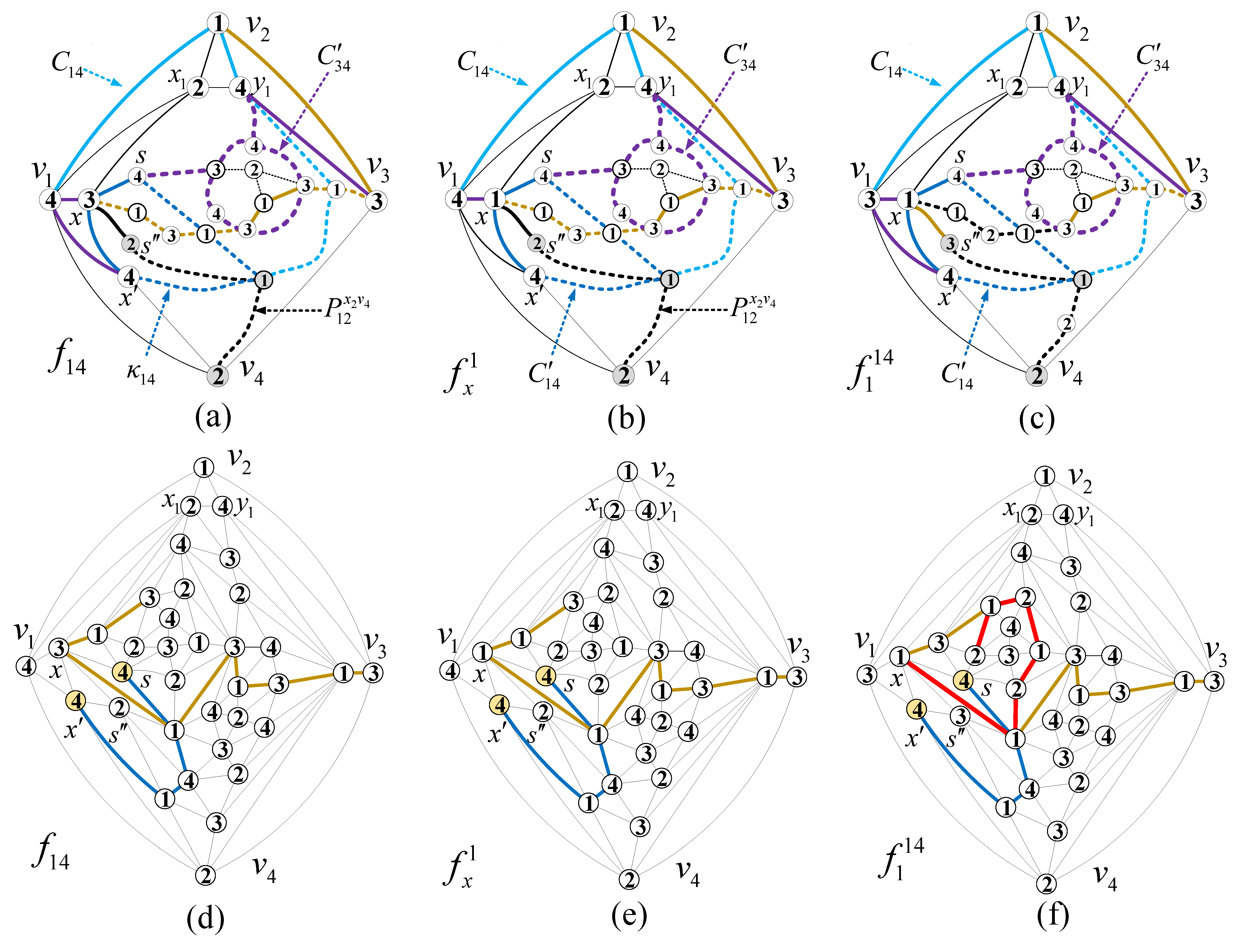}\\
  \caption {Illustration for the process of obtaining a decycle coloring when $P_{14}^{x's} \neq x's's$ and $\ell_{13}^{xv_3}$ contains neither bridge-vertex nor semi-bridge-vertex, and an example}\label{fignew4-45}
\end{figure}

{\bf Claim 2.} \textbf{55-343 type 4-base-modules are decyclizable}.

{\bf Proof of Claim 2.} For the 55-343 type 4-base-module shown in Figure \ref{fig4-17} (b), its 14-sloped-axis  under $f_2$ is shown in Figure \ref{fig4-25}(a), in which there is a unique pseudo 44-edge $v_1x$ that is in the pseudo 14-cycle $C_{14}$. Observe that the pseudo 44-edge $v_1x$ is 33-type. If $(G^{C_4})_{24}^{f_2}$ contains no odd-cycle, then the pseudo edge can be eliminated by conducting a pseudo 44-edge $K$-change for  the 24-component containing $v_1$, and the proof is completed. So, it is enough to deal with the case that under $f_2$, $v_1x$ is an edge of an odd-length 24-cycle $C_{24}$. We consider the following two cases.

\begin{figure}[H]
  \centering
  \includegraphics[width=10cm]{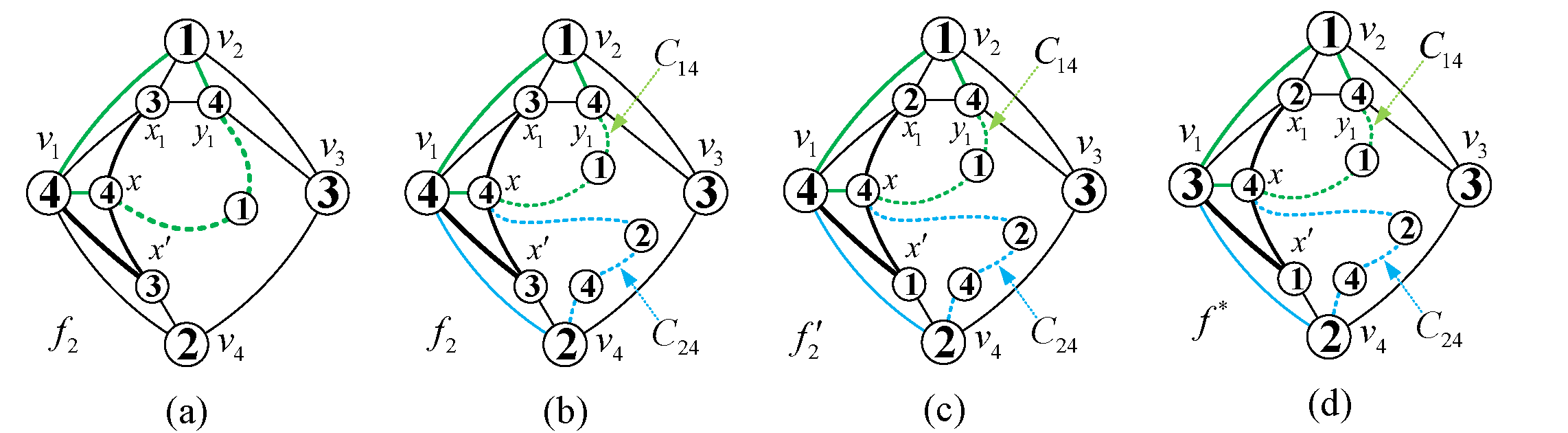}\\
  \caption {Illustration for the proof of 55-343 type 4-base-modules when $C_{14}$ and $C_{24}$ are nonintersecting under $f_2$}\label{fig4-25}
\end{figure}

{\bf Case 1.} Under $f_2$, $C_{14}$ and $C_{24}$ are nonintersecting. In this case, $v_4$ must be a vertex of $C_{24}$; see Figure \ref{fig4-25} (b).
Thus, we first conduct $K$-changes for the 13-component and 23-component  in the interior of $C_{24}$ and $C_{14}$, respectively,  and denote by $f'_2$ the resulting pseudo 4-coloring. Obviously, $f'_2(x')=1$ and $f'_2(x_1)=2$; see Figure \ref{fig4-25} (c). Based on $f'_2$, we further change the color of $v_1$ from 4 to 3, and obtain a decycle coloring $f^*$; see Figure \ref{fig4-25} (d).

{\bf Case 2.} Under $f_2$, $C_{14}$ and $C_{34}$ are intersecting. By Theorem \ref{thm4.6},  $f_2$ contains either a 23-cycle $C_{23}$ or a 12-cycle $C_{12}$. So, $C^2(f_2)$ contains at least three bichromatic cycles such that one of them is a bichromatic cycle without pseudo edges (called proper bichromatic cycle), and two of them are  pseudo bichromatic  cycles $C_{14}$ and $C_{24}$. Without loss of generality, we assume the proper bichromatic cycle is  $C_{23}$ (the case for $C_{12}$ is similar). Then, $C_{23}$ is not a module-cycle of $f_1$ (otherwise, $G^{C_4}$ is a module-cycle 4-base-module).

\begin{figure}[H]
  \centering
  \includegraphics[width=8cm]{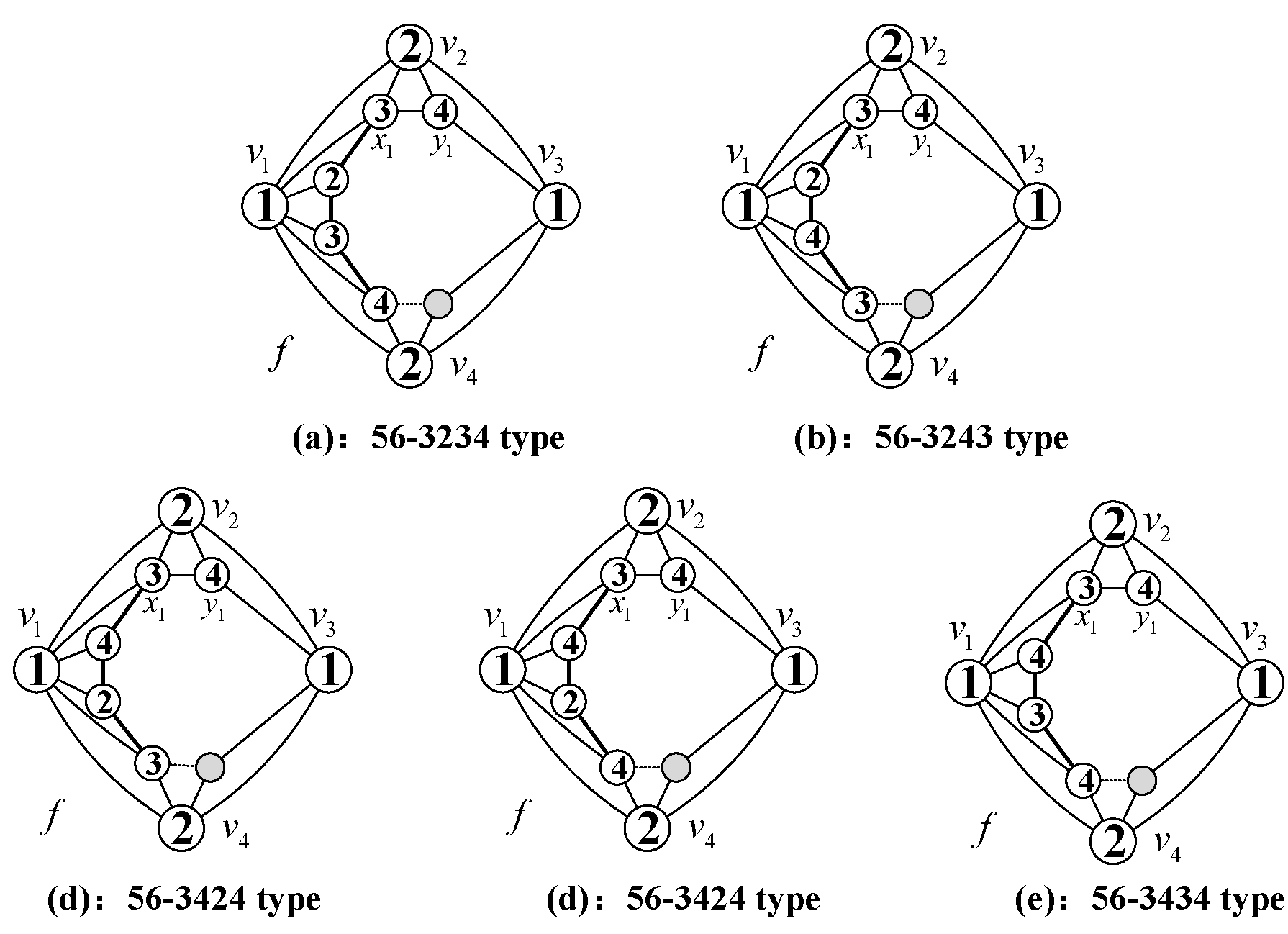}\\
  \caption {Illustration for the five types of 4-base-modules}\label{fig4-26}
\end{figure}

Note that in the 55-324 type 4-base-module, $C^2(f_2)$ contains bichromatic cycles $C_{14}, C_{34}, C_{23}$, where $C_{14}$ and $C_{34}$ are intersecting; however, in the 55-343 type 4-base-module, $C^2(f_2)$ contains bichromatic cycles $C_{14}, C_{24}, C_{23}$, where $C_{14}$ and $C_{24}$ are intersecting. Therefore, the process of constructing a decycle coloring for the 55-343 type 4-base-module is similar to that for the 55-324 type 4-base-module, and is omitted here.

If $d(v_1)=6$, then there are total five cases of colorings under $f_2$; see Figure \ref{fig4-26} (a) $\sim$ (e) for the structures of these 4-base-modules under the module-coloring $f$, which are called $56$-3234 type 4-base-modules, $56$-3243 type 4-base-modules, $56$-3423 type 4-base-modules, $56$-3424 type 4-base-modules, and $56$-3434 type 4-base-modules, respectively. The process of constructing a decycle coloring for these five type 4-base-modules are similar to that for 55-324 type 4-base-modules.

{\bf Claim 3.} \textbf{56-3234 type 4-base-modules are decyclizable}.

{\bf Proof of Claim 3.} For the 56-3234 type 4-base-module shown in Figure \ref{fig4-26} (a), its 14-sloped-axis  under $f_2$ is shown in  Figure \ref{fig4-27}(a). Based on $f_2$, conduct a $K$-change for the 23-component in the interior of  $C_{14}$, and denote by $f_{14}$ the resulting pseudo 4-coloring; see Figure \ref{fig4-27}(b). Under $f_{14}$, the pseudo 44-edge $v_1x$ is 22-type. So, we can obtain a decycle coloring $f^*$ for this type by a similar way as that for 55-324 types.

\begin{figure}[H]
  \centering
  \includegraphics[width=6cm]{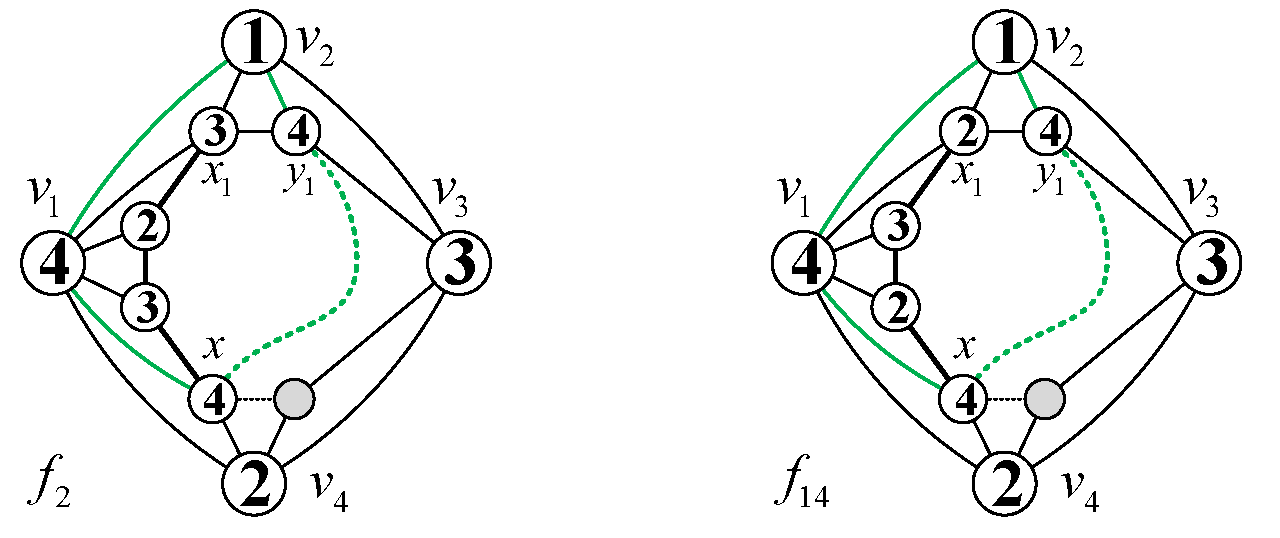}\\
  (a) \hspace{3cm}(b)
  \caption {Illustration for the conversion from 56-3234 type to 55-324 type}\label{fig4-27}
\end{figure}

{\bf Claim 4.} \textbf{56-3243 type 4-base-modules are decyclizable}.

{\bf Proof of Claim 4.} For the 56-3243 type 4-base-module shown in Figure \ref{fig4-26} (b), its 14-sloped-axis  under $f_2$ is shown in  Figure \ref{fig4-28}(a). Based on $f_2$, conduct a $K$-change for the 23-component in the interior of  $C_{14}$, and denote by $f_{14}$ the resulting pseudo 4-coloring; see Figure \ref{fig4-28}(b). Under $f_{14}$, the pseudo 44-edge $v_1x$ is 33-type. So, we can obtain a decycle coloring $f^*$ for this type by a similar way as that for 55-343 types.

\begin{figure}[H]
  \centering
  \includegraphics[width=6cm]{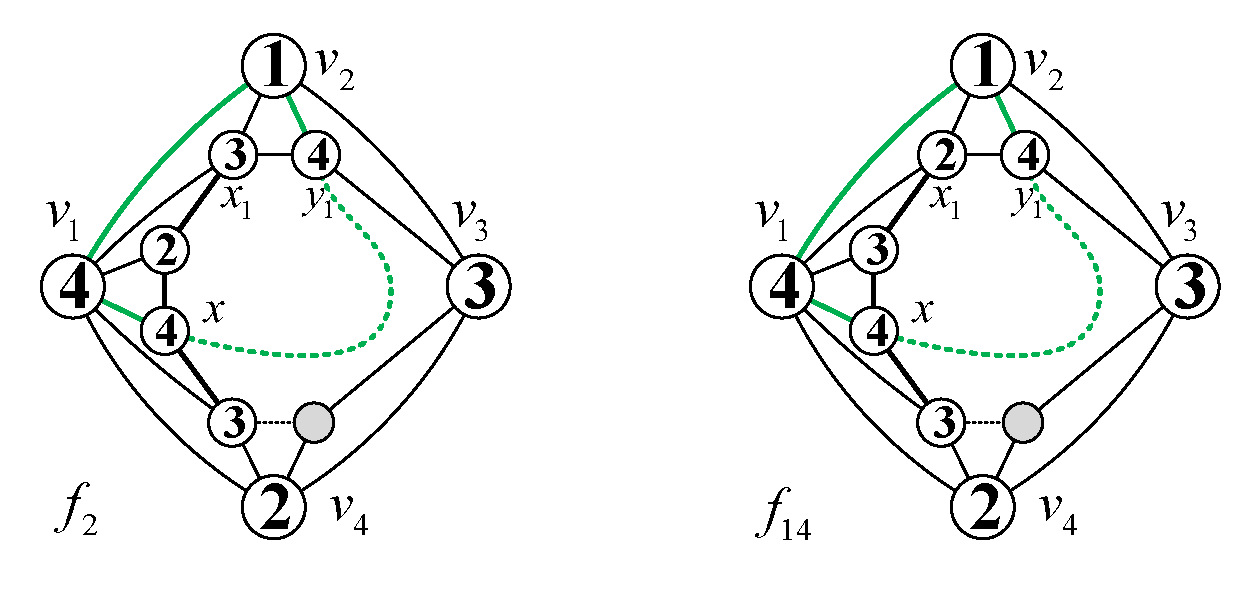}\\
  (a) \hspace{3cm}(b)
  \caption {Illustration for the conversion from 56-3243 type to 55-343 type}\label{fig4-28}
\end{figure}

{\bf Claim 5.} \textbf{56-3423 type 4-base-modules are decyclizable}.

\begin{figure}[H]
  \centering
  \includegraphics[width=5.5cm]{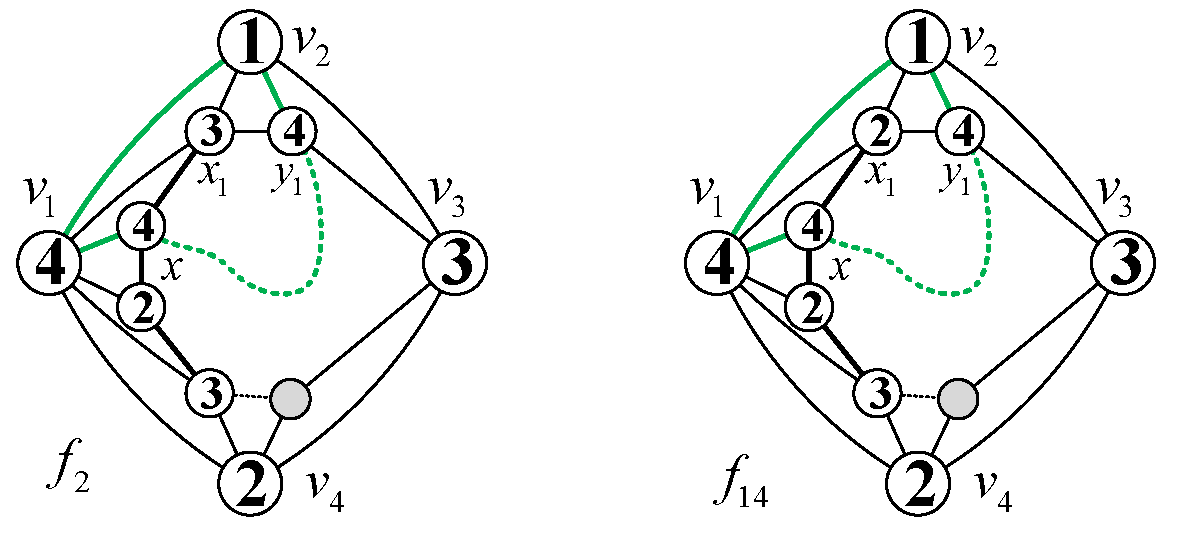}\\
  (a) \hspace{3cm}(b)
  \caption {Illustration for the conversion from 56-3423 type to 55-324 type}\label{fig4-29}
\end{figure}

{\bf Proof of Claim 5.} For the 56-3423 type 4-base-module shown in Figure \ref{fig4-26} (c), its 14-sloped-axis  under $f_2$ is shown in  Figure \ref{fig4-29}(a). Based on $f_2$, conduct a $K$-change for the 23-component in the interior of  $C_{14}$, and denote by $f_{14}$ the resulting pseudo 4-coloring; see Figure \ref{fig4-29}(b). Under $f_{14}$, the pseudo 44-edge $v_1x$ is 22-type. So, we can obtain a decycle coloring $f^*$ for this type by a similar way as that for 55-324 type.

{\bf Claim 6.} \textbf{56-3424 type 4-base-modules are decyclizable}.

{\bf Proof of Claim 6.} For the 56-3424 type 4-base-module shown in Figure \ref{fig4-26} (d), there are two cases for its 14-sloped-axis structure under $f_2$. One case is that there is an 14-sloped-axis from $x'$ to $y_1$, which forms a pseudo 14-cycle $C_{14}$ together with $v_1x',v_1v_2$ and  $y_1v_2$;   see Figure \ref{fig4-30}(a). For this case, conduct a $K$-change for the 23-component in the interior of  $C_{14}$, and obtain a pseudo 4-coloring, denote by $f_{14}$; see Figure \ref{fig4-30}(b). Under $f_{14}$,  no vertex that is adjacent to $v_1$ is colored with 3. So, based on $f_{14}$ we can obtain a decycle coloring $f^*$  by recoloring $v_1$ with 3. The other case is that under $f_2$ the pseudo 44-edge $v_1x'$ is not on an  odd-length 14-cycle, but the pseudo 44-edge $v_1x$ is on an odd-length 14-cycle. So, this pseudo 44-edge $v_1x'$ can be eliminated by
conducting a pseudo 44-edge $K$-change for the 14-component,  and then obtain a new pseudo 4-coloring $f'_2$; see Figure \ref{fig4-30}(c).
Since under $f'_2$ the pseudo 44-edge $v_1x$ is 22-type, we can obtain a decycle coloring $f^*$ for this type by a similar way as that for 55-324 type.

\begin{figure}[H]
  \centering
  \includegraphics[width=8cm]{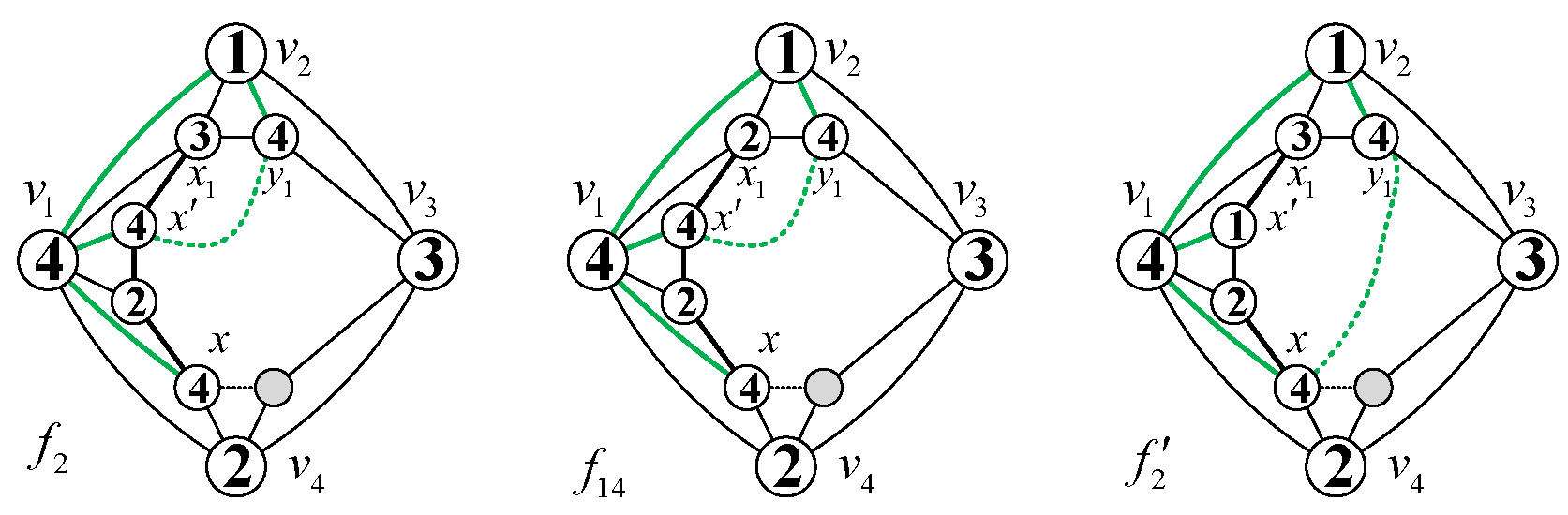}\\
  (a) \hspace{2cm}(b)\hspace{2cm}(c)
  \caption {Illustration for the conversion from 56-3424 type to 55-324 type}\label{fig4-30}
\end{figure}

{\bf Claim 7.} \textbf{56-3434 type 4-base-modules are decyclizable}.

\begin{figure}[H]
  \centering
  \includegraphics[width=10cm]{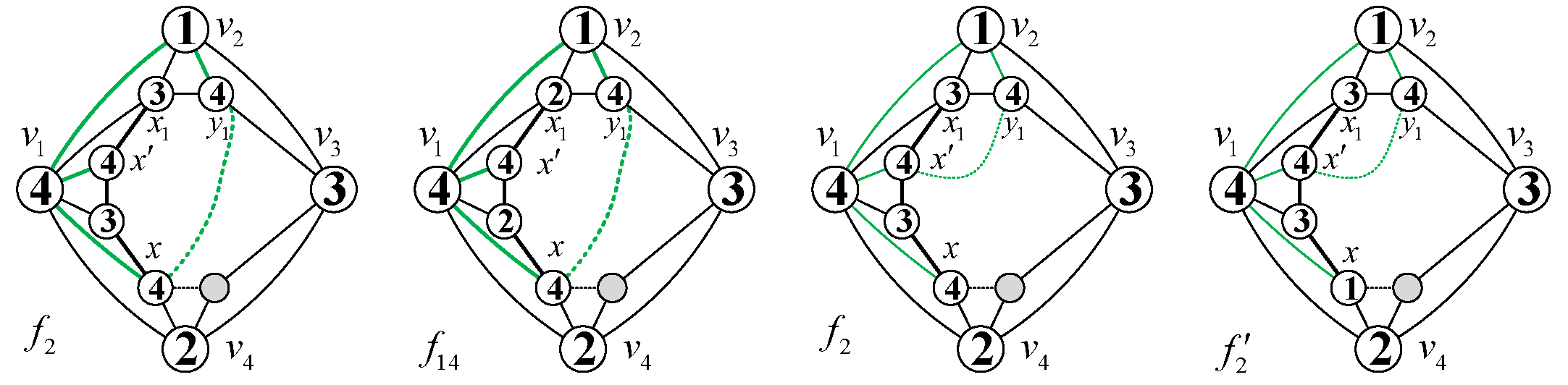}\\
  (a) \hspace{1.8cm}(b)\hspace{1.8cm}(c)\hspace{1.8cm}(d)
  \caption {Illustration for the conversion from 56-3434 type to 55-343 type}\label{fig4-31}
\end{figure}

{\bf Proof of Claim 7.} For the 56-3434 type 4-base-module shown in Figure \ref{fig4-26} (e), there are two cases for its 14-sloped-axis  under $f_2$. One case is that there is an 14-sloped-axis from $x$ to $y_1$, which contains the pseudo 44-edge $v_1x$; see Figure \ref{fig4-31}(a). For this case, conduct a $K$-change for the 23-component in the interior of  the pseudo bichromatic cycle $C_{14}$, and obtain a pseudo 4-coloring $f_{14}$; see Figure \ref{fig4-31}(b). Under $f_{14}$,  no vertex that is adjacent to $v_1$ is colored with 3. So, based on $f_{14}$, we can obtain a decycle coloring $f^*$  by recoloring $v_1$ with 3. The other case is that under $f_2$ there is no 14-sloped-axis from $x$ to $y_1$, but there is an 14-sloped-axis from $x'$ to $y_1$, which implies that the pseudo 44-edge $v_1x$ is not on an odd-length 14-cycle; see Figure \ref{fig4-30}(c). So, the pseudo 44-edge $v_1x$ can be eliminated by conducting a pseudo 44-edge ($v_1x$) $K$-change for the 14-component, and obtain a new pseudo 4-coloring $f'_2$; see Figure \ref{fig4-30}(d).
Since under $f'_2$ the pseudo 44-edge $v_1x$ is 33-type, we can obtain a decycle coloring $f^*$ for this type by a similar way as that for 55-343 type.

\subsection{Cycle-type and cyclic-cycle-type 4-base-modules}

Recall that, in Section \ref{sec4-1}, 4-base-modules are divided into three classes: tree-type, cycle-type, and cyclic-cycle-type. In Section \ref{sec3-6}, we have proved that tree-type 4-base-modules are decyclizable. In this section, we will show that the other two types of 4-base-modules are also decyclizable.

\begin{theorem}\label{thm3-7}
 Let $G^{C_4}$ be a 4-base-module and $f$ be a module-coloring of $G^{C_4}$, where $C_4=v_1v_2v_3v_4v_1$, $f(v_1)=f(v_3)=1$, and  $f(v_2)=f(v_4)=2$.  Suppose that $f$ is a cycle-coloring or cyclic-cycle coloring, under which there are two intersecting module-paths between $v_2$ and $v_4$, 23-module-path $\ell^{23}$ and 24-module-path $\ell^{24}$. If $\ell^{24}$ does not intersect with a bichromatic cycle or a cyclic-cycle, then the following two statements hold.

 (1) The 4-coloring $f_1$, obtained from $f$ by  conducting a $K$-change for the 13-component containing $v_3$, contains  12-cycles or 23-cycles that intersect with the 24-module-path $\ell^{24}$.

 (2) There exists a 4-coloring $f^*\in C_4^0(G^{C_4})$ such that $f^*(v_2)\neq f^*(v_4)$.
\end{theorem}

\begin{proof}
(1) Let $u_1,u_2,\ldots, u_p, p\geq 3$ be the consecutive vertices in $\ell^{24}$ that are colored with 2  under $f$, where $u_1=v_2$, and $u_p=v_4$. Consider the $\ell^{24}$-related graph $H_2^f(\ell^{24})$. Since (under $f$) $\ell^{24}$ intersects with neither a bichromatic cycle nor a cyclic-cycle, it follows that $u_1, u_2,\ldots,u_p$ belong to the same 12- and 23-component of $f$, where 12-component contains the 12-cycle $C_4$. Therefore, the subgraph of $H_2^f(\ell^{24})$ induced by solid lines is connected and contains a cycle, and the subgraph of $H_2^f(\ell^{24})$ induced by dashed lines is a tree, i.e., $H_2^f(\ell^{24})$ has in total $2p-1$ lines. Consider the $\ell^{24}$-related graph $H_2^{f_1}(\ell^{24})$ under $f_1$ which is obtained on the basis of $H_2^f(\ell^{24})$; we see that $f_1$ contains   12-cycles or 23-cycles that intersects with the 24-module-path $\ell^{24}$. This completes the proof.

Based on the proof of statement (1), we can prove statement (2) with a similar proof as that in Theorem \ref{thm4-7}. \qed
\end{proof}

\begin{theorem}\label{thm3-8}
 Let $G^{C_4}$ be a MMP-type 4-base-module  and $f$ be a module-coloring of $G^{C_4}$, where $C_4=v_1v_2v_3v_4v_1$, $f(v_1)=f(v_3)=1$, and  $f(v_2)=f(v_4)=2$.  Suppose that $d_{G^{C_4}}(v_2)=4, d_{G^{C_4}}(v_4)\in \{5,6\}$, and under $f$  there are two intersecting module-paths between $v_2$ and $v_4$, 23-module-path $\ell^{23}$ and 24-module-path $\ell^{24}$. Then, there exists a 4-coloring $f^*\in C_4^0(G^{C_4})$ such that $f^*(v_2)\neq f^*(v_4)$.
\end{theorem}

\begin{proof}
Without loss of generality, we assume that $\ell^{24}$ intersects with a bichromatic cycle or a cyclic-cycle. First, under $f$, conduct a $K$-change for the 13-component containing $v_3$ and denote by $f_1$ the resulting 4-coloring.  Consider $f_1$. If by $f_1$ we can deduce that $G^{C_4}$ is a Kempe 4-base-module, then the result holds; otherwise, with an analogous proof as that in Theorem  \ref{thm4-7}, we can obtain a 4-coloring $f_2$ based on $f_1$. Now, if $f_2$ contains a bichromatic cycle whose pseudo edge can be eliminated, then the result holds; otherwise, we can obtain a 4-coloring $f_{14}$ based on $f_2$, and then obtain a decycle coloring  with a similar proof as that in Theorem \ref{thm4-7}.
\qed
\end{proof}

\begin{figure}[H]
  \centering
  \includegraphics[width=8cm]{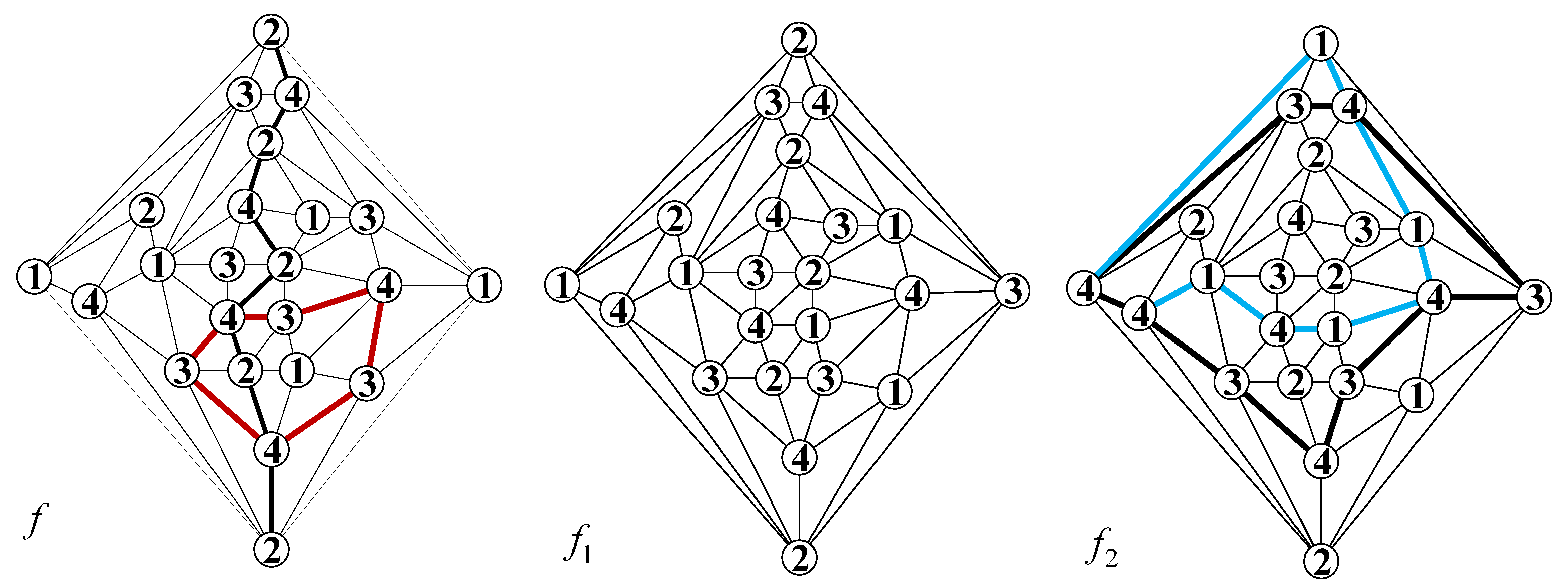}\\
  \caption {An example of the decycle  process  for MMP 4-base-modules based on a module-coloring $f$}\label{fignew51}
\end{figure}

Figure \ref{fignew51} gives an example to illustrate the decycle process for MMP 4-base-modules based on a module-coloring $f$, with a similar way as that for tree-type 4-base-modules, in which we can see how to obtain $f_2$ from $f$. Since the decycle process under $f_2$ is analogous to that for tree-type 4-base-modules, we omit the detailed steps.

\section{55- and 56-configurations are reducible}

\begin{theorem}\label{main}
55- and 56-configurations in a maximal planar graph  with $\delta=5$ are reducible.
\end{theorem}

\begin{proof}
By induction on the number of vertices. Let $G$ be an MPG containing $n$ vertices. This is clearly true for $n=12$. Suppose that the conclusion is true whenever $n\leq k$ and consider the case $n=k+1$, where $k\geq 12$. In the following, we use $\{1,2,3,4\}$ to denote the color set.

Let $v_2$ be an arbitrary 5-vertex of $G$. Then, $G[N_G[v_2]]$ is 5-wheel $W_5$. We denote by  $C_5=uv_3v_4v_1v_5u$ the cycle of $W_5$, i.e., $W_5$=$v_2$-$uv_3v_4v_1v_5u$; see Figure \ref{figure5-1} (a), in which we use the 5-wheel $W_5$ to simply represent the graph $G$.  By the induction hypothesis, $G-v_2$ is 4-colorable. If there exists a 4-coloring $g \in C_4^0(G-v_2)$  such that $|g(C_5)|=3$, then $g$ can be extended to a 4-coloring of $G$ by coloring $v_2$ with the color $\{1,2,3,4\}\setminus g(C_5)$. So, we may assume that  $C_5$ is colored with four distinct colors under every 4-coloring $f$ of $G-v_2$.  Without loss of generality, we assume that $f(u)=1, f(v_3)=f(v_5)=2$,  $f(v_4)=3$, and $f(v_1)=4$; see Figure \ref{figure5-1} (a).

If $u$ and $v_1$  are not in the same 14-component of $f$, then we carry out  a $K$-change for the 14-component of $f$ containing vertex $u$ (or $v_1$) and obtain a new 4-coloring of $G-v_2$, say $f'$. Clearly, $|f'(C_5)|=3$, a contradiction to the above assumption.
So, we may assume that  $f$ contains an 14-path of $f$ from $u$ to $v_1$. By the same reason, there is also an 13-path of $f$ from $u$ to $v_4$; see Figure \ref{figure5-1} (b).
 Clearly, $v_5$ and $v_3$ do not belong to the same $23$-component of $f$. Then, based on $f$, we conduct
a $K$-change for the 23-component of $f$ containing vertex $v_5$ and obtain a new 4-coloring of $G-v_2$, denoted by $f_1$. We have that $f_1(v_5)=f_1(v_4)=3, f_1(v_1)=4, f_1(v_3)=2$, and $f_1(u)=1$; see Figure \ref{figure5-1} (c). By a reverse derivation, we can also obtain $f$ from $f_1$. We use $\sim$ to denote the relation between $f$ and $f_1$, i.e., $f\sim f_1$. With an analogous discussion as $f$, we see that under $f_1$ there exists a 24-path $P_{24}^{v_3v_1}$ from $v_3$ to $v_1$ and so $v_4$ and $u$ do not belong to the same 13-component of $f_1$; see Figure \ref{figure5-1} (d). We then conduct a $K$-change for the 13-component of $f_1$ in the interior of $P_{24}^{v_3v_1}\cup v_1v_2v_3$ and obtain a new 4-coloring of $G-v_2$, denoted by $f_2$; see Figure \ref{figure5-1} (e). Under $f_2$, $f_2$ contains a 23-path  $P_{23}^{v_3v_5}$  from $v_3$ to $v_5$, and $u$ and $v_1$ do not belong to the same 14-component of $f_2$; see Figure \ref{figure5-1} (f). By conducting a $K$-change for the 14-component of $f_2$ in the interior of $P_{23}^{v_3v_5}\cup v_5v_2v_3$, we obtain a new 4-coloring $f_3$ of $G-v_2$; see Figure \ref{figure5-1} (g).
Finally, we can also see that $f_3$ contains  a 13-path $P_{13}^{v_4v_5}$ from $v_4$ to $v_5$, and hence $v_1$ and $v_3$ do not belong to the same 24-component of $f_3$; see Figure \ref{figure5-1} (h). We then conduct a $K$-change on the 24-component of $f_3$ in the interior of $P_{13}^{v_4v_5}\cup v_5v_2v_4$, and obtain a new 4-coloring $f_4$ of $G-v_2$; see Figure \ref{figure5-1} (i).

 \begin{figure}[H]
  \centering
  \includegraphics[width=11.5cm]{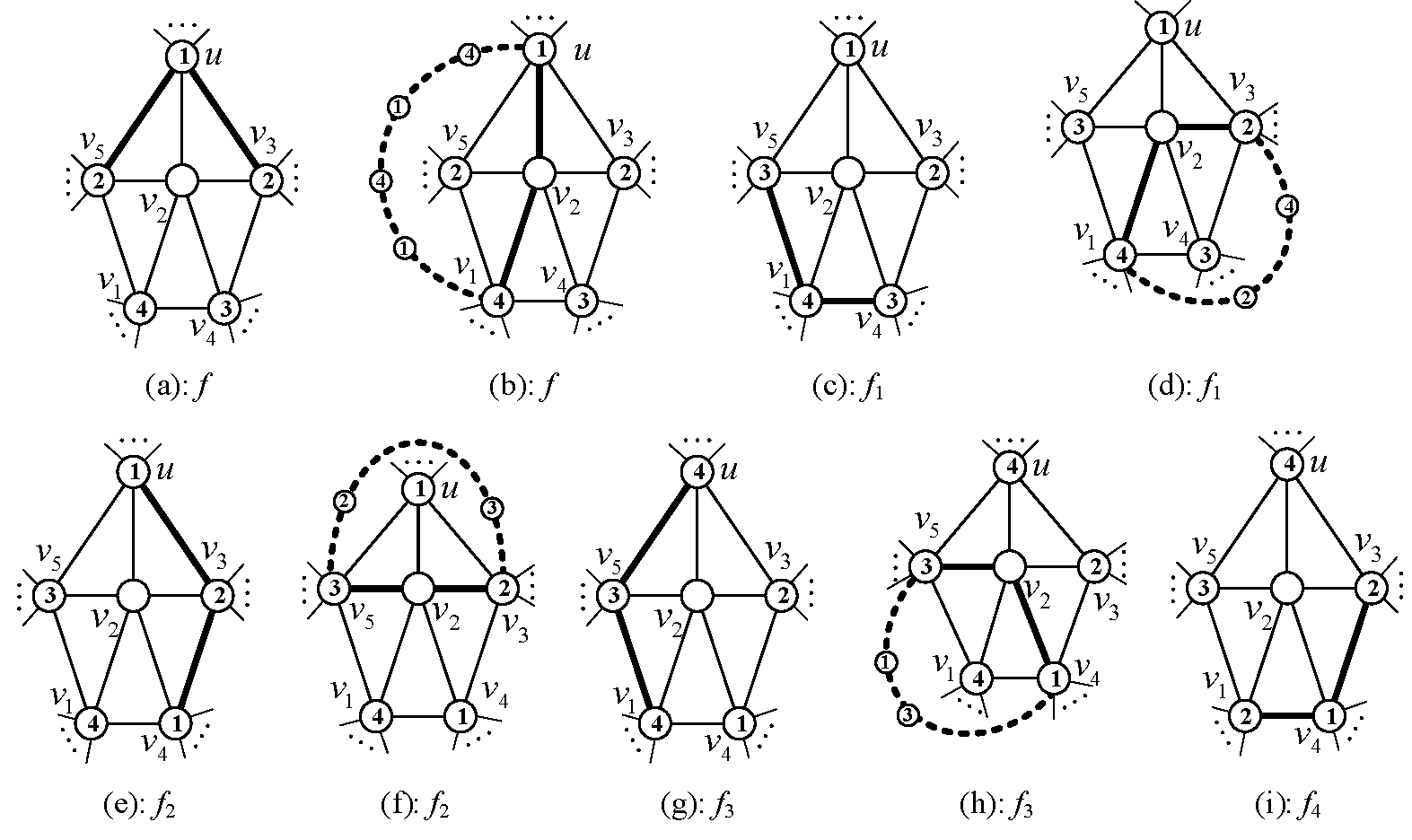}\\
  \caption {The equivalence of the five 4-colorings of $G$ in the condition that $C_5$ are colored with four colors}\label{figure5-1}
\end{figure}

Let a 2-path denote a path of length 2. Observe that every 4-coloring in $\{f,f_1,f_2,f_3,f_4\}$ has a bichromatic 2-path  whose edges are on $C_5$, and any two distinct 4-colorings in $\{f,f_1,f_2,f_3,f_4\}$ have distinct such bichromatic 2-paths. So, under all 4-colorings of $C_5$, there are in total five distinct bichromatic 2-paths on $C_5$. This implies that when we only consider the 4-coloring restricted to $C_5$, $\sim$ is an equivalence relation, and $f\sim f_1 \sim f_2 \sim f_3 \sim f_4 \sim f$.

In 1904, Wernicke \cite{r24} proved that  every MPG of minimum degree 5 contains a 55-configuration (consisting of two adjacent vertices of degree 5) or a 56-configuration  (consisting of a vertex of degree 5 adjacent to a vertex of degree 6); see Figure \ref{fig5-2} (a) and (b). 
So, $G$ contains a 55-configuration $G_{5,5}$ or a 56-configuration $G_{5,6}$. Without loss of generality, we always assume that the 5-wheel contained in $G_{5,5}$ and $G_{5,6}$ is $W_5$, the 5-wheel shown in Figure \ref{figure5-1} (a).  By the equivalency of 4-colorings in $\{f,f_1,f_2,f_3,f_4\}$, we further assume that  under $f$ the bichromatic 2-path of $C_5$ is $v_3v_4v_1$, where $f(v_3)=f(v_1)=1, f(v_4)=2, f(v_5)=3$, and $f(u)=4$; see  Figures \ref{fig5-2} (a) and (b) ({Note that the gray vertices in Figure  \ref{fig5-2}  are colored with some unfixed colors under $f$}).

Now, we will prove that $f$ can be extended to a 4-coloring of $G$. We first consider the case that $G$ contains $G_{5,5}$. 

Suppose that $G$ contains $G_{5,5}$. We conduct an E4WO on the bichromatic 2-path $v_3v_4v_1$, and the resulting graph is shown in Figure \ref{fig5-2} (c), denoted by $G^*$, where $v'_4$ and $v_4$ are the two vertices replacing $v_4$ and $x$ is  the  center of the new 4-wheel $x-v_3v_4v_1v'_4v_3$. Clearly, $G^*$ is an MPG of order $k+3$. Based on $f$, there is a 4-coloring $f'$ of  $G^*-v_2$, called \emph{the natural coloring of $f$}, such that $f'(v)=f(v)$ for every $v\in V(G^*)\setminus \{v'_4,v_4,x,v_2\}$, $f'(v_4)=f'(v'_4)=2$, and  $f'(x)=3$ or $4$, where $v_2$ is the unique vertex not assigned to a color under $f'$ (see Figure \ref{fig5-2} (c)). Furthermore, we can obtain a 4-coloring $f''$ of $G^*$ based on $f'$ by recoloring $v'_4$ with 3,  coloring $v_2$ with 2, and  coloring $x$ with 4; see Figure \ref{fig5-2} (d). Observe that $f''(v'_4)\neq f''(v_4)$. In the following, we will prove that there exists a 4-coloring $f'''$ of $G^*$ derived from $f''$ such that $f'''(v'_4)=f'''(v_4)$. Thus, under $f'''$, when conducting  a C4WO on the 4-wheel $x-v_3v_4v_1v'_4v_3$, we can obtain a 4-coloring of $G$.

\begin{figure}[H]
  \centering
  \includegraphics[width=12cm]{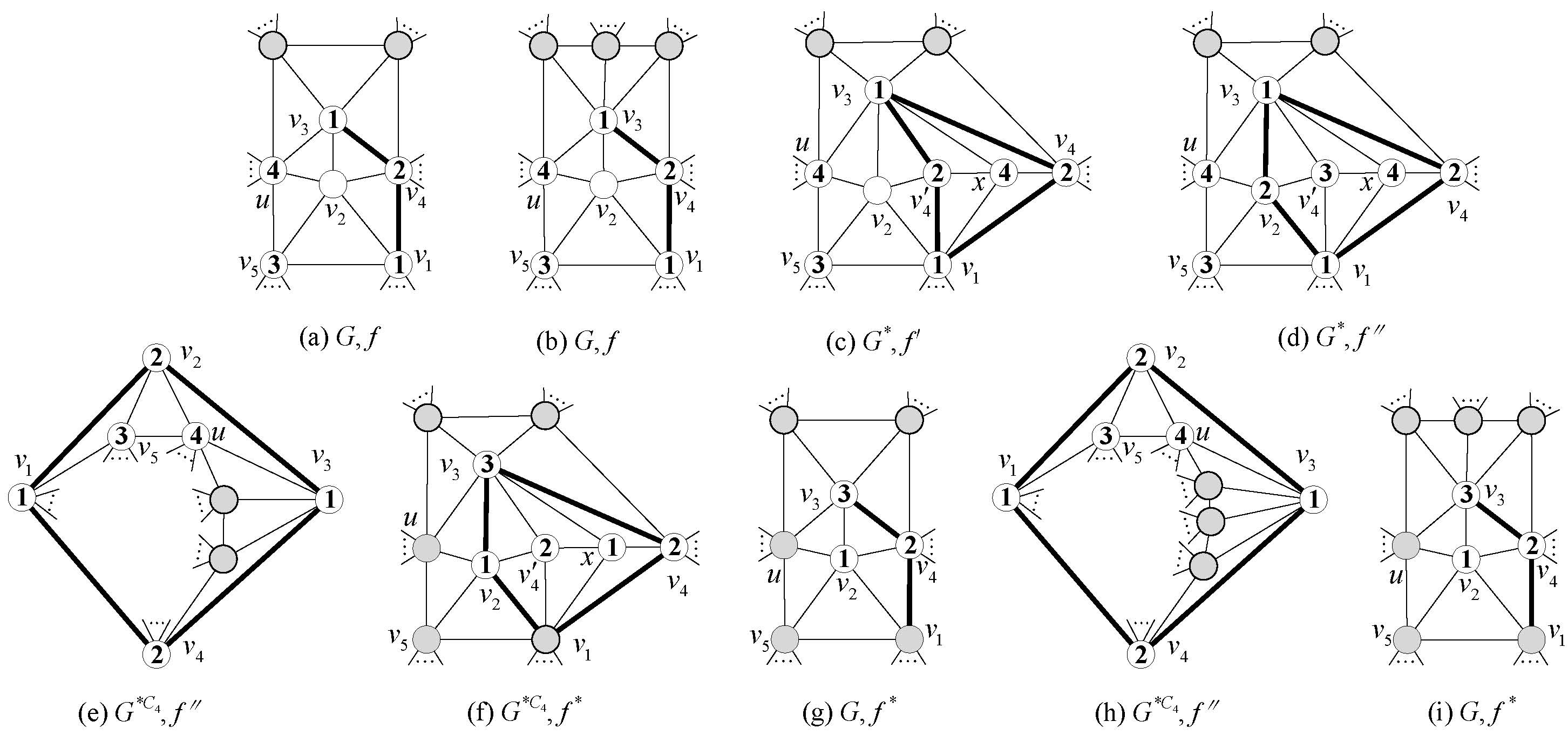}\\
  \caption {An illustration diagram of the proof of Theorem \ref{main}}\label{fig5-2}
\end{figure}

Observe that the  SMPG that is composed by the 4-cycle $C_4=v_1v_2v_3v_4v_1$ and its interior is the smallest 4-base-module $B^4$, and $G^{*C_4}=G^*-\{x,v'_4\}$ is a 4-chromatic SMPG.
By topological transformation, $G^{*C_4}$ has a plane embedding such that $C_4$ is the outer face (unbounded face); see Figure \ref{fig5-2} (e). It is clear that $d_{G^{*C_4}}(v_2)=4$ and $d_{G^{*C_4}}(v_3)=5$.

 Now, if $G^{*C_4}$ is not a 4-base-module, then by Theorem \ref{thm3.4}, there exists a $f^*\in C_4^0(G^{*C_4})$ such that $f^*(v_2)\neq f^*(v_4)$. If $G^{*C_4}$ is  a 4-base-module, then  by Theorems \ref{thm4-3}, \ref{thm4-7},\ref{thm3-7},\ref{thm3-8}, there also exists a $f^*\in C_4^0(G^{*C_4})$ such that $f^*(v_2)\neq f^*(v_4)$. So, without loss of generality, we may assume that $f^*(v_2)=1, f^*(v_4)=2$, and  $f^*(v_3)=3$, whether $G^{*C_4}$ is a 4-base-module or not. Then, $f^*$ can be extended to the desired 4-coloring $f'''$ of $G^*$ by letting $f'''(v)=f^*(v)$ for every $v\in V(G^*)\setminus \{v'_4,x\}$, $f'''(v'_4)=2$ (this can be done since $f^*(v_1)\neq 2$), and $f'''(x)=1$. Thus, the restricted coloring of  $f'''$ to the resulting graph (i.e., $G$) by  conducting a C4WO on the 4-wheel $x-v_3v_4v_1v'_4v_3$ is a 4-coloring of $G$; see Figure \ref{fig5-2} (g).

In the above, we prove that $55$-configuration is reducible. The proof for $G_{5,6}$ is similar to that for $G_{5,5}$. The difference between them is that $d_{G^{*{C_4}}}(v_3)=6$ (see Figure \ref{fig5-2} (h)), where $G^{*{C_4}}=G^{*}-\{v'_4, x\}$ and $G^*$ is the resulting graph obtained from $G$ by conducing an E4WO on the bichromatic 2-path $v_3v_4v_1$.

If $G^{*{C_4}}$ is a 4-base-module, then by Theorems \ref{thm4-3}, \ref{thm4-7},\ref{thm3-7},\ref{thm3-8}, there exists a 4-coloring $f^{*}\in C_{4}^{0}(G^{*{C_4}})$ such that $f^{*}(v_2)\neq f^{*}(v_4)$. 
Furthermore, $f^{*}$ can be extended to a 4-coloring  of $G^{*}$  for which $v'_4$ and $v_4$ are colored with the same color. Again, by conducting a C4WO on the new 4-wheel, we  can obtain a 4-coloring of $G$;  see Figure \ref{fig5-2} (i). This completes the proof. \qed

\end{proof}

\begin{figure}[H]
  \centering
  \includegraphics[width=12.3cm]{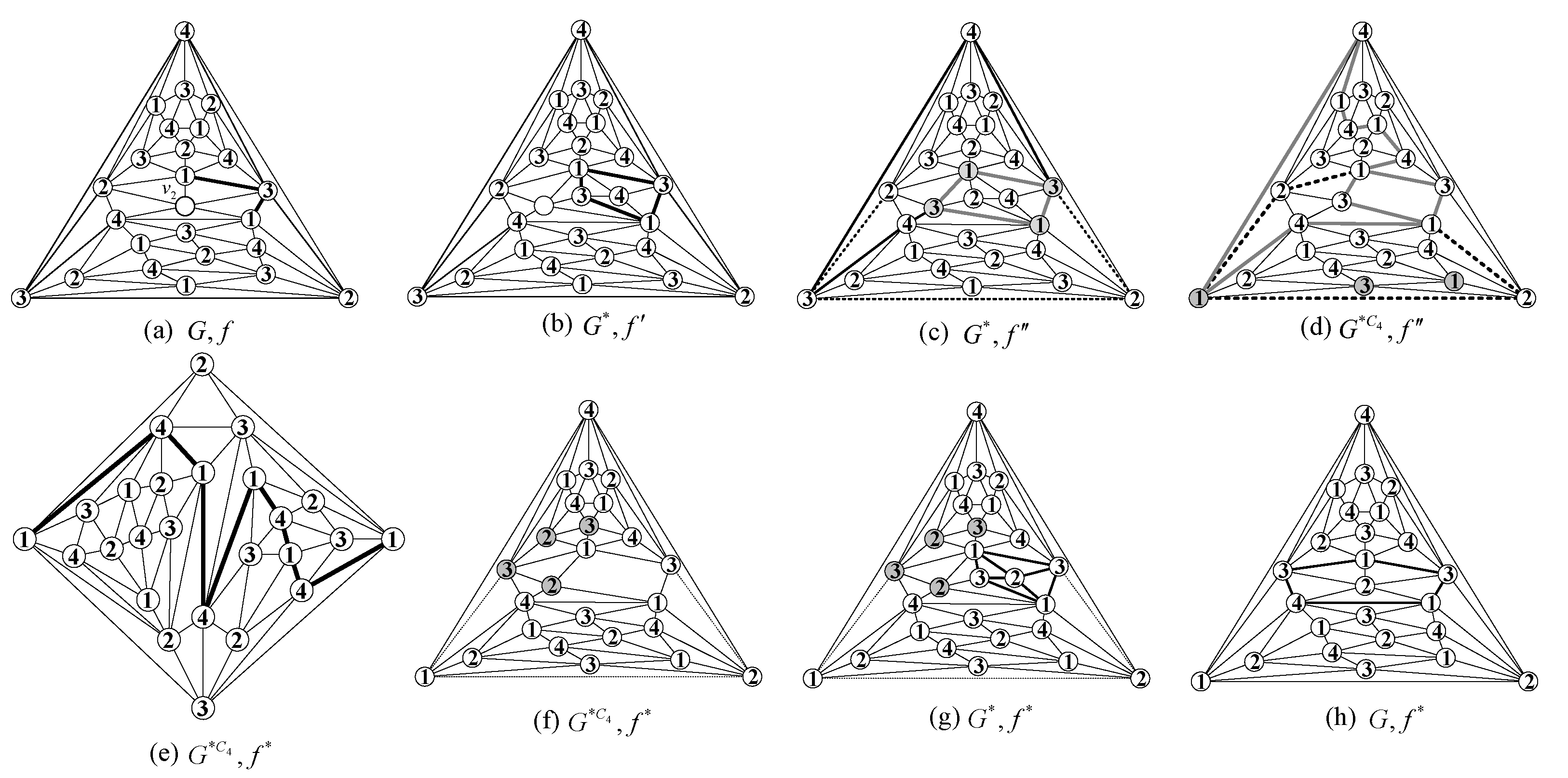}\\
  \caption {The process of coloring the Heawood counterexample by Theorem \ref{main}}\label{figure5-3}
\end{figure}

Here, we take the well-known Heawood counterexample as an example to show the process of our proof in Theorem \ref{main}; see Figure \ref{figure5-3}, where $v_2$ is the unique vertex of degree 5 whose color is unfixed. The two vertices with color 1 in the neighborhood of $v_2$ have degree 6; see Figure \ref{figure5-3}(a). The vertices of  the bichromatic 2-path adjacent to $v_2$ (marked with bold solid lines) are colored with 1,3,1, respectively. By conducting an E4WO on this bichromatic 2-path we obtain a graph $G^*$ with a natural coloring $f'$; see Figure \ref{figure5-3}(b). Based on $f'$, we color $v_2$ with 3 and change the color of the vertex of degree 4 adjacent to $v_2$ from 3 to 2, and obtain a 4-coloring of $G^*$, denoted by $f''$ (see Figure \ref{figure5-3} (c)). Now, delete the two vertices in the interior of the $B^4$ subgraph of $G^*$ and obtain a 4-chromatic SMPG $G^{*C_4}$. We also denote by $f''$ the restricted coloring of $f''$ to $G^{*C_4}$. Observe that there is a 24-cycle of $f''$;  we conduct (in $G^{*C_4}$) a $K$-change for the   13-component of $f''$ in the interior of the 24-cycle. The resulting coloring is denoted by $f'''$ (see  Figure \ref{figure5-3} (d)). Since under $f'''$ there is an 14-endpoint-path and an 12-endpoint-path,  and under $f''$ there is a 34-endpoint-path and a 23-endpoint-path (see  Figure \ref{figure5-3} (c)), it follows that the obtained SMPG is not a 4-base-module.

So, under $f'''$, by conducting a $K$-change for the 23-component containing one of the endpoints, we can obtain a 4-coloring of $G^{*C_4}$; see Figure \ref{figure5-3} (e)$\sim$ (f). Now, add two vertices in the interior of $C_4$ and the resulting graph is $G^*$. We recolor the two vertices in the interior of $B^*$ and obtain a coloring, also denoted by $f^*$; see Figure \ref{figure5-3} (g). Under $f^*$, conduct a C4WO on the 4-wheel (marked by bold solid lines) and obtain the MPG $G$. The restricted coloring of $f^*$ to $G$ is a 4-coloring of $G$; see Figure \ref{figure5-3} (h).


\begin{thebibliography}{10}
\providecommand{\url}[1]{{#1}}
\providecommand{\urlprefix}{URL }
\expandafter\ifx\csname urlstyle\endcsname\relax
  \providecommand{\doi}[1]{DOI~\discretionary{}{}{}#1}\else
  \providecommand{\doi}{DOI~\discretionary{}{}{}\begingroup
  \urlstyle{rm}\Url}\fi

\bibitem{r7}
Appel, K., Haken, W.: Every planar map is four colorable. part i: Discharging.
\newblock Illinois Journal of Mathematics \textbf{21}(3), 429--490 (1977)

\bibitem{r8}
Appel, K., Haken, W., Koch, J.: Every planar map is four colorable. part i:
  Reducibility.
\newblock Illinois Journal of Mathematics \textbf{21}(3), 491--567 (1977)

\bibitem{r15}
Barnette, D.: On generating planar graphs.
\newblock Discrete Mathematics \textbf{7}(3-4), 199--208 (1974)

\bibitem{r17}
Batagelj, V.: Inductive definition of two restricted classes of triangulations.
\newblock Discrete Mathematics \textbf{52}(2-3), 113--121 (1984)

\bibitem{r11}
Beineke, L.W., Wilson, R.J.: Selected topics in graph theory.
\newblock NCambridge: Academic Press (1978)

\bibitem{r2}
Birkhoff, G.D.: A determinant formula for the number of ways of coloring a map.
\newblock The Annals of Mathematics \textbf{14}(2), 42--46 (1912)

\bibitem{r22}
Bondy, J.A., Murty, U.S.R.: Graph Theory with applications.
\newblock Springer (2008)

\bibitem{r16}
Butler, J.W.: A generation procedure for the simple 3-polytopes with cyclically
  5-connected graphs.
\newblock Canadian Journal of Mathematics \textbf{7}(3-4), 199--208 (1974)

\bibitem{r13}
Eberhard, V.: Zur morphologie der polyeder.
\newblock BG Teubner (1891)

\bibitem{r4}
Fisk, S.: Geometric coloring theory.
\newblock Advances in Mathematics \textbf{24}(3), 298--340 (1977)

\bibitem{r3}
Greenwell, D., Kronk, H.V.: Uniquely line-colorable graphs.
\newblock Canadian Mathematical Bulletin \textbf{16}(4), 525--529 (1973)

\bibitem{r6}
Heawood, P.J.: Map-colour theorems.
\newblock Quarterly Journal of Mathematics \textbf{24}, 332--338 (1890)

\bibitem{r1}
Kempe, A.B.: On the geographical problem of the four colors.
\newblock American Journal of Mathematics \textbf{2}(3), 193--200 (1879)

\bibitem{r23}
Liu, X.Q., Xu, J.: A special type of domino extending-contracting operations.
\newblock Journal of Electronics \& Information Technology \textbf{39}(1),
  221--230 (2017)

\bibitem{r9}
Robertson, N., Daniel, S., Seymour, P., Thomas, R.: A new proof of the
  four-colour theorem.
\newblock Electronic Research Announcements of the American Mathematical
  Society \textbf{2}, 17--25 (1996)

\bibitem{r5}
Soifer, A.: The mathematical coloring book : Mathematics of coloring and the
  colorful life of its creators.
\newblock Heidelberg: Springer (2009)

\bibitem{r10}
Tommy, R.J., Bjarne, T.: Graph Coloring Problems.
\newblock New York, US, John Wiley \& Sons, Inc. (1994)

\bibitem{r14}
Wagner, K.: Bemerkungen zum vierfarbenproblem.
\newblock Jahresbericht der Deutschen Mathematiker-Vereinigung \textbf{46},
  26--32 (1936)

\bibitem{r24}
Wernicke, P.: {\"{U}}ber den kartographischen vierfarbensatz.
\newblock Math. Ann. \textbf{58}, 413--426 (1904)

\bibitem{r12}
Xu, J.: Science Press

\bibitem{r18}
Xu, J.: Theory on structure and coloring of maximal planar graphs (2): Domino
  configurations and extending- contracting operations.
\newblock Journal of Electronics \& Information Technology \textbf{38}(6),
  1271--1327 (2016)

\bibitem{r20}
Xu, J.: Theory on structure and coloring of maximal planar graphs (3): Purely
  tree-colorable and uniquely 4-colorable maximal planar graph conjectures.
\newblock Journal of Electronics \& Information Technology \textbf{38}(6),
  1328--1353 (2016)

\bibitem{r21}
Xu, J., Li, Z.P., Zhu, E.Q.: On purely tree-colorable planar graphs.
\newblock Information Processing Letters \textbf{116}(8), 532--536 (2016)

\bibitem{r19}
Xu, J., Liu, X.Q.: Research progress on the theory of kempe changes.
\newblock Journal of Electronics \& Information Technology \textbf{39}(6),
  1493--1502 (2017)

\end{thebibliography}
\end{document}